\documentclass[12pt]{article}

\usepackage{amsmath,amssymb}
\usepackage{amsthm}
\usepackage{cite}
\usepackage{enumerate}
\usepackage[mathlines]{lineno}
\linenumbers
\usepackage[colorlinks=true, allcolors=blue]{hyperref}
\usepackage{pdfpages} 

\usepackage{tikz}
\tikzset{
every node/.style={draw, circle, inner sep=2pt}
}

\usepackage[portrait, total={6.5in, 8.8in}]{geometry}
\usepackage{soul}
\usepackage{cancel}

\newtheorem{theorem}{Theorem}[section]
\newtheorem{lemma}[theorem]{Lemma}
\newtheorem{proposition}[theorem]{Proposition}
\newtheorem{corollary}[theorem]{Corollary}

\theoremstyle{definition}
\newtheorem{definition}[theorem]{Definition}
\newtheorem{observation}[theorem]{Observation}
\newtheorem{remark}[theorem]{Remark}
\newtheorem{example}[theorem]{Example}

\newtheorem{question}[theorem]{Question}
\newtheorem{conjecture}[theorem]{Conjecture}
\newtheorem*{theorem1.5}{Theorem 1.5}
\newtheorem*{theorem1.7}{Theorem 1.7}

\newcommand{\bone}{{\bf 1}}
\newcommand{\bzero}{{\bf 0}}
\newcommand{\trans}{^\top}

\newcommand{\bx}{{\bf x}}
\newcommand{\bv}{{\bf v}}
\newcommand{\bc}{{\bf c}}
\newcommand{\bu}{{\bf u}}
\newcommand{\be}{{\bf e}}

\DeclareMathOperator{\sign}{sgn}
\newcommand{\Scl}{\mathcal{S}^{\rm cl}}

\renewcommand{\vec}[1]{{\bf #1}}
\renewcommand{\bar}{\overline}

\definecolor{Green}{RGB}{34, 139, 34}

\renewcommand{\c}[1]{\mathcal{#1}}
\newcommand{\sm}{\setminus}
\newcommand{\mker}{\operatorname{ker}}
\renewcommand{\S}{\mathcal S}

\title{Complementary Vanishing Graphs}
\author{
Craig Erickson\thanks{Saint Paul,
MN, USA (cerickson.phd@gmail.com).} \and
Luyining Gan\thanks{Department of Mathematics and Statistics, University of Nevada Reno, Reno, NV 89557, USA (lgan@unr.edu).} \and
J{\"u}rgen Kritschgau\thanks{Department of Mathematics, Carnegie Mellon University, Pittsburgh, PA 15213, USA (jkritsch@andrew.cmu.edu).} \and
Jephian C.-H.~Lin\thanks{Department of Applied Mathematics, National Sun Yat-sen University, Kaohsiung 80424, Taiwan (jephianlin@gmail.com)} \and
Sam Spiro\thanks{Department of Mathematics, UC San Diego, La Jolla, CA 92093, USA (sspiro@ucsd.edu).}
}
\date{\today}

\begin{document}

\nolinenumbers
\maketitle

\begin{abstract}
Given a graph $G$ with vertices $\{v_1,\ldots,v_n\}$, we define $\mathcal{S}(G)$ to be the set of symmetric matrices $A=[a_{i,j}]$ such that for $i\ne j$ we have $a_{i,j}\ne 0$ if and only if $v_iv_j\in E(G)$.  Motivated by the Graph Complement Conjecture, we say that a graph $G$ is complementary vanishing if there exist matrices $A \in \mathcal{S}(G)$ and $B \in \mathcal{S}(\overline{G})$ such that $AB=O$. We provide combinatorial conditions for when a graph is or is not complementary vanishing, and we characterize which graphs are complementary vanishing in terms of certain minimal complementary vanishing graphs. In addition to this, we determine which graphs on at most $8$ vertices are complementary vanishing.
\end{abstract}  


\medskip


\section{Introduction}

Let $G$ be a simple graph with vertex set $\{v_1, v_2, \dots, v_n\}$. The \emph{set of symmetric matrices described by $G$}, 
$\mathcal{S}(G)$, is the set of all real symmetric $n\times n$ matrices $A = \begin{bmatrix}a_{i,j}\end{bmatrix}$ 
 such that for $i\ne j$ we have $a_{i,j}\ne 0$ if and only if $v_iv_j\in E(G)$.  Note that no restrictions are placed on the diagonal entries of $A\in \mathcal{S}(G)$. Given a graph $G$, the \emph{inverse eigenvalue problem of a graph} (IEP-G for short) refers to the problem of determining the spectra of matrices in $\S(G)$. 
A large amount of work has been done in this area, see for example~\cite{GK60, FH07, BF04, AAC13,BFH17, BBF20}.

Specifying exactly which graphs can achieve a particular spectrum is generally hard, and because of this, much of the work in the IEP-G considers parameters that measure more general spectral properties.
For example, the \emph{minimum rank}, $\text{mr}(G)$, of a graph $G$ is defined as the minimum rank of a matrix $A\in \S(G)$. 
The \emph{maximum nullity}, $M(G)$, of a graph $G$ is defined as the maximum nullity of a matrix $A \in \S(G)$.
Since $\S(G)$ is closed under translations by $rI_n$ for real values $r$, the parameter $M(G)$ is also equal to the largest multiplicity of an eigenvalue of $A\in \S(G)$. 
It follows from the definitions that $\text{mr}(G)+M(G)=n$ whenever $G$ is an $n$-vertex graph. 
The minimum rank and maximum nullity of graphs was one of the main subjects of the 2006 American Institute of Mathematics workshop~\cite{AIM06}.
This workshop was a catalyst for research on the IEP-G, and there the following conjecture was posed.

\begin{conjecture}[Graph Complement Conjecture]\label{Conj:GCC}
For any graph $G$ of order $n$,
\[
    \text{mr} (G) + \text{mr} (\overline{G}) \le n +2,
\]
where $\overline{G}$ denotes the complement of $G$.  Equivalently,
\[
M(G) + M(\overline{G}) \geq n - 2.    
\]
\end{conjecture}

 Conjecture~\ref{Conj:GCC} is an example of a Nordhaus--Gaddum problem, which are problems where one tries to prove bounds for $f(G)+f(\overline{G})$ where  $f$ is some graph-theoretic function.
An analogous Nordhaus--Gaddum conjecture for the Colin de Verdi\`{e}re number $\mu(G)$ (defined in~\cite{Y91}) was made by Kotlov et al.~\cite{KLV97}. 
Levene, Oblak and \v{S}migoc verified a Nordhaus--Gaddum conjecture for the minimum number of distinct eigenvalues for several graph families~\cite{LOS19}.

Conjecture~\ref{Conj:GCC} is open in general and is evidently  difficult. 
Some results related to Conjecture~\ref{Conj:GCC} for joins of graphs were studied and established by Barioli et al.\ \cite{BBF12}, and the authors also showed that Conjecture~\ref{Conj:GCC} is true for graphs on up to $10$ vertices.
Li, Nathanson, and Phillips~\cite{LNP14} showed that it suffices to prove Conjecture~\ref{Conj:GCC} for a certain class of graphs called complement critical graphs.  Theorem 3.16 in \cite{AIM06} implies that  Conjecture~\ref{Conj:GCC} is true for trees.  On the other hand, it was shown in \cite{AIM08} that the zero forcing number $Z(G)$ is an upper bound of $M(G)$, and it is known $Z(G) + Z(\overline{G}) \geq n-2$ for any graph on $n$ vertices \cite{EKSTRAND20131862,JORET2012488}.



With Conjecture~\ref{Conj:GCC} in mind, we make the following definition.
\begin{definition}
A graph $G$ is said to be \emph{complementary vanishing} if there are matrices $A\in\S(G)$ and $B\in\S(\overline{G})$ such that $AB = O$ (the zero matrix).
\end{definition}
The motivation for this definition is the following observation.

\begin{proposition}
If $G$ is an $n$-vertex graph which is complementary vanishing, then
\[M(G)+M(\overline{G})\ge n.\]
In particular, Conjecture~\ref{Conj:GCC} holds for $G$.
\end{proposition}
\begin{proof}
By assumption there exist $A\in \S(G),B\in \S(\overline{G})$ such that $AB=O$.  Thus the sum of the nullities of $A$ and $B$ is at least $n$. 
This in turn implies that $M(G)+M(\overline{G})\ge n$ as desired.
\end{proof}

The goal of this paper is to establish necessary and sufficient conditions for a graph to be complementary vanishing. 
Our main result characterizes which graphs are complementary vanishing  in terms of a  class of  ``minimal'' complementary vanishing  graphs $\mathcal {M}$. 
\begin{definition}
~
\begin{itemize}
    \item Let $\mathcal M$ denote the set of complementary vanishing graphs $G$  such that both $G$ and $\overline{G}$ are connected.
    \item Let $\mathcal{R}$ denote the smallest set of graphs containing $\mathcal{M}$ which is closed under taking disjoint unions, joins, and complements\footnote{We will use $G\sqcup H$ to denote the disjoint union of two graphs $G,H$. We recall that the join $G\vee H$ of two graphs $G,H$ is obtained by taking the disjoint union of $G$ and $H$ and then adding all possible edges between vertices of $G$ to vertices of $H$.}.
    \item Let $\mathcal C$ denote all the graphs $G$ such that either in $G$ or $\overline{G}$ there exist distinct vertices $u,v,w$ with $v,w\notin N(u)$, $|N(u)\setminus N(w)|=1$, and $|N(u)\setminus N(v)|=0$. 
\end{itemize}

\end{definition}

\begin{theorem}\label{thm:biconnected}
    A graph $G$ is complementary vanishing if and only if $G\in \mathcal {R}\setminus \mathcal C$.
\end{theorem}

Roughly speaking, Theorem~\ref{thm:biconnected} states that a graph $G$ is complementary vanishing  if and only if $G$ can be constructed out of minimal complementary vanishing  graphs $\mathcal{M}$ while avoiding a particular combinatorial structure.  

\begin{example}
Let $G$ be the graph obtained by adding a perfect matching to $K_{4,4}$.  We will use Theorem~\ref{thm:biconnected} to determine whether $G$ is complementary vanishing  or not. 
It is easy to check that $G\notin \mathcal C$, so $G$ will be complementary vanishing  if and only if $G\in \c{R}$. Observe that
\[G=(K_2\sqcup K_2)\vee (K_2\sqcup K_2)=((K_1\vee K_1)\sqcup (K_1\vee K_1))\vee ((K_1\vee K_1)\sqcup (K_1\vee K_1)).\]
Thus, to have $G\in \c{R}$, it would suffice to have $K_1\in \c{M}$, and it is easy to verify that this is the case.  This proves that $G$ is complementary vanishing.

\end{example}
One can also use Theorem~\ref{thm:biconnected} to determine which trees are complementary vanishing.
\begin{corollary}
A tree is complementary vanishing if and only if it is a star.
\end{corollary}
\begin{proof}
If $T$ is a tree which is not a star, then there exists an induced path $uxvw$ in $T$ where $u$ is a leaf.  Note that $v,w\notin N(u),\ |N(u)\sm N(w)|=1$, and $|N(u)\sm N(v)|=0$.  We conclude that $T\in \mathcal{C}$, and hence $T$ is not complementary vanishing by Theorem~\ref{thm:biconnected}.

If $T$ is a star, then $T\in \c{R}$ (since $T$ is the disjoint union of $K_1$'s joined to a $K_1$) and it is not difficult to show that $T\notin \c{C}$, so Theorem~\ref{thm:biconnected} implies that $T$ is complementary vanishing.  Alternatively, one can show that $T$ is complementary vanishing directly by taking $A\in \S(T)$ to be the adjacency matrix of $T$ and $B\in \S(\overline{T})$ to be the Laplacian matrix of the complement of $T$, since in this case $AB=O$. 
\end{proof}

By Theorem~\ref{thm:biconnected}, it suffices to characterize the set of graphs $\mathcal{M}$ in order to characterize the set of complementary vanishing  graphs. 
To this end, we determine exactly which graphs on at most 8 vertices are in $\c{M}$.

\begin{theorem}\label{thm:conditions}
A graph $G$ on $n\le 8$ vertices is in $\c{M}$ if and only if it is one of the graphs listed in Appendix A.
\end{theorem}


A precise break down of the number of pairs $\{G,\overline{G}\}$ of connected graphs on $n\leq 8$ vertices which are complementary vanishing  is given in Table \ref{tab:complementary vanishing ornotcomplementary vanishing }.

\begin{table}[h]
\begin{center}
\begin{tabular}{c|c|c|c}
 $n$ & not complementary vanishing & complementary vanishing  & total pairs\\
 \hline
 1 & 0 & 1 & 1\\
 2 & 0 & 0 & 0\\
 3 & 0 & 0 & 0\\
 4 & 1 & 0 & 1 \\
 5 & 5 & 0 & 5\\
 6 & 34 & 0 & 34\\
 7 & 327 & 4 & 331 \\
 8 & 4917 &  32 & 4949\
\end{tabular}
\end{center}
\caption{Each row breaks down the number of pairs $\{G,\overline{G}\}$ which are complementary vanishing, where $G$ and $\overline{G}$ are connected graphs on $n$ vertices.}
\label{tab:complementary vanishing ornotcomplementary vanishing }
\end{table}

\textbf{Organization and Notation.}
The rest of the paper is organized as follows.  In Section~\ref{Sec2}, we give  combinatorial conditions for when a graph is complementary vanishing  or not. In Section~\ref{Sec3}, we prove Theorem~\ref{thm:biconnected} by studying a slightly stronger notion of being complementary vanishing.  
In Section~\ref{Sec4}, we prove Theorem~\ref{thm:conditions} and describe several algorithms which can be used to test whether a graph is complementary vanishing  or not.
We close with a few open problems in Section~\ref{Sec5}. 

Throughout we use standard notation from graph theory and linear algebra.  Given a graph $G$, we define the \textit{open neighborhood} $N_G(v)$ to be the set of vertices adjacent to $v$ in $G$, and the \textit{closed neighborhood} $N_G[v]=N_G(v)\cup \{v\}$. We will drop the subscript on $N$ when the graph $G$ is clear from context.

  We write $O_{m\times n}$ for the $m\times n$ zero matrix, which we denote simply by $O$ whenever the dimensions are clear from context.  We use bold lower case letters to denote vectors, and in particular we write $\bzero$ and $\bone$ to denote the all 0's and all 1's vectors.

\section{Combinatorial Conditions for Complementary Vanishing Graphs}\label{Sec2}

The results in this section use structural properties of graphs to determine whether it is possible for a graph to be complementary vanishing. 
Most of our conditions will show that a graph cannot be complementary vanishing since the diagonal entries for $A$ and $B$ are over-constrained. 
The following observation leads to our first condition for the diagonal entries of $A$ and $B$.   Here and throughout the paper, if $S$ is a set of vertices of a graph $G$ then $\overline{S}:=V(G)\setminus S$.

\begin{observation}\label{obs:disjointNeighbors}
If $v\in V(G)$, then $N_{\overline{G}}[v] = \overline{N_G(v)}$ and $N_{\overline{G}}(v) = \overline{N_G[v]}$.
\end{observation}

\begin{lemma}\label{one.diag.ent.zero}
Let $A\in\mathcal S(G)$, $B\in\mathcal S(\overline G)$. 
If $AB = O$, then $a_{i,i}$ or $b_{i,i}$ is $0$ for each $1\leq i \leq n$.
\end{lemma}

\begin{proof}
Because $AB=O$, taking the dot product of the $i\text{th}$ row of $A$ and the $i\text{th}$ column of $B$ gives
\[
0 = \sum_{1\leq k\leq n} a_{i,k}b_{k,i}= a_{i,i}b_{i,i} + \sum_{k\in N_G(i)\cap N_{\overline{G}}(i)} a_{i,k}b_{k,i}
= a_{i,i}b_{i,i} + \sum_{k\in N_G(i)\cap \overline{N_{G}[i]}} a_{i,k}b_{k,i},
\]
where the second equality used, for example, that $a_{i,k}=0$ if $k\notin N_G[i]$ by definition of $\mathcal{S}(G)$, and the third used Observation~\ref{obs:disjointNeighbors}.  Since $N_G(i)\cap\overline{N_G[i]} = \emptyset$, this reduces to
\[
0=a_{i,i}b_{i,i}.
\]
This implies $a_{i,i} = 0$ or $b_{i,i} = 0$.
\end{proof}

The key fact in proving Lemma \ref{one.diag.ent.zero} is that $N_G(i)$ is contained in $N_{G}[i]$ (since this made the final sum equal to 0).
This step in the proof can be generalized to pairs of vertices $i\neq j$ when the closed neighborhood of $j$ contains the neighborhood of $i$.

\begin{lemma}\label{prop.subneighborhood}
Let $A\in\mathcal S(G)$ and $B\in\mathcal S(\overline G)$ be such that $AB = O$. For distinct vertices $i,j\in V(G)$, if $N_G(i)\subseteq N_G[j]$, then we have the following.
\begin{enumerate}
\item[(1)] If $i\sim_G j$
, then $b_{j,j} = 0$.
\item[(2)] If $i\not\sim_G j$
, then $a_{i,i}=0$.
\end{enumerate}
\end{lemma}
\begin{proof}
Let $i,j$ be distinct vertices with $N_G(i) \subseteq N_G[j]$. Then we have $N_G(i) \cap \overline{N_G[j]} = \emptyset$. 
Taking the dot product of the $i\text{th}$ row of $A$ and $j\text{th}$ column of $B$ gives
\begin{align*}
0 &= a_{i,i}b_{i,j} + a_{i,j}b_{j,j} + \sum_{k\in N_G(i)\cap \overline{N_{G}[j]}} a_{i,k}b_{k,j} \\
&= a_{i,i}b_{i,j} + a_{i,j}b_{j,j}.
\end{align*}

Suppose $i\sim_G j$. Then $a_{i,j}\neq0$ and $b_{i,j}=0$ (since $i\not\sim_{\overline G} j$).
Therefore, $a_{i,j}b_{j,j} = 0$, which implies that $b_{j,j} = 0$. A completely symmetric argument gives the result when $i\not\sim_G j$.
\end{proof}

We can also extract information from the dot product of the $i\text{th}$ row of $A$ and the $j$th column of $B$ if we ``almost'' have $N_G(i)\subseteq N_G[j]$. 


\begin{lemma}\label{prop.one.private.neighbor}
Let $A\in\mathcal S(G)$ and $B\in\mathcal S(\overline{G})$ be such that $AB = O$. For distinct vertices $i,j\in V(G)$, if $N_G(i)\setminus N_G[j] = \{v\}$ for some $v\in V(G)$, then we have the following. 
\begin{enumerate}
\item[(1)] If $i\sim_G j$
, then $b_{j,j} \neq 0$ and $a_{j,j} = 0$.
\item[(2)] If $i\not\sim_G j$
, then $a_{i,i} \neq 0$ and $b_{i,i} = 0$.
\end{enumerate}
\end{lemma}

\begin{proof}
Let $N_G(i) \setminus N_G[j] = \{v\}$. Then $N_G(i) \cap \overline{N_G[j]} = \{v\}$ and taking the dot product of the $i\text{th}$ row of $A$ and the $j\text{th}$ column of $B$ gives
\begin{align*}
0 &= a_{i,i}b_{i,j} + a_{i,j}b_{j,j} + \sum_{k\in N_G(i)\cap \overline{N_{G}[j]}} a_{i,k}b_{k,j} \\
&= a_{i,i}b_{i,j} + a_{i,j}b_{j,j} + a_{i,v}b_{v,j}
\end{align*}

Suppose $i\sim_G j$. Then we have that $a_{i,j}\neq0$ and $b_{i,j}=0$ (since $i\not\sim_{\overline G} j$). 
This implies that $a_{i,j}b_{j,j} + a_{i,v}b_{v,j} = 0$ where  $a_{i,j}\neq0$, $a_{i,v}\neq0$, and $b_{v,j}\neq0$. 
Solving for $b_{j,j}$ shows that  $b_{j,j}\neq0$. 
Moreover, $a_{j,j} = 0$ by Lemma \ref{one.diag.ent.zero}.  A completely symmetric argument gives the result when $i\not \sim_G j$.
\end{proof}

Applying both Lemmas \ref{prop.subneighborhood} and \ref{prop.one.private.neighbor} can lead to contradictions along the diagonal of $A$. 
This gives the following corollary, which helps motivate the set $\mathcal{C}$ defined in the introduction.

\begin{corollary}\label{cor:noInducedPath}
Let $G$ be a graph such that there exist distinct vertices $u,v,w$ with $v,w\notin N(u)$, $|N(u)\setminus N(w)|=1$, and $|N(u)\setminus N(v)|=0$.  
Then $G$ is not complementary vanishing.
\end{corollary}

\begin{proof}
For the sake of contradiction suppose that $G$ is complementary vanishing with matrices $A\in \mathcal S(G)$ and $B\in \mathcal S(\overline G).$
Notice that  $|N(u)\setminus N(w)|=1$ and $u\not\sim w$ implies $a_{u,u}\neq 0$ by Lemma \ref{prop.one.private.neighbor}.  
Furthermore, $N(u)\subseteq N(v)$ and $u\not\sim v$ implies $a_{u,u}=0$ by Lemma \ref{prop.subneighborhood}.
This is a contradiction,  so it follows that $G$ is not complementary vanishing.
\end{proof}

Lemma \ref{one.diag.ent.zero}, Lemma \ref{prop.subneighborhood}, and Lemma \ref{prop.one.private.neighbor}  are the elementary calculations we use to show that particular graphs are not complementary vanishing.
One improvement on these elementary calculations is to look for induced cycles with particular neighborhood properties. 

\begin{lemma}
\label{lem:oddcycle}
Let $G$ be a graph that has an induced odd cycle $C$ on vertices $1,2,\ldots,2k+1$ and a vertex $v$ such that $1,2,\ldots,2k+1\notin N(v)$ and \[\bigcup_{1\le i\leq 2k+1} N(i)\setminus \{1,\ldots,2k+1\} \subseteq N(v).\] 
If $A\in\mathcal S(G)$ and $B\in\mathcal S(\overline{G})$ satisfy $AB=O$, then $a_{i,i}\neq 0$ for some $i\in \{1,2,\ldots,2k+1\}$.
\end{lemma}

Note that $\bigcup N(i)\setminus\{1,\ldots,2k+1\} \subseteq N(v)$ is equivalent to saying that $v$ is adjacent to the entire ``external neighborhood'' of $\{1,\ldots,2k+1\}$. 

\begin{proof}
Assume for the sake of contradiction that there exist such $A,B$ with $a_{i,i}=0$ for all $i\in\{1,\dots,2k+1\}$. 
By taking the dot product of the $i\text{th}$ row of $A$ with the $v\text{th}$ column of $B$, we obtain (writing indices mod $2k+1$ and using the symmetry of $A$ and $B$) \[a_{i-1,i}b_{i-1,v}+a_{i,i+1}b_{i+1,v}=0,\] where we used that $a_{i,i}=0$ and the fact that the neighborhood of $v$ contains the external neighborhood of $C$.
Also note that each term in this sum is nonzero since implicitly we have assumed $v\notin \{1,\ldots,2k+1\}$. 

For convenience, define $c_{2i-1}=a_{i,i+1}$ and $c_{2i}=b_{i+1,v}$ for all $1\le i\le 2k+1$.
Notice that $c_j$ is an entry from $B$ if $j$ is even, and $c_j$ is an entry from $A$ when $j$ is odd.
The relationship above can be rewritten as $c_{2i-3}c_{2i-4}+c_{2i-1}c_{2i}=0$, where the indices of $c$ are written mod $4k+2$.  
In particular, this implies that amongst all of the terms $c_j c_{j+1}$, exactly $2k+1$ of them are negative.  Thus,
\[1=\sign\left(\prod_j c_j^2\right)=\prod_j \sign(c_jc_{j+1})=(-1)^{2k+1}=-1,\]
which is a contradiction.
\end{proof}

The following example shows the power of Lemma~\ref{lem:oddcycle}.

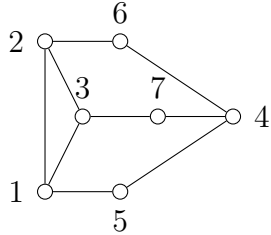
\begin{figure}[h]
\centering
\begin{tikzpicture}
\node[label={left:$1$}] (1) at (0,0) {};
\node[label={left:$2$}] (2) at (0,2) {};
\node[label={above:$3$}] (3) at (0.5,1) {};
\node[label={below:$5$}] (5) at (1,0) {};
\node[label={above:$6$}] (6) at (1,2) {};
\node[label={above:$7$}] (7) at (1.5,1) {};
\node[label={right:$4$}] (4) at (2.5,1) {};
\draw (1) -- (2) -- (3) -- (1);
\draw (1) -- (5) -- (4);
\draw (2) -- (6) -- (4);
\draw (3) -- (7) -- (4);
\end{tikzpicture}
\caption{A graph that is not complementary vanishing due to Lemma~\ref{lem:oddcycle}. Notice that Corollary~\ref{cor:noInducedPath} does not apply for this graph.}
\label{fig:oddcycle}
\end{figure}

\begin{example}
Let $G$ be the graph in Figure~\ref{fig:oddcycle}.  Suppose $A = \begin{bmatrix} a_{i,j} \end{bmatrix} \in \S(G)$ and $B = \begin{bmatrix} b_{i,j} \end{bmatrix} \in \S(\overline{G})$ are such that $AB = O$.  By applying the latter case of  Lemma~\ref{prop.one.private.neighbor} with $u = 1$ and $w = 5$, we know $b_{1,1} \neq 0$ and $a_{1,1} = 0$.  By symmetry, $a_{2,2} = a_{3,3} = 0$.  However, by Lemma~\ref{lem:oddcycle} with the odd cycle on $\{1,2,3\}$ and $v = 4$, at least one of $a_{1,1}$, $a_{2,2}$, and $a_{3,3}$ is nonzero, which is a contradiction.  Thus, $G$ is not complementary vanishing.
\end{example}

We end the section with some ways to construct complementary vanishing graphs.
We say that two vertices $u,v$ are \textit{twins} in a graph $G$ if $N(u)\sm \{v\}=N(v)\sm \{u\}$.  Note that we allow for twins to either be adjacent or nonadjacent.

\begin{proposition}
\label{prop:twin}
If $G$ is a graph such that every vertex has at least one twin,  then $G$ is complementary vanishing.
\end{proposition}


\begin{proof}
Partition the vertex set of $G$ into equivalence classes of twins so that 
$V(G) =V= V_1\cup \cdots \cup V_k$ where every $u,v\in V_i$ are twins. 
By the hypothesis, $|V_i|\geq 2$ for $i=1,\dots,k$.
Let $\vec{x}^{(i)}, \vec{y}^{(i)}\in \mathbb R^{|V|}$ be orthogonal vectors that are nonzero on entries indexed by $V_i$ and zero otherwise. 
Since $|V_i|\geq 2$ for $1\leq i \leq k$, we know that the vectors $\vec{x}^{(i)},\vec{y}^{(i)}$ exist for $1\leq i \leq k$. 
Notice that $(\vec{x}^{(i)})^\top\vec {y}^{(j)}=0$ for all $1\leq i,j\leq k$.

 Let $P=\{(i,j): uv\in E(G), u\in V_i, v\in V_j\}$ and $\overline P=\{(i,j): uv\notin E(G), u\in V_i, v\in V_j\}$.  Notice that every pair $(i,j)$ with $i,j\in \{1,\ldots,k\}$ distinct appears in exactly one of the sets $P$ or $\overline{P}$ since every two vertices in $V_i$ and $V_j$ are twins.
Let 
\[A= \sum_{(i,j)\in P} \vec x^{(i)}(\vec x^{(j)})^\top\] and 
\[B= \sum_{(i,j)\in \overline{P}} \vec y^{(i)}(\vec y^{(j)})^\top,\]
where each term is a square matrix indexed by $V$ resulting from an outer product.
Notice that $A\in \mathcal S(G)$, $B\in \mathcal S(\overline{G})$, and that
\[AB= \sum_{(i,j)\in P}\sum_{(\ell,m)\in \overline{P}} \vec x^{(i)}(\vec x^{(j)})^\top\vec y^{(\ell)}(\vec y^{(m)})^\top = O.\]
This completes the proof.
\end{proof}

Proposition \ref{prop:twin} is tight in the sense that there are graphs $G$ which are not complementary vanishing where all but one vertex in $G$ has a twin. 
Indeed, consider $P_3\sqcup P_2$
with  $uxv$ the copy of $P_3$ and $yw$ the copy of $P_2$.
Then $x$ does not have a twin, $|N(u)\setminus N(w)|=1$, $|N(u)\setminus N(v)|=0$ and $v,w\notin N(u)$. 
Thus, $P_3\sqcup P_2$ is not complementary vanishing by Corollary~\ref{cor:noInducedPath}.

\begin{lemma}\label{lem:duplication}
Let $G$ be a graph with  $A\in\S(G)$ and $B\in\S(\overline{G})$ such that $AB=O$.  If $a_{i,i} = 0$, then the graph $H$ obtained by adding a new vertex $j$ with $N_{H}(j)=N_G(i)$ is complementary vanishing.
\end{lemma}
\begin{proof}
Let $n = |V(G)|$.  For convenience of illustrating the matrices, we assume $i = n$ 
and write
\[ A = \begin{bmatrix}
 C & \bu \\
 \bu\trans & 0 \\
 \end{bmatrix}
 \text{ and }
 B = \begin{bmatrix}
 D & \bv \\
 \bv\trans & b \\ 
 \end{bmatrix}.\]
Thus, $AB = O$ can be expanded as 
\[\begin{aligned}
 CD + \bu\bv\trans &= O, \\
 C\bv + b\bu &= \bzero, \\
 \bu\trans D &= \bzero\trans, \\
 \bu\trans\bv &= 0.
\end{aligned}\]

Let $j = n + 1$ and $\zeta = \frac{1}{\sqrt{2}}\bone$, where $\bone$ is the all $1$'s vector in $\mathbb{R}^2$.  Then we may construct the following matrices
\[ A' = \begin{bmatrix}
 C & \bu\zeta\trans \\
 \zeta\bu\trans & O_{2,2} \\
 \end{bmatrix}
 \text{ and }
 B' = \begin{bmatrix}
 D & \bv\zeta\trans \\
 \zeta\bv\trans & E \\ 
 \end{bmatrix},\]
where $E$ is either  
\[\frac{b}{2}\begin{bmatrix} 1 & 1 \\ 1 & 1 \end{bmatrix} 
\text{ or } 
\begin{bmatrix} 1 & -1 \\ -1 & 1 \end{bmatrix}\]
depending on if $b \neq 0$ or $b = 0$.  Note that with this we always have $\zeta\trans E = b\zeta\trans$.

By direct computation of each block of $A'B'$, we have 
\[\begin{aligned}
CD + \bu\zeta\trans\zeta\bv\trans &= CD + \bu\bv\trans = O, \\
C\bv\zeta\trans + \bu\zeta\trans E &= (C\bv + b\bu)\zeta\trans = \bzero\zeta\trans = O, \\
\zeta\bu\trans D &= \zeta\bzero\trans = O, \\ 
\zeta\bu\trans\bv\zeta\trans &=  0\zeta\zeta\trans = O.
\end{aligned}\]
Along with the fact that $A'\in\S(H)$ and $B'\in\S(\overline{H})$, we have that $H$ and $\overline{H}$ are both complementary vanishing.
\end{proof}

Note that the above proof easily generalizes to duplicating a vertex any number of times or duplicating multiple vertices.

\section{Proof of Theorem~\ref{thm:biconnected}}\label{Sec3}



In order to prove Theorem~\ref{thm:biconnected}, we will  need to understand when the disjoint union and join of two graphs are complementary vanishing, and to do this we need consider  graphs which are in some sense ``robust'' with respect to being complementary vanishing.

\subsection{Robust Graphs}\label{subsection:robustgraphs}

A graph $G$ is said to be \emph{$\alpha$-robust} if there are matrices $A\in\S(G)$ and $B\in\mathcal S(\overline{G})$ such that $AB = O$ and $\mker(A)$ contains a nowhere-zero vector. 
Similarly we say that $G$ is \emph{$\beta$-robust} if there are matrices $A\in\S(G)$ and $B\in\S(\overline{G})$ such that $AB = O$ and $\mker(B)$ contains a nowhere-zero vector.   Note that $\alpha$-robust and $\beta$-robust graphs are in particular complementary vanishing.

\begin{proposition}\label{prop:robust}
Let $G$ and $H$ be two graphs.  Then the following are equivalent:
\begin{itemize}
    \item[(1a)] The graphs $G$ and $H$ are both $\alpha$-robust.
    \item[(1b)] The disjoint union $G\sqcup H$ is $\alpha$-robust.
    \item[(1c)] The disjoint union $G\sqcup H$ is complementary vanishing.
\end{itemize}
Similarly, the following are equivalent:
\begin{itemize}
    \item[(2a)] The graphs $G$ and $H$ are both $\beta$-robust.
    \item[(2b)] The join $G\vee H$ is $\beta$-robust.
    \item[(2c)] The join $G\vee H$ is complementary vanishing.
\end{itemize}
\end{proposition}

\begin{proof}
Suppose $G$ and $H$ are both $\alpha$-robust. 
By definition, there exist matrices $A_G\in\S(G)$, $A_H\in\S(H)$, $B_G\in\S(\overline{G})$, $B_H\in\S(\overline{H})$ such that $A_GB_G = O$ and $A_HB_H = O$.  
Moreover, $\mker(A_G)$ and $\mker(A_H)$ contain nowhere-zero vectors $\bv_G$ and $\bv_H$, respectively.  Thus, we have 
\[A = \begin{bmatrix}
 A_G & O \\ 
 O & A_H 
\end{bmatrix}\in\S(G\sqcup H),\ 
B = \begin{bmatrix}
 B_G & \bv_G\bv_H^\top \\
 \bv_H\bv_G^\top & B_H
\end{bmatrix}\in\S(\overline{G\sqcup H}), \]
and $AB = O$.  Moreover, the vector $\bv=\begin{bmatrix} \bv_G \\ \bv_H \end{bmatrix}$ is a nowhere-zero vector in $\mker(A)$.  Thus $G\sqcup H$ is $\alpha$-robust, showing that (1a) implies (1b).

A graph being $\alpha$-robust immediately implies that it is complementary vanishing, so  (1b) implies (1c).
To show (1c) implies (1a),  suppose $G\sqcup H$ is complementary vanishing. 
Then there are matrices 
\[A = \begin{bmatrix}
 A_G & O \\ 
 O & A_H 
\end{bmatrix}\in\S(G\sqcup H) \text{ and } 
B = \begin{bmatrix}
 B_G & C \\
 C^\top & B_H
\end{bmatrix}\in\S(\overline{G\sqcup H})\]
such that $AB=O$ for some matrices $A_G\in\S(G)$, $A_H\in\S(H)$, $B_G\in\S(\overline{G})$, $B_H\in\S(\overline{H})$, and some $C$ which is nowhere-zero.  
Thus, we have $A_GB_G = O$ and $A_GC = O$.  Since any column of $C$ is a nowhere-zero vector in $\mker(A_G)$, we know $G$ is $\alpha$-robust.  
Similarly, $H$ is also $\alpha$-robust, proving that (1c) implies (1a).

Notice that $G,H$ satisfying one of (2a), (2b), or (2c) is equivalent to $\overline{G},\overline{H}$ satisfying one of (1a), (1b), or (1c).
Therefore, the equivalence of (2a), (2b), and  (2c) follows from the equivalence of (1a), (1b), and (1c).
\end{proof}

The next lemma gives an effective way for finding nowhere-zero vectors $\bx$ such that $B\bx$ is also nowhere-zero.

\begin{lemma}\label{colspace}
If $B$ is a matrix which does not contain a row of zeros, then there exists a nowhere-zero vector $\bx$ such that $B\bx$ is nowhere-zero.
\end{lemma}

\begin{proof}
Suppose that $\bv$ is a vector in the column space of $B$ with the fewest $0$ entries.
For the sake of contradiction, suppose that $\bv_i=0$ for some $i$.
Since $B$ does not contain a row of zeros, there exists some column $\bc$ of $B$ such that $\bc_i \neq 0$. 
Choose $\varepsilon>0$ so that  \[\varepsilon\max_{j}\{|\bc_j|\}< \min_{\substack{j\\\bv_j\neq 0}}\{|\bv_j|\}.\]
Notice that $\bv+\varepsilon\bc$ has fewer $0$ entries than $\bv$ by construction.
This contradiction leads to the conclusion that $\bv$ is a nowhere-zero vector.

Suppose that $\bx$ is a vector with the minimum number of $0$ entries such that $B\bx$ is nowhere-zero (this vector exists by the previous paragraph). 
For the sake of contradiction, suppose that $\bx_i=0$. 
Let $\bc$ be the $i\text{th}$ column of $B$.
Choose $\varepsilon>0$ such that \[\varepsilon\max_{j}\{|\bc_j|\}< \min_{j}\{|(B\bx)_j|\}.\]
Recall that $\min_{j}\{|(B\bx)_j|\}$ is positive since $B\bx$ is nowhere-zero.
Notice that $\bx+ \varepsilon\be_i$ has exactly one less $0$ entry than $\bx$ and $B(\bx+\varepsilon\be_i) = B\bx+\varepsilon\bc$ is nowhere-zero by construction. 
This contradiction leads to the conclusion that $\bx$ is a nowhere-zero vector such that $B\bx$ is also nowhere-zero.
\end{proof}

By using this lemma, we give some sufficient conditions for a complementary vanishing graph to be $\alpha$- or $\beta$-robust.

\begin{lemma}\label{lem:dom}
Let $G$ be a complementary vanishing graph.
\begin{itemize}
\item[(1)] If $G$ has no dominating vertices, then $G$ is $\alpha$-robust.
\item[(2)] If $G$ has no isolated vertices, then $G$ is $\beta$-robust.
\end{itemize}
\end{lemma}

\begin{proof}
Since $G$ is complementary vanishing, there are matrices $A\in\S(G)$ and $B\in\S(\overline{G})$ such that $AB = O$.
First assume that $G$ has no dominating vertices.  This means $\overline{G}$ has no isolated vertices,  
and thus the column space of $B$ contains a nowhere-zero vector $\bv$ by Lemma \ref{colspace}. Having $AB = O$ implies $A\bv = \bzero$, so $\bv$ is a nowhere-zero vector in $\mker(A)$. Therefore, $G$ is $\alpha$-robust. 
An identical argument shows that if $G$ has no isolated vertices, then there is some nowhere-zero $\bx$ in the row space of $A$. 
This implies that $G$ is $\beta$-robust. 
\end{proof}

\begin{lemma}\label{robustlyrobust}
Let $G$ be a graph. If there exist matrices $A\in \S(G), B\in \S(\overline{G})$, and a nowhere-zero vector $\bx$ such that $AB=O$ and  $B\bx$ is nowhere-zero, then $G\sqcup K_1$ is $\beta$-robust.

\end{lemma}

\begin{proof}

Assume there exist matrices and vectors as in the hypothesis of the lemma.  
Let $\bv=B\bx$, which is nowhere-zero by assumption.
Define
\[A'=\begin{bmatrix}A & O\\ O & 0\end{bmatrix},\, B'=\begin{bmatrix}B & \bv\\ \bv^\top & \bx\trans \bv\end{bmatrix}.\]
Observe that $A'\in \S(G\sqcup K_1),B'\in \S(\overline{G\sqcup K_1})$. 
Furthermore, $A'B'=O$ since $AB=O$ and $A\bv=AB \bx=\bzero$.  
If $\bx'=\begin{bmatrix}-\bx\\ 1\end{bmatrix}$, then $\bx'$ is nowhere-zero and
\[B'\bx'=\begin{bmatrix}-B\bx + \bv\\ -{\bv}^\top\bx+\bx\trans \bv\end{bmatrix}=\bzero.\]
Thus, $G\sqcup K_1$ is $\beta$-robust.
\end{proof}

\begin{corollary} \label{cor:robustlyrobust}
If $G$ is complementary vanishing and does not have a dominating vertex, then $G\sqcup K_1$ is $\beta$-robust.
\end{corollary}
\begin{proof}
By assumption there exist $A\in \S(G),B\in \S(\overline{G})$ with $AB=O$. 
Since $G$ has no dominating vertices, $\overline{G}$ has no isolated vertices and $B$ does not have a row of zeros. 
By Lemma \ref{colspace}, there exists nowhere-zero vector $\bx$ such that $B\bx$ is also nowhere-zero. 
Therefore, $G\sqcup K_1$ is $\beta$-robust by Lemma \ref{robustlyrobust}.
\end{proof}

\subsection{Completing the Proof of Theorem~\ref{thm:biconnected}}\label{subsection:proofof1}

We first recall the statement of Theorem~\ref{thm:biconnected}.  Let $\c{M}$  denote the set of graphs such that $G$ is complementary vanishing, $G$ is connected, and $\overline{G}$ is connected.  
One can check that $K_1\in \c{M}$.  
Let $\c{R}$ denote the smallest set of graphs which contains $\c{M}$ and which is closed under taking disjoint unions, joins, and complements. 
For example, having $K_1\in \c{M}$ implies that $\c{R}$ contains every complete multipartite graph and every threshold graph.
Let $\mathcal C$ denote all the graphs $G$ such that either in $G$ or $\overline{G}$ there exist distinct vertices $u,v,w$ with $v,w\notin N(u)$, $|N(u)\setminus N(w)|=1$, and $|N(u)\setminus N(v)|=0$. 
Notice that graphs in $\mathcal C$ are not complementary vanishing by Corollary  \ref{cor:noInducedPath}.   Our aim is to prove the following.


\begin{theorem1.5}
    A graph $G$ is complementary vanishing if and only if $G\in \c{R}\setminus \mathcal C$.
\end{theorem1.5}

Part of Theorem \ref{thm:biconnected} can be proven immediately.

\begin{proposition}\label{prop:easyway}
If $G\notin \c R$, then $G$ is not complementary vanishing.
\end{proposition}

\begin{proof}
We will prove the proposition by minimal counterexample. 
Suppose that $G$ is complementary vanishing and a vertex minimal graph not in $\c R.$
This implies that either $G$ or $\overline G$ is not connected as otherwise $G\in \c M\subseteq \c R.$ 
This implies that either $G = H\sqcup K$ or $G = H\vee K$ where $H,K$ are non-empty complementary vanishing graphs by Proposition~\ref{prop:robust}. 
Since $G$ is assumed to be a vertex minimal counterexample, it follows that $H,K\in \c R$.   However, this implies that $G\in \c R$; which is a contradiction. 
\end{proof}

In order to prove Theorem \ref{thm:biconnected}, we need to keep track of the graphs in $\c R$ that are close to falling within $\mathcal C$. 
Let $\c{D}\subseteq \c{R}$ denote the set of graphs $G$ that contain a leaf which is adjacent to a vertex of degree at least $2$.
For example, stars on at least three vertices are in $\c D$, as is a triangle with a pendant (notice that both of these graphs are in $\c R$). The main observation regarding $\c{D}$ is the following.


\begin{lemma}\label{lem:danger}
If $D\in \c{D}$ and $H$ is a graph on at least one vertex, then $D\sqcup H\in \mathcal C$. Furthermore, $D\sqcup H$ is not complementary vanishing.
\end{lemma}
\begin{proof}
By definition of $\c D$, the graph $D$ contains a leaf $u$  adjacent to $x$ such that there exists $v\in N(x)\setminus \{u\}.$  
Notice that $|N(u)\setminus N(v)| = 0$. 
Since $H$ is not the empty graph, there exists a vertex $w$ in $H$ such that $|N(u)\setminus N(w)| = 1$. 
By Corollary \ref{cor:noInducedPath}, $D\sqcup H\in\mathcal{C}$ is not complementary vanishing.
\end{proof}

To prove Theorem~\ref{thm:biconnected}, we prove the following stronger version to aid with an inductive proof.

\begin{theorem}\label{thm:nonminTech}
Let $G\in \c{R}$.
\begin{itemize}
\item[(1)] If $G\in \mathcal C$, then $G$ is not complementary  vanishing.
\item[(2)] If $G\notin \mathcal C$ and $G\in \c{D}$, then $G$ is $\beta$-robust but not $\alpha$-robust.
\item[(3)] If $G\notin \mathcal C$ and  $\overline{G}\in \c{D}$, then $G$ is $\alpha$-robust but not $\beta$-robust.
\item[(4)] If $G\notin \mathcal C$ and  $G,\overline{G}\notin \c{D}$, then $G$ is both $\alpha$ and $\beta$-robust.
\end{itemize}
\end{theorem}

\begin{proof}
\textbf{Statement (1):} This follows immediately from Corollary~\ref{cor:noInducedPath}.

\textbf{Statement (2):} Suppose that $G\in (\c D\setminus \mathcal C)\cap \c R$. 
Since $G \in \c D$, it follows that $G\sqcup K_1$ is not complementary vanishing by Lemma \ref{lem:danger}. 
By Proposition \ref{prop:robust} and the fact that $K_1$ is $\alpha$-robust, it follows that $G$ is not $\alpha$-robust.

For the sake of contradiction, suppose that $G$  is disconnected. 
This implies that $G=H\sqcup K$, and without loss of generality we can assume $H$ contains a leaf with a neighbor whose degree is at least $2$ since $G\in \mathcal D.$ 
In particular, $H\in \c D$, and therefore, $G\in \mathcal C$ by Lemma~\ref{lem:danger}. 
This is a contradiction.
Thus, we can assume that $G$ is connected.

For the sake of contradiction, suppose that $G$ is a 
counterexample to the statement. 
If $G\in \c M$, then $G$ is complementary vanishing by definition, and it is $\beta$-robust by Lemma~\ref{lem:dom} since $G$ is connected (and since $G\ne K_1$ because $K_1\notin \c{D}$). 
Therefore, $G\notin \c M$.  Thus, there exists non-empty  $H,K\in \c R$ such that $G= H\vee K$ or $G= H\sqcup K$, and we must have $G=H\vee K$ since $G$ is connected.
However, $G\in \c{D}$ must have a leaf, which is only possible if $H=K=K_1$. 
In this case, $G$ does not have a vertex of degree at least $2$, which is a contradiction to $G\in \c{D}$. 

\textbf{Statement (3):} Suppose $\overline G\in \c D$ and $G\notin \mathcal C$.
Since $\mathcal C$ is closed under complements, it follows that $\overline G\notin \mathcal C$. 
By statement (2), we see that $\overline G$ is $\beta$-robust but not $\alpha$-robust. 
Thus, $G$ is $\alpha$-robust but not $\beta$-robust. 

\textbf{Statement (4):} Suppose $G\notin \mathcal C$ and $G,\overline G\notin \c D$ is a vertex minimal counterexample.  It is easy to see that $G=K_1$ is both $\alpha$ and $\beta$-robust, so this is not a counterexample.
Observe that graphs in $\c M\sm \{K_1\}$ are  complementary vanishing with no isolated vertices or dominating vertices.
Therefore, graphs in $\c M$ are both $\alpha$ and $\beta$-robust by Lemma  \ref{lem:dom}.
Thus we can assume that $G\notin \c M$.
In particular, $G=H\sqcup K$ or $G= H\vee K$ for some non-empty $H,K\in \c R$.
Since $\mathcal C$ is closed under complements, we can assume that $G= H\sqcup K$ without loss of generality. 

Since $G\notin \mathcal C$, it follows from Lemma~\ref{lem:danger} that neither $H$ nor $K$ is in $\c D\cup \mathcal C$.  If $\overline{H}\in \c D$, then $H$ is $\alpha$-robust by statement (3), and if $\overline{H}\notin \c D$ then $H$ is $\alpha$-robust since $G$ is a vertex minimal counterexample to statement (4).  Similarly we conclude that $K$ is $\alpha$-robust, and hence by Proposition~\ref{prop:robust} we see that $G$ is complementary vanishing and $\alpha$-robust. 

If $G$ does not have an isolated vertex, then we are done by Lemma \ref{lem:dom}. Therefore, we can assume that $K= K_1$.
Notice that if $H=K_1$, then it is easy to show that $G$ is not a counterexample. 
Furthermore, if $H$  does not have a dominating vertex, then $G$ is $\beta$-robust by Corollary \ref{cor:robustlyrobust}. 
Thus, we can assume  $H$ contains at least two vertices and a dominating vertex. 
However, this implies  $\overline G$ has a leaf which is adjacent to a vertex of degree at least two. 
This is a contradiction, since we assume  $\overline G\notin \c D$. 
\end{proof}
We can now prove our main result.
\begin{proof}[Proof of Theorem~\ref{thm:biconnected}]
Suppose that $G\notin \c R\setminus \mathcal C$. 
If $G\in \mathcal C$, then $G$ is not complementary vanishing by Corollary \ref{cor:noInducedPath}.
If $G\notin \c R$, then $G$ is not complementary vanishing by Proposition~\ref{prop:easyway}. 

Suppose that $G\in \c R\setminus \mathcal C$. 
There are three cases which exhaust all possibilities: either $G\in \c D$, $\overline G\in \c D$, or $G, \overline G\notin \c D$. 
In any case, $G$ is complementary vanishing by Theorem~\ref{thm:nonminTech}.
\end{proof}

\section{Algorithmic approaches and the proof of Theorem~\ref{thm:conditions}}\label{Sec4}

In this section we introduce two algorithmic methods.  The first uses Gr\"obner basis arguments to conclude that a graph is not complementary vanishing.  
The second is an algorithm that takes a matrix $A\in\S(G)$ and checks  for the existence of a matrix $B\in\S(\overline{G})$ such that $AB = O$.  Using this second algorithm together with a randomly selected $A\in\S(G)$ (or $B\in\S(\overline{G})$) will allow us to find certificates for a graph being complementary vanishing.

\subsection{Gr\"obner basis}
\label{subsec:grob}

\setlength{\fboxrule}{0pt}

\tikzset{every picture/.append style={ultra thick, scale=0.75}}
\tikzset{every node/.style={circle,draw,inner sep=1pt}}
\tikzset{every title node/.style={inner sep=1pt}}


\def\KONE{\fbox{\begin{tikzpicture}
\node (1) at (0,0) {1};
\node[shape=rectangle, color=white, text=black] (title) at (0,-1) {$K_{1}: \texttt{\textsf{@}}$};
\end{tikzpicture}}}

\def\KONEbar{\fbox{\begin{tikzpicture}
\node (1) at (0,0) {1};
\node[shape=rectangle, color=white, text=black] (title) at (0,-1) {$\bar{K}_{1}: \texttt{\textsf{@}}$};
\end{tikzpicture}}}


\def\KTWO{\fbox{\begin{tikzpicture}
\node (1) at (-.5,0) {1};
\node (2) at (0.914,1) {2};
\node[shape=rectangle, color=white, text=black] (title) at (0,-1) {$K_{2}$};
\draw (1) -- (2);
\end{tikzpicture}}}

\def\KTWObar{\fbox{\begin{tikzpicture}
\node (1) at (-.5,0) {1};
\node (2) at (0.914,1) {2};
\node[shape=rectangle, color=white, text=black] (title) at (0,-1) {$\bar{K}_{2}$};
\end{tikzpicture}}}


\def\KONETWO{\fbox{\begin{tikzpicture}
\node (1) at (-0.5,0) {1};
\node (2) at (0.914,1) {2};
\node (3) at (0.914,-1) {3};
\node[shape=rectangle, color=white, text=black] (title) at (0,-2) {$K_{2,1}$};
\draw (1) -- (2);
\draw (1) -- (3);
\end{tikzpicture}}}

\def\KONETWObar{\fbox{\begin{tikzpicture}
\node (1) at (-0.5,0) {1};
\node (2) at (0.914,1) {2};
\node (3) at (0.914,-1) {3};
\node[shape=rectangle, color=white, text=black] (title) at (0,-2) {$\bar{K}_{2,1}$};
\draw (2) -- (3);
\end{tikzpicture}}}

\def\KTHREE{\fbox{\begin{tikzpicture}
\node (1) at (-0.5,0) {1};
\node (2) at (0.914,1) {2};
\node (3) at (0.914,-1) {3};
\node[shape=rectangle, color=white, text=black] (title) at (0,-2) {$K_{3}$};
\draw (1) -- (2) -- (3) -- (1);
\end{tikzpicture}}}

\def\KTHREEbar{\fbox{\begin{tikzpicture}
\node (1) at (-0.5,0) {1};
\node (2) at (0.914,1) {2};
\node (3) at (0.914,-1) {3};
\node[shape=rectangle, color=white, text=black] (title) at (0,-2) {$\bar{K}_{3}$};
\end{tikzpicture}}}


\def\KONETHREE{\fbox{\begin{tikzpicture}
\node (1) at (0,0) {1};
\node (2) at (1.414,1.414) {2};
\node (3) at (1.414,-1.414) {3};
\node (4) at (-2,0) {4};
\node[shape=rectangle, color=white, text=black] (title) at (0,-2) {$K_{3,1}$};
\draw (1) -- (2);
\draw (1) -- (3);
\draw (1) -- (4);
\end{tikzpicture}}}

\def\KONETHREEbar{\fbox{\begin{tikzpicture}
\node (1) at (0,0) {1};
\node (2) at (1.414,1.414) {2};
\node (3) at (1.414,-1.414) {3};
\node (4) at (-2,0) {4};
\node[shape=rectangle, color=white, text=black] (title) at (0,-2) {$\bar{K}_{3,1}$};
\draw (2) -- (3) -- (4) -- (2);
\end{tikzpicture}}}

\def\CFOUR{\fbox{\begin{tikzpicture}
\node (1) at (-1,-1) {1};
\node (2) at (1,-1) {2};
\node (3) at (-1,1) {3};
\node (4) at (1,1) {4};
\node[shape=rectangle, color=white, text=black] (title) at (0,-2) {$K_{2,2}$};
\draw (1) -- (3) -- (2) -- (4) -- (1);
\end{tikzpicture}}}

\def\CFOURbar{\fbox{\begin{tikzpicture}
\node (1) at (-1,-1) {1};
\node (2) at (1,-1) {2};
\node (3) at (-1,1) {3};
\node (4) at (1,1) {4};
\node[shape=rectangle, color=white, text=black] (title) at (0,-2) {$\bar{K}_{2,2}$};
\draw (1) -- (2);
\draw (3) -- (4);
\end{tikzpicture}}}

\def\KONEONETWO{\fbox{\begin{tikzpicture}
\node (1) at (0,-1) {1};
\node (2) at (0,1) {2};
\node (3) at (-1.414,0) {3};
\node (4) at (1.414,0) {4};
\node[shape=rectangle, color=white, text=black] (title) at (0,-2) {$K_{2,1,1}$};
\draw (1) -- (2) -- (3) -- (1) -- (4) -- (2);
\end{tikzpicture}}}

\def\KONEONETWObar{\fbox{\begin{tikzpicture}
\node (1) at (0,-1) {1};
\node (2) at (0,1) {2};
\node (3) at (-1.414,0) {3};
\node (4) at (1.414,0) {4};
\node[shape=rectangle, color=white, text=black] (title) at (0,-2) {$\bar{K}_{2,1,1}$};
\draw (3) -- (4);
\end{tikzpicture}}}

\def\KFOUR{\fbox{\begin{tikzpicture}
\node (1) at (-1,-1) {1};
\node (2) at (1,-1) {2};
\node (3) at (-1,1) {3};
\node (4) at (1,1) {4};
\node[shape=rectangle, color=white, text=black] (title) at (0,-2) {$K_4$};
\draw (1) -- (2) -- (3) -- (4) -- (1);
\draw (1) -- (3);
\draw (2) -- (4);
\end{tikzpicture}}}

\def\KFOURbar{\fbox{\begin{tikzpicture}
\node (1) at (-1,-1) {1};
\node (2) at (1,-1) {2};
\node (3) at (-1,1) {3};
\node (4) at (1,1) {4};
\node[shape=rectangle, color=white, text=black] (title) at (0,-2) {$\bar{K}_4$};
\end{tikzpicture}}}


\def\KONEFOUR{\fbox{\begin{tikzpicture}
\node (1) at (0,0) {1};
\node (2) at (-1.414,1.414) {2};
\node (3) at (1.414,1.414) {3};
\node (4) at (1.414,-1.414) {4};
\node (5) at (-1.414,-1.414) {5};
\node[shape=rectangle, color=white, text=black] (title) at (0,-2) {$K_{4,1}$};
\draw (1) -- (2);
\draw (1) -- (3);
\draw (1) -- (4);
\draw (1) -- (5);
\end{tikzpicture}}}

\def\KONEFOURbar{\fbox{\begin{tikzpicture}
\node (1) at (0,0) {1};
\node (2) at (-1.414,1.414) {2};
\node (3) at (1.414,1.414) {3};
\node (4) at (1.414,-1.414) {4};
\node (5) at (-1.414,-1.414) {5};
\node[shape=rectangle, color=white, text=black] (title) at (0,-2) {$\bar{K}_{4,1}$};
\draw (2) -- (3) -- (4) -- (5) -- (2);
\end{tikzpicture}}}

\def\HALFBOWTIE{\fbox{\begin{tikzpicture}
\node (1) at (0,0) {1};
\node (2) at (-1.414,1.414) {2};
\node (3) at (1.414,1.414) {3};
\node (4) at (1.414,-1.414) {4};
\node (5) at (-1.414,-1.414) {5};
\node[shape=rectangle, color=white, text=black] (title) at (0,-2) {Half Bowtie};
\draw (1) -- (2); \draw (1) -- (3); \draw (1) -- (4); \draw (1) -- (5); \draw (3) -- (4); 
\end{tikzpicture}}}

\def\HALFBOWTIEbar{\fbox{\begin{tikzpicture}
\node (1) at (1,0) {1};
\node (2) at (-1.414,1.414) {2};
\node (3) at (1.414,1.414) {3};
\node (4) at (1.414,-1.414) {4};
\node (5) at (-1.414,-1.414) {5};
\node[shape=rectangle, color=white, text=black] (title) at (0,-2) {Is this CV?};
\draw (2) -- (3); \draw (2) -- (4); \draw (2) -- (5); \draw (3) -- (5); \draw (4) -- (5); 
\end{tikzpicture}}}

\def\KTWOTHREE{\fbox{\begin{tikzpicture}
\node (1) at (-1.5,1.414) {1};
\node (2) at (-1,0) {2};
\node (3) at (-1.5,-1.414) {3};
\node (4) at (1.414,-1.414) {4};
\node (5) at (1.414,1.414) {5};
\node[shape=rectangle, color=white, text=black] (title) at (0,-2.1) {$K_{3,2}$};
\draw (1) -- (4);
\draw (1) -- (5);
\draw (2) -- (4);
\draw (2) -- (5);
\draw (3) -- (4);
\draw (3) -- (5);
\end{tikzpicture}}}

\def\KTWOTHREEbar{\fbox{\begin{tikzpicture}
\node (1) at (-1.5,1.414) {1};
\node (2) at (-1,0) {2};
\node (3) at (-1.5,-1.414) {3};
\node (4) at (1.414,-1.414) {4};
\node (5) at (1.414,1.414) {5};
\node[shape=rectangle, color=white, text=black] (title) at (0,-2.1) {$\bar{K}_{3,2}$};
\draw (1) -- (2) -- (3) -- (1);
\draw (4) -- (5);
\end{tikzpicture}}}

\def\KTHREEONEONE{\fbox{\begin{tikzpicture}
\node (1) at (-1.5,1.414) {1};
\node (2) at (-1,0) {2};
\node (3) at (-1.5,-1.414) {3};
\node (4) at (1.414,-1.414) {4};
\node (5) at (1.414,1.414) {5};
\node[shape=rectangle, color=white, text=black] (title) at (0,-2.1) {$K_{3,1,1}$};
\draw (1) -- (4);
\draw (1) -- (5);
\draw (2) -- (4);
\draw (2) -- (5);
\draw (3) -- (4);
\draw (3) -- (5);
\draw (4) -- (5);
\end{tikzpicture}}}

\def\KTHREEONEONEbar{\fbox{\begin{tikzpicture}
\node (1) at (-1.5,1.414) {1};
\node (2) at (-1,0) {2};
\node (3) at (-1.5,-1.414) {3};
\node (4) at (1.414,-1.414) {4};
\node (5) at (1.414,1.414) {5};
\node[shape=rectangle, color=white, text=black] (title) at (0,-2.1) {$\bar{K}_{3,1,1}$};
\draw (1) -- (2) -- (3) -- (1);
\end{tikzpicture}}}

\def\KTWOTWOONE{\fbox{\begin{tikzpicture}
\node (2) at (-1, 1.25) {2};
\node (3) at (-1, -1.25) {3};
\node (4) at (1, -1.25) {4};
\node (5) at (1, 1.25) {5};
\node (1) at (2.5,0) {1};
\node[shape=rectangle, color=white, text=black] (title) at (0,-2) {$K_{2,2,1}$};
\draw (1) -- (2);
\draw (1) -- (3);
\draw (1) -- (4);
\draw (1) -- (5);
\draw (5) -- (2) -- (4) -- (3) -- (5);
\end{tikzpicture}}}

\def\KTWOTWOONEbar{\fbox{\begin{tikzpicture}
\node (2) at (-1, 1.25) {2};
\node (3) at (-1, -1.25) {3};
\node (4) at (1, -1.25) {4};
\node (5) at (1, 1.25) {5};
\node (1) at (2.5,0) {1};
\node[shape=rectangle, color=white, text=black] (title) at (0,-2) {$\bar{K}_{2,2,1}$};
\draw (2) -- (3);
\draw (4) -- (5);
\end{tikzpicture}}}

\def\KTWOONEONEONE{\fbox{\begin{tikzpicture}[scale=2]
\node (1) at (-0.02, 0.995) {1};
\node (2) at (0.89, 0.262) {2};
\node (3) at (-0.682, -0.787) {3};
\node (4) at (0.485, -0.835) {4};
\node (5) at (-1.0, 0.341) {5};
\node[shape=rectangle, color=white, text=black] (title) at (0,-1.25) {$K_{2,1,1,1}$};
\draw (1) -- (3); \draw (1) -- (4); \draw (1) -- (5);
\draw (2) -- (3); \draw (2) -- (4); \draw (2) -- (5);
\draw (3) -- (4) -- (5) -- (3);
\end{tikzpicture}}}

\def\KTWOONEONEONEbar{\fbox{\begin{tikzpicture}[scale=2]
\node (1) at (-0.02, 0.995) {1};
\node (2) at (0.89, 0.262) {2};
\node (3) at (-0.682, -0.787) {3};
\node (4) at (0.485, -0.835) {4};
\node (5) at (-1.0, 0.341) {5};
\node[shape=rectangle, color=white, text=black] (title) at (0,-1.25) {$\bar{K}_{2,1,1,1}$};
\draw (1) -- (2);
\end{tikzpicture}}}

\def\KFIVE{\fbox{\begin{tikzpicture}[scale=2]
\node (1) at (-0.02, 0.995) {1};
\node (2) at (0.89, 0.262) {2};
\node (3) at (-0.682, -0.787) {3};
\node (4) at (0.485, -0.835) {4};
\node (5) at (-1.0, 0.341) {5};
\node[shape=rectangle, color=white, text=black] (title) at (0,-1.25) {$K_{5}$};
\draw (1) -- (2) -- (3) -- (4) -- (5) -- (1);
\draw (1) -- (3) -- (5) -- (2) -- (4) -- (1);
\end{tikzpicture}}}

\def\KFIVEbar{\fbox{\begin{tikzpicture}[scale=2]
\node (1) at (-0.02, 0.995) {1};
\node (2) at (0.89, 0.262) {2};
\node (3) at (-0.682, -0.787) {3};
\node (4) at (0.485, -0.835) {4};
\node (5) at (-1.0, 0.341) {5};
\node[shape=rectangle, color=white, text=black] (title) at (0,-1.25) {$\bar{K}_{5}$};
\end{tikzpicture}}}


\def\KONEFIVE{\fbox{\begin{tikzpicture}[scale=2]
\node (1) at (-0.062, 0.024) {1};
\node (2) at (-0.02, 0.995) {2};
\node (3) at (0.89, 0.262) {3};
\node (4) at (-0.682, -0.787) {4};
\node (5) at (0.485, -0.835) {5};
\node (6) at (-1.0, 0.341) {6};
\node[shape=rectangle, color=white, text=black] (title) at (0,-1.25) {$K_{5,1}$};
\draw (6) -- (1); \draw (2) -- (1); \draw (3) -- (1); \draw (4) -- (1); \draw (5) -- (1); 
\end{tikzpicture}}}

\def\KONEFIVEbar{\fbox{\begin{tikzpicture}[scale=2]
\node (1) at (-0.062, 0.024) {1};
\node (2) at (-0.02, 0.995) {2};
\node (3) at (0.89, 0.262) {3};
\node (4) at (-0.682, -0.787) {4};
\node (5) at (0.485, -0.835) {5};
\node (6) at (-1.0, 0.341) {6};
\node[shape=rectangle, color=white, text=black] (title) at (0,-1.25) {$\bar{K}_{5,1}$};
\draw (2) -- (3) -- (4) -- (5) -- (6) -- (2);
\draw (2) -- (4) -- (6) -- (3) -- (5) -- (2);
\end{tikzpicture}}}

\def\KFOURTWO{\fbox{\begin{tikzpicture}[scale=2]
\node (1) at (0.336, -0.654) {1};
\node (2) at (-1.0, 0.621) {2};
\node (3) at (-0.912, -0.795) {3};
\node (4) at (0.237, 0.824) {4};
\node (5) at (-0.7, -0.076) {5};
\node (6) at (0.05, 0.08) {6};
\node[shape=rectangle, color=white, text=black] (title) at (-.3,-1.15) {$K_{4,2}$};
\draw (1) -- (5); \draw (1) -- (6); \draw (2) -- (5); \draw (2) -- (6); \draw (3) -- (5); 
\draw (3) -- (6); \draw (4) -- (5); \draw (4) -- (6); 
\end{tikzpicture}}}

\def\KFOURTWObar{\fbox{\begin{tikzpicture}[scale=2]
\node (1) at (0.336, -0.654) {1};
\node (2) at (-1.0, 0.621) {2};
\node (3) at (-0.912, -0.795) {3};
\node (4) at (0.237, 0.824) {4};
\node (5) at (-0.7, -0.076) {5};
\node (6) at (0.05, 0.08) {6};
\node[shape=rectangle, color=white, text=black] (title) at (-.3,-1.15) {$\bar{K}_{4,2}$};
\draw (1) -- (2); \draw (1) -- (3); \draw (1) -- (4); \draw (2) -- (3); \draw (2) -- (4); \draw (3) -- (4); \draw (5) -- (6); 
\end{tikzpicture}}}

\def\KFOURONEONE{\fbox{\begin{tikzpicture}[scale=2]
\node (1) at (0.336, -0.654) {1};
\node (2) at (-1.0, 0.621) {2};
\node (3) at (-0.912, -0.795) {3};
\node (4) at (0.237, 0.824) {4};
\node (5) at (-0.7, -0.076) {5};
\node (6) at (0.05, 0.08) {6};
\node[shape=rectangle, color=white, text=black] (title) at (-.3,-1.15) {$K_{4,1,1}$};
\draw (1) -- (5); \draw (1) -- (6); \draw (2) -- (5); \draw (2) -- (6); \draw (3) -- (5); 
\draw (3) -- (6); \draw (4) -- (5); \draw (4) -- (6); \draw (5) -- (6); 
\end{tikzpicture}}}

\def\KFOURONEONEbar{\fbox{\begin{tikzpicture}[scale=2]
\node (1) at (0.336, -0.654) {1};
\node (2) at (-1.0, 0.621) {2};
\node (3) at (-0.912, -0.795) {3};
\node (4) at (0.237, 0.824) {4};
\node (5) at (-0.7, -0.076) {5};
\node (6) at (0.05, 0.08) {6};
\node[shape=rectangle, color=white, text=black] (title) at (-.3,-1.15) {$\bar{K}_{4,1,1}$};
\draw (1) -- (2); \draw (1) -- (3); \draw (1) -- (4); \draw (2) -- (3); \draw (2) -- (4); \draw (3) -- (4); 
\end{tikzpicture}}}

\def\KTHREETHREE{\fbox{\begin{tikzpicture}[scale=2]
\node (1) at (-1, 0.75) {1};
\node (2) at (0, 0.75) {2};
\node (3) at (1, 0.75) {3};
\node (4) at (-1, -0.75) {4};
\node (5) at (0, -0.75) {5};
\node (6) at (1, -0.75) {6};
\node[shape=rectangle, color=white, text=black] (title) at (-.1,-1.15) {$K_{3,3}$};
\draw (1) -- (4); \draw (1) -- (5); \draw (1) -- (6); \draw (2) -- (4); \draw (2) -- (5); 
\draw (2) -- (6); \draw (3) -- (4); \draw (3) -- (5); \draw (3) -- (6); 
\end{tikzpicture}}}

\def\KTHREETHREEbar{\fbox{\begin{tikzpicture}[scale=2]
\node (1) at (-1, 0.75) {1};
\node (2) at (0, 0.4) {2};
\node (3) at (1, 0.75) {3};
\node (4) at (-1, -0.75) {4};
\node (5) at (0, -0.4) {5};
\node (6) at (1, -0.75) {6};
\node[shape=rectangle, color=white, text=black] (title) at (-.1,-1.15) {$\bar{K}_{3,3}$};
\draw (1) -- (2); \draw (1) -- (3); \draw (2) -- (3); \draw (4) -- (5); \draw (4) -- (6); \draw (5) -- (6); 
\end{tikzpicture}}}

\def\KTHREETWOONE{\fbox{\begin{tikzpicture}[scale=2]
\node (1) at (-1, 0.75) {1};
\node (2) at (0, 0.45) {2};
\node (3) at (1, 0.75) {3};
\node (4) at (-1, -0.75) {4};
\node (5) at (0, -0.45) {5};
\node (6) at (1, -0.75) {6};
\node[shape=rectangle, color=white, text=black] (title) at (-.1,-1.15) {$K_{3,2,1}$};
\draw (1) -- (4); \draw (1) -- (5); \draw (1) -- (6); \draw (2) -- (4); \draw (2) -- (5); 
\draw (2) -- (6); \draw (3) -- (4); \draw (3) -- (5); \draw (3) -- (6); \draw (4) -- (5) -- (6);
\end{tikzpicture}}}

\def\KTHREETWOONEbar{\fbox{\begin{tikzpicture}[scale=2]
\node (1) at (-1, 0.75) {1};
\node (2) at (0, 0.4) {2};
\node (3) at (1, 0.75) {3};
\node (4) at (-1, -0.75) {4};
\node (5) at (0, -0.4) {5};
\node (6) at (1, -0.75) {6};
\node[shape=rectangle, color=white, text=black] (title) at (-.1,-1.15) {$\bar{K}_{3,2,1}$};
\draw (1) -- (2); \draw (1) -- (3); \draw (2) -- (3); \draw (4) -- (6);
\end{tikzpicture}}}

\def\KTHREEONEONEONE{\fbox{\begin{tikzpicture}[scale=2]
\node (1) at (-1, 0.75) {1};
\node (2) at (0, 0.45) {2};
\node (3) at (1, 0.75) {3};
\node (4) at (-1, -0.75) {4};
\node (5) at (0, -0.45) {5};
\node (6) at (1, -0.75) {6};
\node[shape=rectangle, color=white, text=black] (title) at (0,-1.15) {$K_{3,1,1,1}$};
\draw (1) -- (4); \draw (1) -- (5); \draw (1) -- (6); \draw (2) -- (4); \draw (2) -- (5); \draw (2) -- (6); \draw (3) -- (4); \draw (3) -- (5); \draw (3) -- (6); \draw (4) -- (5); \draw (4) -- (6); \draw (5) -- (6); 
\end{tikzpicture}}}

\def\KTHREEONEONEONEbar{\fbox{\begin{tikzpicture}[scale=2]
\node (1) at (-1, 0.75) {1};
\node (2) at (0, 0.4) {2};
\node (3) at (1, 0.75) {3};
\node (4) at (-1, -0.75) {4};
\node (5) at (0, -0.4) {5};
\node (6) at (1, -0.75) {6};
\node[shape=rectangle, color=white, text=black] (title) at (0,-1.15) {$\bar{K}_{3,1,1,1}$};
\draw (1) -- (2); \draw (1) -- (3); \draw (2) -- (3); 
\end{tikzpicture}}}

\def\KTWOTWOTWO{\fbox{\begin{tikzpicture}[scale=2]
\node (1) at (-1, 0) {1};
\node (2) at (-0.5, -0.866) {2};
\node (3) at (0.5, -0.866) {3};
\node (4) at (1, 0) {4};
\node (5) at (0.5, 0.866) {5};
\node (6) at (-0.5, 0.866) {6};
\node[shape=rectangle, color=white, text=black] (title) at (0,-1.25) {$K_{2,2,2}$};
\draw (1) -- (3); \draw (1) -- (4); \draw (1) -- (5); \draw (1) -- (6); \draw (2) -- (3); \draw (2) -- (4); \draw (2) -- (5); \draw (2) -- (6); \draw (3) -- (5); \draw (3) -- (6); \draw (4) -- (5); \draw (4) -- (6); 
\end{tikzpicture}}}

\def\KTWOTWOTWObar{\fbox{\begin{tikzpicture}[scale=2]
\node (1) at (-1, 0) {1};
\node (2) at (-0.5, -0.866) {2};
\node (3) at (0.5, -0.866) {3};
\node (4) at (1, 0) {4};
\node (5) at (0.5, 0.866) {5};
\node (6) at (-0.5, 0.866) {6};
\node[shape=rectangle, color=white, text=black] (title) at (0,-1.25) {$\bar{K}_{2,2,2}$};
\draw (1) -- (2); \draw (3) -- (4); \draw (5) -- (6); 
\end{tikzpicture}}}

\def\KTWOTWOONEONE{\fbox{\begin{tikzpicture}[scale=2]
\node (1) at (-1, 0) {1};
\node (2) at (-0.5, -0.866) {2};
\node (3) at (0.5, -0.866) {3};
\node (4) at (1, 0) {4};
\node (5) at (0.5, 0.866) {5};
\node (6) at (-0.5, 0.866) {6};
\node[shape=rectangle, color=white, text=black] (title) at (0,-1.25) {$K_{2,2,1,1}$};
\draw (1) -- (3); \draw (1) -- (4); \draw (1) -- (5); \draw (1) -- (6); \draw (2) -- (3); \draw (2) -- (4); \draw (2) -- (5); \draw (2) -- (6); \draw (3) -- (5); \draw (3) -- (6); \draw (4) -- (5); \draw (4) -- (6); \draw (5) -- (6); 
\end{tikzpicture}}}

\def\KTWOTWOONEONEbar{\fbox{\begin{tikzpicture}[scale=2]
\node (1) at (-1, 0) {1};
\node (2) at (-0.5, -0.866) {2};
\node (3) at (0.5, -0.866) {3};
\node (4) at (1, 0) {4};
\node (5) at (0.5, 0.866) {5};
\node (6) at (-0.5, 0.866) {6};
\node[shape=rectangle, color=white, text=black] (title) at (0,-1.25) {$\bar{K}_{2,2,1,1}$};
\draw (1) -- (2); \draw (3) -- (4); 
\end{tikzpicture}}}

\def\KTWOONEONEONEONE{\fbox{\begin{tikzpicture}[scale=2]
\node (1) at (-1, 0) {1};
\node (2) at (-0.5, -0.866) {2};
\node (3) at (0.5, -0.866) {3};
\node (4) at (1, 0) {4};
\node (5) at (0.5, 0.866) {5};
\node (6) at (-0.5, 0.866) {6};
\node[shape=rectangle, color=white, text=black] (title) at (0,-1.25) {$K_{2,1,1,1,1}$};
\draw (1) -- (3); \draw (1) -- (4); \draw (1) -- (5); \draw (1) -- (6); \draw (2) -- (3); \draw (2) -- (4); \draw (2) -- (5); \draw (2) -- (6); \draw (3) -- (4); \draw (3) -- (5); \draw (3) -- (6); \draw (4) -- (5); \draw (4) -- (6); \draw (5) -- (6); 
\end{tikzpicture}}}

\def\KTWOONEONEONEONEbar{\fbox{\begin{tikzpicture}[scale=2]
\node (1) at (-1, 0) {1};
\node (2) at (-0.5, -0.866) {2};
\node (3) at (0.5, -0.866) {3};
\node (4) at (1, 0) {4};
\node (5) at (0.5, 0.866) {5};
\node (6) at (-0.5, 0.866) {6};
\node[shape=rectangle, color=white, text=black] (title) at (0,-1.25) {$\bar{K}_{2,1,1,1,1}$};
\draw (1) -- (2);
\end{tikzpicture}}}

\def\KSIX{\fbox{\begin{tikzpicture}[scale=2]
\node (1) at (-1, 0) {1};
\node (2) at (-0.5, -0.866) {2};
\node (3) at (0.5, -0.866) {3};
\node (4) at (1, 0) {4};
\node (5) at (0.5, 0.866) {5};
\node (6) at (-0.5, 0.866) {6};
\node[shape=rectangle, color=white, text=black] (title) at (0,-1.25) {$K_{6}$};
\draw (1) -- (2); \draw (1) -- (3); \draw (1) -- (4); \draw (1) -- (5); \draw (1) -- (6); \draw (2) -- (3); \draw (2) -- (4); \draw (2) -- (5); \draw (2) -- (6); \draw (3) -- (4); \draw (3) -- (5); \draw (3) -- (6); \draw (4) -- (5); \draw (4) -- (6); \draw (5) -- (6); 
\end{tikzpicture}}}

\def\KSIXbar{\fbox{\begin{tikzpicture}[scale=2]
\node (1) at (-1, 0) {1};
\node (2) at (-0.5, -0.866) {2};
\node (3) at (0.5, -0.866) {3};
\node (4) at (1, 0) {4};
\node (5) at (0.5, 0.866) {5};
\node (6) at (-0.5, 0.866) {6};
\node[shape=rectangle, color=white, text=black] (title) at (0,-1.25) {$\bar{K}_{6}$};
\end{tikzpicture}}}

\def\KSIXONE{\fbox{\begin{tikzpicture}[scale=2]
\node (1) at (-1, 0) {1};
\node (2) at (-0.5, -0.866) {2};
\node (3) at (0.5, -0.866) {3};
\node (4) at (1, 0) {4};
\node (5) at (0.5, 0.866) {5};
\node (6) at (-0.5, 0.866) {6};
\node (7) at (0, 0) {7};
\node[shape=rectangle, color=white, text=black] (title) at (0,-1.25) {$K_{6,1}$};
\draw (1) -- (7); \draw (2) -- (7); \draw (3) -- (7); \draw (4) -- (7); \draw (5) -- (7); \draw (6) -- (7);
\end{tikzpicture}}}

\def\KSIXONEbar{\fbox{\begin{tikzpicture}[scale=2]
\node (1) at (-1, 0) {1};
\node (2) at (-0.5, -0.866) {2};
\node (3) at (0.5, -0.866) {3};
\node (4) at (1, 0) {4};
\node (5) at (0.5, 0.866) {5};
\node (6) at (-0.5, 0.866) {6};
\node (7) at (1.2, 0.725) {7};
\node[shape=rectangle, color=white, text=black] (title) at (0,-1.25) {$\bar{K}_{6,1}$};
\draw (1) -- (2); \draw (1) -- (3); \draw (1) -- (4); \draw (1) -- (5); \draw (1) -- (6); \draw (2) -- (3); \draw (2) -- (4); \draw (2) -- (5); \draw (2) -- (6); \draw (3) -- (4); \draw (3) -- (5); \draw (3) -- (6); \draw (4) -- (5); \draw (4) -- (6); \draw (5) -- (6); 
\end{tikzpicture}}}

\def\KFIVETWO{\fbox{\begin{tikzpicture}[scale=2.5]
\node (1) at (-0.02, 0.995) {1};
\node (2) at (0.89, 0.262) {2};
\node (3) at (-0.682, -0.787) {3};
\node (4) at (0.485, -0.835) {4};
\node (5) at (-1.0, 0.341) {5};
\node (6) at (-0.2, 0.3) {6};
\node (7) at (0.4, -0.2) {7};
\node[shape=rectangle, color=white, text=black] (title) at (0,-1.2) {$K_{5,2}$};
\draw (1) -- (6); \draw (1) -- (7); \draw (2) -- (6); \draw (2) -- (7); \draw (3) -- (6); \draw (3) -- (7); \draw (4) -- (6); \draw (4) -- (7); \draw (5) -- (6); \draw (5) -- (7);
\end{tikzpicture}}}

\def\KFIVETWObar{\fbox{\begin{tikzpicture}[scale=2.5]
\node (1) at (-0.02, 0.995) {1};
\node (2) at (0.89, 0.262) {2};
\node (3) at (-0.682, -0.787) {3};
\node (4) at (0.485, -0.835) {4};
\node (5) at (-1.0, 0.341) {5};
\node (6) at (-0.02, 0) {6};
\node (7) at (0.9, -0.4) {7};
\node[shape=rectangle, color=white, text=black] (title) at (0,-1.2) {$\bar{K}_{5,2}$};
\draw (1) -- (2); \draw (1) -- (3); \draw (1) -- (4); \draw (1) -- (5); \draw (2) -- (3); \draw (2) -- (4); \draw (2) -- (5); \draw (3) -- (4); \draw (3) -- (5); \draw (4) -- (5); \draw (6) -- (7); 
\end{tikzpicture}}}

\def\KFIVEONEONE{\fbox{\begin{tikzpicture}[scale=2.5]
\node (1) at (-0.02, 0.995) {1};
\node (2) at (0.89, 0.262) {2};
\node (3) at (-0.682, -0.787) {3};
\node (4) at (0.485, -0.835) {4};
\node (5) at (-1.0, 0.341) {5};
\node (6) at (-0.2, 0.3) {6};
\node (7) at (0.4, -0.2) {7};
\node[shape=rectangle, color=white, text=black] (title) at (0,-1.2) {$K_{5,1,1}$};
\draw (1) -- (6); \draw (1) -- (7); \draw (2) -- (6); \draw (2) -- (7); \draw (3) -- (6); \draw (3) -- (7); \draw (4) -- (6); \draw (4) -- (7); \draw (5) -- (6); \draw (5) -- (7); \draw (6) -- (7); 
\end{tikzpicture}}}

\def\KFIVEONEONEbar{\fbox{\begin{tikzpicture}[scale=2.5]
\node (1) at (-0.02, 0.995) {1};
\node (2) at (0.89, 0.262) {2};
\node (3) at (-0.682, -0.787) {3};
\node (4) at (0.485, -0.835) {4};
\node (5) at (-1.0, 0.341) {5};
\node (6) at (-0.02, 0) {6};
\node (7) at (0.9, -0.4) {7};
\node[shape=rectangle, color=white, text=black] (title) at (0,-1.2) {$\bar{K}_{5,1,1}$};
\draw (1) -- (2); \draw (1) -- (3); \draw (1) -- (4); \draw (1) -- (5); \draw (2) -- (3); \draw (2) -- (4); \draw (2) -- (5); \draw (3) -- (4); \draw (3) -- (5); \draw (4) -- (5); 
\end{tikzpicture}}}

\def\KFOURTHREE{\fbox{\begin{tikzpicture}[scale=2.5]
\node (1) at (-1.0, 0.235) {1};
\node (2) at (0.88, -0.178) {2};
\node (3) at (0.299, -0.928) {3};
\node (4) at (-0.457, 0.838) {4};
\node (5) at (0.1, -0.1) {5};
\node (6) at (0.25, 0.54) {6};
\node (7) at (-0.6, -0.4) {7};
\node[shape=rectangle, color=white, text=black] (title) at (0,-1.25) {$K_{4,3}$};
\draw (1) -- (5); \draw (1) -- (6); \draw (1) -- (7); \draw (2) -- (5); \draw (2) -- (6); \draw (2) -- (7); \draw (3) -- (5); \draw (3) -- (6); \draw (3) -- (7); \draw (4) -- (5); \draw (4) -- (6); \draw (4) -- (7); 
\end{tikzpicture}}}

\def\KFOURTHREEbar{\fbox{\begin{tikzpicture}[scale=2.5]
\node (1) at (-1.0, 0.235) {1};
\node (2) at (0.88, -0.178) {2};
\node (3) at (0.299, -0.928) {3};
\node (4) at (-0.457, 0.838) {4};
\node (5) at (0.2, -0.225) {5};
\node (6) at (0.25, 0.54) {6};
\node (7) at (-0.6, -0.4) {7};
\node[shape=rectangle, color=white, text=black] (title) at (0,-1.25) {$\bar{K}_{4,3}$};
\draw (1) -- (2); \draw (1) -- (3); \draw (1) -- (4); \draw (2) -- (3); \draw (2) -- (4); \draw (3) -- (4); \draw (5) -- (6); \draw (5) -- (7); \draw (6) -- (7); 
\end{tikzpicture}}}

\def\KFOURTWOONE{\fbox{\begin{tikzpicture}[scale=2.5]
\node (1) at (-1.0, 0.235) {1};
\node (2) at (0.88, -0.178) {2};
\node (3) at (0.299, -0.928) {3};
\node (4) at (-0.457, 0.838) {4};
\node (5) at (0.1, -0.1) {5};
\node (6) at (0.25, 0.54) {6};
\node (7) at (-0.6, -0.4) {7};
\node[shape=rectangle, color=white, text=black] (title) at (0,-1.25) {$K_{4,2,1}$};
\draw (1) -- (5); \draw (1) -- (6); \draw (1) -- (7); \draw (2) -- (5); \draw (2) -- (6); \draw (2) -- (7); \draw (3) -- (5); \draw (3) -- (6); \draw (3) -- (7); \draw (4) -- (5); \draw (4) -- (6); \draw (4) -- (7); \draw (5) -- (7); \draw (6) -- (7); 
\end{tikzpicture}}}

\def\KFOURTWOONEbar{\fbox{\begin{tikzpicture}[scale=2.5]
\node (1) at (-1.0, 0.235) {1};
\node (2) at (0.88, -0.178) {2};
\node (3) at (0.299, -0.928) {3};
\node (4) at (-0.457, 0.838) {4};
\node (5) at (0.2, -0.225) {5};
\node (6) at (0.25, 0.54) {6};
\node (7) at (-0.6, -0.4) {7};
\node[shape=rectangle, color=white, text=black] (title) at (0,-1.25) {$\bar{K}_{4,2,1}$};
\draw (1) -- (2); \draw (1) -- (3); \draw (1) -- (4); \draw (2) -- (3); \draw (2) -- (4); \draw (3) -- (4); \draw (5) -- (6); 
\end{tikzpicture}}}

\def\KFOURONEONEONE{\fbox{\begin{tikzpicture}[scale=2.5]
\node (1) at (-1.0, 0.235) {1};
\node (2) at (0.88, -0.178) {2};
\node (3) at (0.299, -0.928) {3};
\node (4) at (-0.457, 0.838) {4};
\node (5) at (0.1, -0.1) {5};
\node (6) at (0.25, 0.54) {6};
\node (7) at (-0.6, -0.4) {7};
\node[shape=rectangle, color=white, text=black] (title) at (0,-1.25) {$K_{4,1,1,1}$};
\draw (1) -- (5); \draw (1) -- (6); \draw (1) -- (7); \draw (2) -- (5); \draw (2) -- (6); \draw (2) -- (7); \draw (3) -- (5); \draw (3) -- (6); \draw (3) -- (7); \draw (4) -- (5); \draw (4) -- (6); \draw (4) -- (7); \draw (5) -- (6); \draw (5) -- (7); \draw (6) -- (7); 
\end{tikzpicture}}}

\def\KFOURONEONEONEbar{\fbox{\begin{tikzpicture}[scale=2.5]
\node (1) at (-1.0, 0.235) {1};
\node (2) at (0.88, -0.178) {2};
\node (3) at (0.299, -0.928) {3};
\node (4) at (-0.457, 0.838) {4};
\node (5) at (0.2, -0.225) {5};
\node (6) at (0.25, 0.54) {6};
\node (7) at (-0.6, -0.4) {7};
\node[shape=rectangle, color=white, text=black] (title) at (0,-1.25) {$\bar{K}_{4,1,1,1}$};
\draw (1) -- (2); \draw (1) -- (3); \draw (1) -- (4); \draw (2) -- (3); \draw (2) -- (4); \draw (3) -- (4); 
\end{tikzpicture}}}

\def\KTHREETHREEONE{\fbox{\begin{tikzpicture}[scale=2.5]
\node (1) at (0.0, 1.0) {1};
\node (2) at (-0.782, 0.623) {2};
\node (3) at (-0.975, -0.223) {3};
\node (4) at (-0.434, -0.901) {4};
\node (5) at (0.434, -0.901) {5};
\node (6) at (0.975, -0.223) {6};
\node (7) at (0.782, 0.623) {7};
\node[shape=rectangle, color=white, text=black] (title) at (0,-1.2) {$K_{3,3,1}$};
\draw (1) -- (4); \draw (1) -- (5); \draw (1) -- (6); \draw (1) -- (7); \draw (2) -- (4); \draw (2) -- (5); \draw (2) -- (6); \draw (2) -- (7); \draw (3) -- (4); \draw (3) -- (5); \draw (3) -- (6); \draw (3) -- (7); \draw (4) -- (7); \draw (5) -- (7); \draw (6) -- (7); 
\end{tikzpicture}}}

\def\KTHREETHREEONEbar{\fbox{\begin{tikzpicture}[scale=2.5]
\node (1) at (0.0, 1.0) {1};
\node (2) at (-0.782, 0.623) {2};
\node (3) at (-0.975, -0.223) {3};
\node (4) at (-0.434, -0.901) {4};
\node (5) at (0.434, -0.901) {5};
\node (6) at (0.975, -0.223) {6};
\node (7) at (0.782, 0.623) {7};
\node[shape=rectangle, color=white, text=black] (title) at (0,-1.2) {$\bar{K}_{3,3,1}$};
\draw (1) -- (2); \draw (1) -- (3); \draw (2) -- (3); \draw (4) -- (5); \draw (4) -- (6); \draw (5) -- (6); 
\end{tikzpicture}}}

\def\KTHREETWOTWO{\fbox{\begin{tikzpicture}[scale=2.5]
\node (1) at (0.0, 1.0) {1};
\node (2) at (-0.782, 0.623) {2};
\node (3) at (-0.975, -0.223) {3};
\node (4) at (-0.434, -0.901) {4};
\node (5) at (0.434, -0.901) {5};
\node (6) at (0.975, -0.223) {6};
\node (7) at (0.782, 0.623) {7};
\node[shape=rectangle, color=white, text=black] (title) at (0,-1.2) {$K_{3,2,2}$};
\draw (1) -- (4); \draw (1) -- (5); \draw (1) -- (6); \draw (1) -- (7); \draw (2) -- (4); \draw (2) -- (5); \draw (2) -- (6); \draw (2) -- (7); \draw (3) -- (4); \draw (3) -- (5); \draw (3) -- (6); \draw (3) -- (7); \draw (4) -- (6); \draw (4) -- (7); \draw (5) -- (6); \draw (5) -- (7); 
\end{tikzpicture}}}

\def\KTHREETWOTWObar{\fbox{\begin{tikzpicture}[scale=2.5]
\node (1) at (0.0, 1.0) {1};
\node (2) at (-0.782, 0.623) {2};
\node (3) at (-0.975, -0.223) {3};
\node (4) at (-0.434, -0.901) {4};
\node (5) at (0.434, -0.901) {5};
\node (6) at (0.975, -0.223) {6};
\node (7) at (0.782, 0.623) {7};
\node[shape=rectangle, color=white, text=black] (title) at (0,-1.2) {$\bar{K}_{3,2,2}$};
\draw (1) -- (2); \draw (1) -- (3); \draw (2) -- (3); \draw (4) -- (5); \draw (6) -- (7); 
\end{tikzpicture}}}

\def\KTHREETWOONEONE{\fbox{\begin{tikzpicture}[scale=2.5]
\node (1) at (0.0, 1.0) {1};
\node (2) at (-0.782, 0.623) {2};
\node (3) at (-0.975, -0.223) {3};
\node (4) at (-0.434, -0.901) {4};
\node (5) at (0.434, -0.901) {5};
\node (6) at (0.975, -0.223) {6};
\node (7) at (0.782, 0.623) {7};
\node[shape=rectangle, color=white, text=black] (title) at (0,-1.2) {$K_{3,2,1,1}$};
\draw (1) -- (4); \draw (1) -- (5); \draw (1) -- (6); \draw (1) -- (7); \draw (2) -- (4); \draw (2) -- (5); \draw (2) -- (6); \draw (2) -- (7); \draw (3) -- (4); \draw (3) -- (5); \draw (3) -- (6); \draw (3) -- (7); \draw (4) -- (6); \draw (4) -- (7); \draw (5) -- (6); \draw (5) -- (7); \draw (6) -- (7); 
\end{tikzpicture}}}

\def\KTHREETWOONEONEbar{\fbox{\begin{tikzpicture}[scale=2.5]
\node (1) at (0.0, 1.0) {1};
\node (2) at (-0.782, 0.623) {2};
\node (3) at (-0.975, -0.223) {3};
\node (4) at (-0.434, -0.901) {4};
\node (5) at (0.434, -0.901) {5};
\node (6) at (0.975, -0.223) {6};
\node (7) at (0.782, 0.623) {7};
\node[shape=rectangle, color=white, text=black] (title) at (0,-1.2) {$\bar{K}_{3,2,1,1}$};
\draw (1) -- (2); \draw (1) -- (3); \draw (2) -- (3); \draw (4) -- (5); 
\end{tikzpicture}}}

\def\KTHREEONEONEONEONE{\fbox{\begin{tikzpicture}[scale=2.5]
\node (1) at (0.0, 1.0) {1};
\node (2) at (-0.782, 0.623) {2};
\node (3) at (-0.975, -0.223) {3};
\node (4) at (-0.434, -0.901) {4};
\node (5) at (0.434, -0.901) {5};
\node (6) at (0.975, -0.223) {6};
\node (7) at (0.782, 0.623) {7};
\node[shape=rectangle, color=white, text=black] (title) at (0,-1.2) {$K_{3,1,1,1,1}$};
\draw (1) -- (4); \draw (1) -- (5); \draw (1) -- (6); \draw (1) -- (7); \draw (2) -- (4); \draw (2) -- (5); \draw (2) -- (6); \draw (2) -- (7); \draw (3) -- (4); \draw (3) -- (5); \draw (3) -- (6); \draw (3) -- (7); \draw (4) -- (5); \draw (4) -- (6); \draw (4) -- (7); \draw (5) -- (6); \draw (5) -- (7); \draw (6) -- (7);
\end{tikzpicture}}}

\def\KTHREEONEONEONEONEbar{\fbox{\begin{tikzpicture}[scale=2.5]
\node (1) at (0.0, 1.0) {1};
\node (2) at (-0.782, 0.623) {2};
\node (3) at (-0.975, -0.223) {3};
\node (4) at (-0.434, -0.901) {4};
\node (5) at (0.434, -0.901) {5};
\node (6) at (0.975, -0.223) {6};
\node (7) at (0.782, 0.623) {7};
\node[shape=rectangle, color=white, text=black] (title) at (0,-1.2) {$\bar{K}_{3,1,1,1,1}$};
\draw (1) -- (2); \draw (1) -- (3); \draw (2) -- (3); 
\end{tikzpicture}}}

\def\KTWOTWOTWOONE{\fbox{\begin{tikzpicture}[scale=2.5]
\node (1) at (0.0, 1.0) {1};
\node (2) at (-0.782, 0.623) {2};
\node (3) at (-0.975, -0.223) {3};
\node (4) at (-0.434, -0.901) {4};
\node (5) at (0.434, -0.901) {5};
\node (6) at (0.975, -0.223) {6};
\node (7) at (0.782, 0.623) {7};
\node[shape=rectangle, color=white, text=black] (title) at (0,-1.2) {$K_{2,2,2,1}$};
\draw (1) -- (3); \draw (1) -- (4); \draw (1) -- (5); \draw (1) -- (6); \draw (1) -- (7); \draw (2) -- (3); \draw (2) -- (4); \draw (2) -- (5); \draw (2) -- (6); \draw (2) -- (7); \draw (3) -- (5); \draw (3) -- (6); \draw (3) -- (7); \draw (4) -- (5); \draw (4) -- (6); \draw (4) -- (7); \draw (5) -- (7); \draw (6) -- (7); 
\end{tikzpicture}}}

\def\KTWOTWOTWOONEbar{\fbox{\begin{tikzpicture}[scale=2.5]
\node (1) at (0.0, 1.0) {1};
\node (2) at (-0.782, 0.623) {2};
\node (3) at (-0.975, -0.223) {3};
\node (4) at (-0.434, -0.901) {4};
\node (5) at (0.434, -0.901) {5};
\node (6) at (0.975, -0.223) {6};
\node (7) at (0.782, 0.623) {7};
\node[shape=rectangle, color=white, text=black] (title) at (0,-1.2) {$\bar{K}_{2,2,2,1}$};
\draw (1) -- (2); \draw (3) -- (4); \draw (5) -- (6); 
\end{tikzpicture}}}

\def\KTWOTWOONEONEONE{\fbox{\begin{tikzpicture}[scale=2.5]
\node (1) at (0.0, 1.0) {1};
\node (2) at (-0.782, 0.623) {2};
\node (3) at (-0.975, -0.223) {3};
\node (4) at (-0.434, -0.901) {4};
\node (5) at (0.434, -0.901) {5};
\node (6) at (0.975, -0.223) {6};
\node (7) at (0.782, 0.623) {7};
\node[shape=rectangle, color=white, text=black] (title) at (0,-1.2) {$K_{2,2,1,1,1}$};
\draw (1) -- (3); \draw (1) -- (4); \draw (1) -- (5); \draw (1) -- (6); \draw (1) -- (7); \draw (2) -- (3); \draw (2) -- (4); \draw (2) -- (5); \draw (2) -- (6); \draw (2) -- (7); \draw (3) -- (5); \draw (3) -- (6); \draw (3) -- (7); \draw (4) -- (5); \draw (4) -- (6); \draw (4) -- (7); \draw (5) -- (6); \draw (5) -- (7); \draw (6) -- (7); 
\end{tikzpicture}}}

\def\KTWOTWOONEONEONEbar{\fbox{\begin{tikzpicture}[scale=2.5]
\node (1) at (0.0, 1.0) {1};
\node (2) at (-0.782, 0.623) {2};
\node (3) at (-0.975, -0.223) {3};
\node (4) at (-0.434, -0.901) {4};
\node (5) at (0.434, -0.901) {5};
\node (6) at (0.975, -0.223) {6};
\node (7) at (0.782, 0.623) {7};
\node[shape=rectangle, color=white, text=black] (title) at (0,-1.2) {$\bar{K}_{2,2,1,1,1}$};
\draw (1) -- (2); \draw (3) -- (4); 
\end{tikzpicture}}}

\def\KTWOONEONEONEONEONE{\fbox{\begin{tikzpicture}[scale=2.5]
\node (1) at (0.0, 1.0) {1};
\node (2) at (-0.782, 0.623) {2};
\node (3) at (-0.975, -0.223) {3};
\node (4) at (-0.434, -0.901) {4};
\node (5) at (0.434, -0.901) {5};
\node (6) at (0.975, -0.223) {6};
\node (7) at (0.782, 0.623) {7};
\node[shape=rectangle, color=white, text=black] (title) at (0,-1.2) {$K_{2,1,1,1,1,1}$};
\draw (1) -- (3); \draw (1) -- (4); \draw (1) -- (5); \draw (1) -- (6); \draw (1) -- (7); \draw (2) -- (3); \draw (2) -- (4); \draw (2) -- (5); \draw (2) -- (6); \draw (2) -- (7); \draw (3) -- (4); \draw (3) -- (5); \draw (3) -- (6); \draw (3) -- (7); \draw (4) -- (5); \draw (4) -- (6); \draw (4) -- (7); \draw (5) -- (6); \draw (5) -- (7); \draw (6) -- (7); 
\end{tikzpicture}}}

\def\KTWOONEONEONEONEONEbar{\fbox{\begin{tikzpicture}[scale=2.5]
\node (1) at (0.0, 1.0) {1};
\node (2) at (-0.782, 0.623) {2};
\node (3) at (-0.975, -0.223) {3};
\node (4) at (-0.434, -0.901) {4};
\node (5) at (0.434, -0.901) {5};
\node (6) at (0.975, -0.223) {6};
\node (7) at (0.782, 0.623) {7};
\node[shape=rectangle, color=white, text=black] (title) at (0,-1.2) {$\bar{K}_{2,1,1,1,1,1}$};
\draw (1) -- (2); 
\end{tikzpicture}}}

\def\KSEVEN{\fbox{\begin{tikzpicture}[scale=2.5]
\node (1) at (0.0, 1.0) {1};
\node (2) at (-0.782, 0.623) {2};
\node (3) at (-0.975, -0.223) {3};
\node (4) at (-0.434, -0.901) {4};
\node (5) at (0.434, -0.901) {5};
\node (6) at (0.975, -0.223) {6};
\node (7) at (0.782, 0.623) {7};
\node[shape=rectangle, color=white, text=black] (title) at (0,-1.2) {$K_{7}$};
\draw (1) -- (2); \draw (1) -- (3); \draw (1) -- (4); \draw (1) -- (5); \draw (1) -- (6); \draw (1) -- (7); \draw (2) -- (3); \draw (2) -- (4); \draw (2) -- (5); \draw (2) -- (6); \draw (2) -- (7); \draw (3) -- (4); \draw (3) -- (5); \draw (3) -- (6); \draw (3) -- (7); \draw (4) -- (5); \draw (4) -- (6); \draw (4) -- (7); \draw (5) -- (6); \draw (5) -- (7); \draw (6) -- (7); 
\end{tikzpicture}}}

\def\KSEVENbar{\fbox{\begin{tikzpicture}[scale=2.5]
\node (1) at (0.0, 1.0) {1};
\node (2) at (-0.782, 0.623) {2};
\node (3) at (-0.975, -0.223) {3};
\node (4) at (-0.434, -0.901) {4};
\node (5) at (0.434, -0.901) {5};
\node (6) at (0.975, -0.223) {6};
\node (7) at (0.782, 0.623) {7};
\node[shape=rectangle, color=white, text=black] (title) at (0,-1.2) {$\bar{K}_{7}$};
\end{tikzpicture}}}

\def\GsevenONE{\fbox{\begin{tikzpicture}[scale=3.3]
\node (1) at (0.69, 0.342) {1};
\node (2) at (0.531, 0.69) {2};
\node (3) at (0.117, 0.245) {3};
\node (4) at (-0.401, 0.279) {4};
\node (5) at (-0.484, -0.231) {5};
\node (6) at (-1.0, -0.566) {6};
\node (7) at (-0.683, -0.759) {7};
\node[shape=rectangle, color=white, text=black] (title) at (0,-0.8) {$G_{1}: \texttt{FxKGW}$};
\draw (1) -- (2); \draw (1) -- (3); \draw (2) -- (3); \draw (3) -- (5); \draw (3) -- (4); \draw (7) -- (6); \draw (7) -- (5); \draw (5) -- (6); \draw (5) -- (4); 
\end{tikzpicture}}}

\def\GsevenONEbar{\fbox{\hspace{-2mm}\begin{tikzpicture}[scale=3.3]
\node (1) at (0.396, -0.036) {1};
\node (2) at (0.042, -0.268) {2};
\node (3) at (-1.0, 0.317) {3};
\node (4) at (0.175, 0.357) {4};
\node (5) at (0.658, -0.676) {5};
\node (6) at (-0.372, -0.077) {6};
\node (7) at (-0.261, 0.382) {7};
\node[shape=rectangle, color=white, text=black] (title) at (0,-0.9) {$\bar{G}_{1}: \texttt{FErv\_}$};
\draw (1) -- (6); \draw (1) -- (7);  \draw (1) -- (5); \draw (2) -- (6); \draw (2) -- (7);  \draw (2) -- (5); \draw (6) -- (4); \draw (3) -- (6);   \draw (7) -- (3); \draw (2) -- (4); \draw (1) -- (4); \draw (4) -- (7);
\end{tikzpicture}}}

\def\GsevenTWO{\fbox{\begin{tikzpicture}[scale=3.65]
\node (1) at (-1.0, -0.564) {1};
\node (2) at (-0.633, -0.05) {2};
\node (3) at (-0.373, -0.365) {3};
\node (5) at (0.052, -0.078) {5};
\node (4) at (-0.166, 0.295) {4};
\node (6) at (0.358, 0.567) {6};
\node (7) at (0.541, 0.195) {7};
\node[shape=rectangle, color=white, text=black] (title) at (0,-0.6) {$G_{2}: \texttt{FzWWw}$};
\draw (1) -- (2); \draw (1) -- (3); \draw (2) -- (3); \draw (2) -- (4); \draw (2) -- (5); \draw (3) -- (4); \draw (3) -- (5); \draw (4) -- (6); \draw (4) -- (7); \draw (5) -- (6); \draw (5) -- (7); \draw (6) -- (7); 
\end{tikzpicture}}}

\def\GsevenTWObar{\fbox{\hspace{-2mm}\begin{tikzpicture}[scale=3]
\node (1) at (0.194, -0.111) {1};
\node (2) at (-0.808, 0.543) {2};
\node (3) at (-1, -0.05) {3};
\node (5) at (0.736, -0.493) {5};
\node (4) at (0.859, -0.113) {4};
\node (7) at (-0.5, -0.2) {7};
\node (6) at (-0.35, 0.35) {6};
\node[shape=rectangle, color=white, text=black] (title) at (0,-0.8) {$\bar{G}_{2}: \texttt{FCff?}$};
\draw (1) -- (6); \draw (1) -- (7); \draw (1) -- (4); \draw (1) -- (5); \draw (2) -- (6); \draw (6) -- (3); \draw (7) -- (2); \draw (3) -- (7); \draw (4) -- (5); 
\end{tikzpicture}}}

\def\GsevenTHREE{\fbox{\begin{tikzpicture}[scale=3.3]
\node (1) at (0.437, 0.754) {1};
\node (2) at (0.722, 0.427) {2};
\node (3) at (0.074, 0.19) {3};
\node (4) at (-0.3, -0.376) {4};
\node (5) at (-0.54, 0.05) {5};
\node (6) at (-0.608, -0.744) {6};
\node (7) at (-1.0, -0.201) {7};
\node[shape=rectangle, color=white, text=black] (title) at (0,-0.8) {$G_{3}: \texttt{FxKWo}$};
\draw (1) -- (2); \draw (1) -- (3); \draw (2) -- (3); \draw (3) -- (4); \draw (3) -- (5); \draw (4) -- (5); \draw (4) -- (6); \draw (4) -- (7); \draw (5) -- (6); \draw (5) -- (7); 
\end{tikzpicture}}}

\def\GsevenTHREEbar{\fbox{\hspace{-1mm}\begin{tikzpicture}[scale=3.5]
\node (1) at (0.052, -0.078) {1};
\node (2) at (-0.166, 0.295) {2};
\node (3) at (-1.0, -0.564) {3};
\node (4) at (0.358, 0.567) {4};
\node (5) at (0.541, 0.195) {5};
\node (6) at (-0.633, -0.05) {6};
\node (7) at (-0.373, -0.365) {7};
\node[shape=rectangle, color=white, text=black] (title) at (-0.1,-0.7) {$\bar{G}_{3}: \texttt{FErfG}$};
\draw (3) -- (6); \draw (7) -- (3); \draw (6) -- (7); \draw (6) -- (1); \draw (6) -- (2); \draw (7) -- (1); \draw (7) -- (2); \draw (1) -- (4); \draw (5) -- (1); \draw (4) -- (2); \draw (2) -- (5);
\end{tikzpicture}}}

\def\GsevenFOUR{\fbox{\begin{tikzpicture}[scale=3.4]
\node (1) at (0.676, -0.544) {1};
\node (2) at (0.21, -0.455) {2};
\node (3) at (0.425, -0.14) {3};
\node (4) at (-0.157, -0.008) {4};
\node (5) at (-0.72, 0.125) {5};
\node (6) at (-0.5, 0.48) {6};
\node (7) at (-1.0, 0.545) {7};
\node[shape=rectangle, color=white, text=black] (title) at (-0.16,-0.8) {$G_{4}: \texttt{FzCWW}$};
\draw (1) -- (2); \draw (1) -- (3); \draw (2) -- (3); \draw (2) -- (4); \draw (3) -- (4); \draw (7) -- (5); \draw (7) -- (6); \draw (5) -- (6); \draw (5) -- (4); \draw (6) -- (4); 
\end{tikzpicture}}}

\def\GsevenFOURbar{\fbox{\hspace{-2mm}\begin{tikzpicture}[scale=3.8]
\node (1) at (-0.233, -0.431) {1};
\node (2) at (-0.572, 0.454) {2};
\node (3) at (-1, -0.35) {3};
\node (4) at (0.45, -0.648) {4};
\node (5) at (-0.132, 0.577) {5};
\node (6) at (-1.0, 0.221) {6};
\node (7) at (0.024, 0.0) {7};
\node[shape=rectangle, color=white, text=black] (title) at (-0.2,-0.7) {$\bar{G}_{4}: \texttt{FCzf\_}$};
\draw (1) -- (7); \draw (7) -- (3); \draw (7) -- (2); \draw (4) -- (7); \draw (1) -- (5); \draw (1) -- (6); \draw (1) -- (4); \draw (3) -- (5); \draw (3) -- (6); \draw (2) -- (5); \draw (2) -- (6); 
\end{tikzpicture}}}

\def\GeightFIVE{\fbox{\begin{tikzpicture}[scale=3.5]
\node (1) at (0.605, -0.356) {1};
\node (2) at (0.352, -0.704) {2};
\node (3) at (-1.0, 0.194) {3};
\node (4) at (-0.84, 0.532) {4};
\node (5) at (-0.549, 0.167) {5};
\node (6) at (-0.733, -0.155) {6};
\node (7) at (-0.399, 0.475) {7};
\node (8) at (-0.067, -0.153) {8};
\node[shape=rectangle, color=white, text=black] (title) at (-0.35,-1.15) {$G_{5}: \texttt{G\_Kpe[}$};
\draw (1) -- (2); \draw (1) -- (8); \draw (2) -- (8); \draw (3) -- (5); \draw (3) -- (6); \draw (3) -- (7); \draw (4) -- (5); \draw (4) -- (6); \draw (4) -- (7); \draw (5) -- (8); \draw (6) -- (8); \draw (7) -- (8); 
\end{tikzpicture}}}

\def\GeightFIVEbar{\fbox{\begin{tikzpicture}[scale=4]
\node (1) at (-0.416, -0.158) {1};
\node (2) at (-0.698, -0.014) {2};
\node (3) at (-0.091, 0.308) {3};
\node (4) at (-0.394, 0.473) {4};
\node (5) at (-0.881, -0.597) {5};
\node (6) at (-1.0, -0.291) {6};
\node (7) at (-0.572, -0.563) {7};
\node (8) at (0.038, 0.842) {8};
\node[shape=rectangle, color=white, text=black] (title) at (-0.5,-0.8) {$\bar{G}_{5}: \texttt{G\^{}rMX\_}$};
\draw (1) -- (3); \draw (1) -- (4); \draw (1) -- (5); \draw (1) -- (6); \draw (1) -- (7); \draw (2) -- (3); \draw (2) -- (4); \draw (2) -- (5); \draw (2) -- (6); \draw (2) -- (7); \draw (3) -- (4); \draw (3) -- (8); \draw (4) -- (8); \draw (5) -- (6); \draw (5) -- (7); \draw (6) -- (7);
\end{tikzpicture}}}

\def\GeightSIX{\fbox{\begin{tikzpicture}[scale=4]
\node (1) at (-0.045, -0.932) {1};
\node (2) at (0.067, 0.402) {2};
\node (3) at (-1.0, -0.144) {3};
\node (4) at (-0.364, 0.739) {4};
\node (5) at (-0.436, 0.243) {5};
\node (6) at (-0.792, 0.299) {6};
\node (7) at (-0.503, -0.46) {7};
\node (8) at (-0.18, -0.148) {8};
\node[shape=rectangle, color=white, text=black] (title) at (-0.6,-0.9) {$G_{6}: \texttt{G?\textbackslash s\^{}c}$};
\draw (1) -- (7); \draw (1) -- (8); \draw (2) -- (5); \draw (2) -- (6); \draw (2) -- (8); \draw (3) -- (5); \draw (3) -- (6); \draw (3) -- (8); \draw (4) -- (5); \draw (4) -- (6); \draw (4) -- (8); \draw (5) -- (7); \draw (6) -- (7); \draw (7) -- (8); 
\end{tikzpicture}}}

\def\GeightSIXbar{\fbox{\hspace{-2mm}\begin{tikzpicture}[scale=4.4]
\node (1) at (-0.494, -0.039) {1};
\node (2) at (-0.5, 0.5) {2};
\node (3) at (-0.65, 0.253) {3};
\node (4) at (-1.1, 0.25) {4};
\node (5) at (-0.3, -0.6) {5};
\node (6) at (0, -0.3) {6};
\node (7) at (-0.987, 0.688) {7};
\node (8) at (0.14, -0.768) {8};
\node[shape=rectangle, color=white, text=black] (title) at (-0.5,-0.8) {$\bar{G}_{6}: \texttt{G$\sim$aJ\_W}$};
\draw (1) -- (2); \draw (1) -- (3); \draw (1) -- (4); \draw (1) -- (5); \draw (1) -- (6); \draw (2) -- (3); \draw (2) -- (4); \draw (2) -- (7); \draw (3) -- (4); \draw (3) -- (7); \draw (4) -- (7); \draw (5) -- (6); \draw (5) -- (8); \draw (6) -- (8); 
\end{tikzpicture}}}

\def\GeightSEVEN{\fbox{\begin{tikzpicture}[scale=3.6]
\node (1) at (0.57, 0.606) {1};
\node (2) at (0.381, -0.733) {2};
\node (3) at (-0.764, 0.73) {3};
\node (4) at (-1.0, -0.605) {4};
\node (5) at (0.31, -0.038) {5};
\node (6) at (-0.114, 0.508) {6};
\node (7) at (-0.285, -0.508) {7};
\node (8) at (-0.729, 0.04) {8};
\node[shape=rectangle, color=white, text=black] (title) at (-0.3,-0.85) {$G_{7}: \texttt{G?qipk}$};
\draw (1) -- (5); \draw (1) -- (6); \draw (2) -- (5); \draw (2) -- (7); \draw (3) -- (6); \draw (3) -- (8); \draw (4) -- (7); \draw (4) -- (8); \draw (5) -- (6); \draw (5) -- (7); \draw (6) -- (8); \draw (7) -- (8); 
\end{tikzpicture}}}

\def\GeightSEVENbar{\fbox{\hspace{-2mm}\begin{tikzpicture}[scale=3.5]
\node (1) at (0.57, 0.606) {1};
\node (2) at (0.381, -0.733) {2};
\node (3) at (-0.764, 0.73) {3};
\node (4) at (-1.0, -0.605) {4};
\node (5) at (0.31, -0.038) {5};
\node (6) at (-0.114, 0.508) {6};
\node (7) at (-0.285, -0.508) {7};
\node (8) at (-0.729, 0.04) {8};
\node[shape=rectangle, color=white, text=black] (title) at (-0.3,-0.9) {$\bar{G}_{7}: \texttt{G$\sim$LTMO}$};
\draw (1) -- (2); \draw (1) -- (3); \draw (1) -- (4); \draw (1) -- (7); \draw (1) -- (8); \draw (2) -- (3); \draw (2) -- (4); \draw (2) -- (6); \draw (2) -- (8); \draw (3) -- (4); \draw (3) -- (5); \draw (3) -- (7); \draw (4) -- (5); \draw (4) -- (6); \draw (5) -- (8); \draw (6) -- (7);
\end{tikzpicture}}}

\def\GeightEIGHT{\fbox{\begin{tikzpicture}[scale=3.6]
\node (1) at (-1.0, 0.771) {1};
\node (2) at (0.486, 0.135) {2};
\node (3) at (-0.888, -0.864) {3};
\node (4) at (-0.052, -0.138) {4};
\node (5) at (-0.502, 0.261) {5};
\node (6) at (-0.052, 0.775) {6};
\node (7) at (-0.961, -0.212) {7};
\node (8) at (-0.099, -0.728) {8};
\node[shape=rectangle, color=white, text=black] (title) at (-0.3,-1) {$G_{8}: \texttt{G?z\textbackslash bc}$};
\draw (1) -- (5); \draw (1) -- (6); \draw (1) -- (7); \draw (2) -- (5); \draw (2) -- (6); \draw (2) -- (8); \draw (3) -- (5); \draw (3) -- (7); \draw (3) -- (8); \draw (4) -- (6); \draw (4) -- (7); \draw (4) -- (8); \draw (5) -- (6); \draw (7) -- (8); 
\end{tikzpicture}}}

\def\GeightEIGHTbar{\fbox{\begin{tikzpicture}[scale=3.5]
\node (1) at (-0.578, 0.124) {1};
\node (2) at (-0.141, 0.913) {2};
\node (3) at (0.188, 0.439) {3};
\node (4) at (-0.802, 0.688) {4};
\node (5) at (-1.0, -0.379) {5};
\node (6) at (0.489, -0.64) {6};
\node (7) at (0.047, -0.181) {7};
\node (8) at (-0.435, -0.964) {8};
\node[shape=rectangle, color=white, text=black] (title) at (0,-1.2) {$\bar{G}_{8}: \texttt{G$\sim$Ca[W}$};
\draw (1) -- (2); \draw (1) -- (3); \draw (1) -- (4); \draw (1) -- (8); \draw (2) -- (3); \draw (2) -- (4); \draw (2) -- (7); \draw (3) -- (4); \draw (3) -- (6); \draw (4) -- (5); \draw (5) -- (7); \draw (5) -- (8); \draw (6) -- (7); \draw (6) -- (8); 
\end{tikzpicture}}}

\def\GeightNINE{\fbox{\begin{tikzpicture}[scale=3.6]
\node (1) at (0.236, -0.524) {1};
\node (2) at (-1.0, 0.718) {2};
\node (3) at (-0.28, -0.6) {3};
\node (4) at (0.35, -0.05) {4};
\node (5) at (-0.722, 0.179) {5};
\node (6) at (-0.486, 0.425) {6};
\node (7) at (-0.10, 0.055) {7};
\node (8) at (-0.35, -0.16) {8};
\node[shape=rectangle, color=white, text=black] (title) at (-0.3,-0.85) {$G_{9}: \texttt{G\textsf{@}PL|w}$};
\draw (1) -- (7); \draw (1) -- (8); \draw (2) -- (5); \draw (2) -- (6); \draw (3) -- (4); \draw (3) -- (7); \draw (3) -- (8); \draw (4) -- (7); \draw (4) -- (8); \draw (5) -- (6); \draw (5) -- (7); \draw (5) -- (8); \draw (6) -- (7); \draw (6) -- (8); 
\end{tikzpicture}}}

\def\GeightNINEbar{\fbox{\hspace{-2mm}\begin{tikzpicture}[scale=3.4]
\node (1) at (-0.549, 0.167) {1};
\node (2) at (-0.067, -0.153) {2};
\node (3) at (-0.733, -0.155) {3};
\node (4) at (-0.399, 0.475) {4};
\node (5) at (-1.0, 0.194) {5};
\node (6) at (-0.84, 0.532) {6};
\node (7) at (0.605, -0.356) {7};
\node (8) at (0.352, -0.704) {8};
\node[shape=rectangle, color=white, text=black] (title) at (-0.3,-0.9) {$\bar{G}_{9}: \texttt{G\}mqAC}$};
\draw (1) -- (2); \draw (1) -- (3); \draw (1) -- (4); \draw (1) -- (5); \draw (1) -- (6); \draw (2) -- (3); \draw (2) -- (4); \draw (2) -- (7); \draw (2) -- (8); \draw (3) -- (5); \draw (3) -- (6); \draw (4) -- (5); \draw (4) -- (6); \draw (7) -- (8); 
\end{tikzpicture}}}

\def\GeightTEN{\fbox{\begin{tikzpicture}[scale=4]
\node (1) at (0.236, -0.524) {1};
\node (2) at (-1.0, 0.718) {2};
\node (3) at (-0.28, -0.6) {3};
\node (4) at (0.35, -0.05) {4};
\node (7) at (-0.722, 0.179) {7};
\node (8) at (-0.486, 0.425) {8};
\node (5) at (-0.10, 0.055) {5};
\node (6) at (-0.35, -0.16) {6};
\node[shape=rectangle, color=white, text=black] (title) at (-0.3,-0.9) {$G_{10}: \texttt{G\textsf{@}ejQk}$};
\draw (1) -- (5); \draw (1) -- (6); \draw (2) -- (7); \draw (2) -- (8); \draw (3) -- (4); \draw (3) -- (6); \draw (3) -- (7); \draw (4) -- (5); \draw (4) -- (8); \draw (5) -- (6); \draw (5) -- (7); \draw (6) -- (8); \draw (7) -- (8); 
\end{tikzpicture}}}

\def\GeightTENbar{\fbox{\hspace{-2mm}\begin{tikzpicture}[scale=3.2]
\node (1) at (-0.403, 0.196) {1};
\node (2) at (0.129, -0.144) {2};
\node (3) at (-0.512, -0.407) {3};
\node (4) at (0.243, 0.417) {4};
\node (5) at (-0.394, -0.905) {5};
\node (6) at (0.711, 0.503) {6};
\node (7) at (0.078, 0.856) {7};
\node (8) at (-1.0, -0.516) {8};
\node[shape=rectangle, color=white, text=black] (title) at (-0.2,-1.15) {$\bar{G}_{10}: \texttt{G\}XSlO}$};
\draw (1) -- (2); \draw (1) -- (3); \draw (1) -- (4); \draw (1) -- (7); \draw (1) -- (8); \draw (2) -- (3); \draw (2) -- (4); \draw (2) -- (5); \draw (2) -- (6); \draw (3) -- (5); \draw (3) -- (8); \draw (4) -- (6); \draw (4) -- (7); \draw (5) -- (8); \draw (6) -- (7); 
\end{tikzpicture}}}

\def\GeightELEVEN{\fbox{\begin{tikzpicture}[scale=4]
\node (1) at (-1.0, 0.918) {1};
\node (2) at (0.198, -0.262) {2};
\node (3) at (-0.839, -0.824) {3};
\node (4) at (-0.6, 0.48) {4};
\node (5) at (-0.959, 0.379) {5};
\node (6) at (-0.301, 0.096) {6};
\node (7) at (-0.891, -0.224) {7};
\node (8) at (-0.359, -0.485) {8};
\node[shape=rectangle, color=white, text=black] (title) at (-0.3,-0.9) {$G_{11}: \texttt{GCdXrK}$};
\draw (1) -- (4); \draw (1) -- (5); \draw (2) -- (6); \draw (2) -- (8); \draw (3) -- (7); \draw (3) -- (8); \draw (4) -- (5); \draw (4) -- (6); \draw (4) -- (7); \draw (5) -- (6); \draw (5) -- (7); \draw (6) -- (8); \draw (7) -- (8); 
\end{tikzpicture}}}

\def\GeightELEVENbar{\fbox{\begin{tikzpicture}[scale=3.7]
\node (1) at (-0.254, 0.359) {1};
\node (2) at (-0.051, -0.118) {2};
\node (3) at (-0.627, 0.037) {3};
\node (4) at (0.079, -0.767) {4};
\node (5) at (-1.0, -0.536) {5};
\node (6) at (-0.499, 0.98) {6};
\node (7) at (0.241, 0.825) {7};
\node (8) at (-0.484, -0.78) {8};
\node[shape=rectangle, color=white, text=black] (title) at (-0.3,-1) {$\bar{G}_{11}: \texttt{GzYeKo}$};
\draw (1) -- (2); \draw (1) -- (3); \draw (1) -- (6); \draw (1) -- (7); \draw (1) -- (8); \draw (2) -- (3); \draw (2) -- (4); \draw (2) -- (5); \draw (2) -- (7); \draw (3) -- (4); \draw (3) -- (5); \draw (3) -- (6); \draw (4) -- (8); \draw (5) -- (8); \draw (6) -- (7); 
\end{tikzpicture}}}

\def\GeightTWELVE{\fbox{\begin{tikzpicture}[scale=4]
\node (1) at (-0.9, 0) {1};
\node (2) at (0.248, -0.745) {2};
\node (3) at (-0.066, -0.956) {3};
\node (4) at (-0.609, 0.725) {4};
\node (5) at (-0.271, 0.717) {5};
\node (6) at (-0.3, 0.25) {6};
\node (7) at (-0.8, 0.45) {7};
\node (8) at (-0.198, -0.344) {8};
\node[shape=rectangle, color=white, text=black] (title) at (-0.75,-0.9) {$G_{12}: \texttt{GGE[rK}$};
\draw (1) -- (6); \draw (1) -- (7); \draw (2) -- (3); \draw (2) -- (8); \draw (3) -- (8); \draw (4) -- (5); \draw (4) -- (6); \draw (4) -- (7); \draw (5) -- (6); \draw (5) -- (7); \draw (6) -- (8); \draw (7) -- (8); 
\end{tikzpicture}}}

\def\GeightTWELVEbar{\fbox{\begin{tikzpicture}[scale=3.7]
\node (1) at (-0.578, 0.39) {1};
\node (2) at (-0.057, -0.085) {2};
\node (3) at (-0.393, -0.31) {3};
\node (4) at (-0.133, 0.608) {4};
\node (5) at (-1.0, 0.119) {5};
\node (6) at (0.348, -0.67) {6};
\node (7) at (-0.139, -0.912) {7};
\node (8) at (-0.862, 0.86) {8};
\node[shape=rectangle, color=white, text=black] (title) at (-0.75,-0.9) {$\bar{G}_{12}: \texttt{GvxbKo}$};
\draw (1) -- (2); \draw (1) -- (3); \draw (1) -- (4); \draw (1) -- (5); \draw (1) -- (8); \draw (2) -- (4); \draw (2) -- (5); \draw (2) -- (6); \draw (2) -- (7); \draw (3) -- (4); \draw (3) -- (5); \draw (3) -- (6); \draw (3) -- (7); \draw (4) -- (8); \draw (5) -- (8); \draw (6) -- (7); 
\end{tikzpicture}}}

\def\GeightTHIRTEEN{\fbox{\begin{tikzpicture}[scale=3.7]
\node (1) at (-0.669, -0.869) {1};
\node (2) at (-1.0, -0.56) {2};
\node (3) at (-0.282, -0.405) {3};
\node (4) at (-0.622, -0.121) {4};
\node (5) at (0.152, 0.12) {5};
\node (6) at (-0.169, 0.402) {6};
\node (7) at (0.55, 0.555) {7};
\node (8) at (0.256, 0.878) {8};
\node[shape=rectangle, color=white, text=black] (title) at (-0.1,-1) {$G_{13}: \texttt{G\}KoW[}$};
\draw (1) -- (2); \draw (2) -- (3); \draw (2) -- (4); \draw (1) -- (3); \draw (1) -- (4); \draw (3) -- (5); \draw (3) -- (6); \draw (4) -- (5); \draw (4) -- (6); \draw (5) -- (7); \draw (5) -- (8); \draw (6) -- (7); \draw (6) -- (8); \draw (7) -- (8); 
\end{tikzpicture}}}

\def\GeightTHIRTEENbar{\fbox{\hspace{-2mm}\begin{tikzpicture}[scale=3.6]
\node (3) at (-0.669, -0.869) {3};
\node (7) at (-0.282, -0.405) {7};
\node (1) at (0.152, 0.12) {1};
\node (4) at (-1.0, -0.56) {4};
\node (8) at (-0.622, -0.121) {8};
\node (2) at (-0.169, 0.402) {2};
\node (5) at (0.55, 0.555) {5};
\node (6) at (0.256, 0.878) {6};
\node[shape=rectangle, color=white, text=black] (title) at (-0.1,-1) {$\bar{G}_{13}: \texttt{G\textsf{@}rNf\_}$};
\draw (1) -- (5); \draw (1) -- (6); \draw (1) -- (7); \draw (1) -- (8); \draw (2) -- (5); \draw (2) -- (6); \draw (2) -- (7); \draw (2) -- (8); \draw (3) -- (4); \draw (3) -- (7); \draw (3) -- (8); \draw (4) -- (7); \draw (4) -- (8); \draw (5) -- (6); 
\end{tikzpicture}}}

\def\GeightFOURTEEN{\fbox{\begin{tikzpicture}[scale=4.25]
\node (1) at (0.129, -0.794) {1};
\node (2) at (-0.943, 0.65) {2};
\node (3) at (-0.281, -0.583) {3};
\node (4) at (-0.019, -0.373) {4};
\node (5) at (-0.492, 0.538) {5};
\node (6) at (-1.0, 0.207) {6};
\node (7) at (-0.736, 0.376) {7};
\node (8) at (-0.482, -0.021) {8};
\node[shape=rectangle, color=white, text=black] (title) at (-0.55,-0.8) {$G_{14}: \texttt{GTPAX\{}$};
\draw (1) -- (3); \draw (1) -- (4); \draw (2) -- (5); \draw (2) -- (6); \draw (2) -- (7); \draw (3) -- (4); \draw (3) -- (8); \draw (4) -- (8); \draw (5) -- (7); \draw (5) -- (8); \draw (6) -- (7); \draw (6) -- (8); \draw (7) -- (8); 
\end{tikzpicture}}}

\def\GeightFOURTEENbar{\fbox{\hspace{-1mm}\begin{tikzpicture}[scale=4.1]
\node (1) at (-0.186, 0.289) {1};
\node (2) at (0.144, 0.008) {2};
\node (3) at (-0.51, -0.126) {3};
\node (4) at (-0.576, -0.481) {4};
\node (5) at (-1.0, -0.041) {5};
\node (6) at (-0.81, 0.312) {6};
\node (7) at (-0.072, -0.543) {7};
\node (8) at (0.3, 0.5) {8};
\node[shape=rectangle, color=white, text=black] (title) at (-0.6,-0.8) {$\bar{G}_{14}: \texttt{Gim|e?}$};
\draw (1) -- (2); \draw (1) -- (5); \draw (1) -- (6); \draw (1) -- (7); \draw (1) -- (8); \draw (2) -- (3); \draw (2) -- (4); \draw (2) -- (8); \draw (3) -- (5); \draw (3) -- (6); \draw (3) -- (7); \draw (4) -- (5); \draw (4) -- (6); \draw (4) -- (7); \draw (5) -- (6); 
\end{tikzpicture}}}

\def\GeightFIFTEEN{\fbox{\begin{tikzpicture}[scale=3.6]
\node (1) at (-0.794, -0.815) {1};
\node (2) at (0.183, -0.822) {2};
\node (3) at (0.347, 0.822) {3};
\node (4) at (-0.593, 0.822) {4};
\node (5) at (0.058, 0.09) {5};
\node (6) at (0.53, -0.104) {6};
\node (7) at (-1.0, 0.092) {7};
\node (8) at (-0.512, -0.085) {8};
\node[shape=rectangle, color=white, text=black] (title) at (-0.4,-1.1) {$G_{15}: \texttt{G\`{}Xksk}$};
\draw (1) -- (2); \draw (1) -- (7); \draw (1) -- (8); \draw (2) -- (5); \draw (2) -- (6); \draw (3) -- (4); \draw (3) -- (5); \draw (3) -- (6); \draw (4) -- (7); \draw (4) -- (8); \draw (5) -- (6); \draw (5) -- (7); \draw (6) -- (8); \draw (7) -- (8); 
\end{tikzpicture}}}

\def\GeightFIFTEENbar{\fbox{\hspace{-2mm}\begin{tikzpicture}[scale=3.8]
\node (1) at (-0.977, -0.184) {1};
\node (2) at (0.387, 0.212) {2};
\node (3) at (0.048, -0.13) {3};
\node (4) at (-0.556, 0.12) {4};
\node (5) at (-1.0, 0.675) {5};
\node (6) at (-0.51, -0.765) {6};
\node (7) at (0.421, -0.706) {7};
\node (8) at (-0.106, 0.778) {8};
\node[shape=rectangle, color=white, text=black] (title) at (-0.3,-1) {$\bar{G}_{15}: \texttt{G]eRJO}$};
\draw (1) -- (3); \draw (1) -- (4); \draw (1) -- (5); \draw (1) -- (6); \draw (2) -- (3); \draw (2) -- (4); \draw (2) -- (7); \draw (2) -- (8); \draw (3) -- (7); \draw (3) -- (8); \draw (4) -- (5); \draw (4) -- (6); \draw (5) -- (8); \draw (6) -- (7); 
\end{tikzpicture}}}

\def\GeightSIXTEEN{\fbox{\begin{tikzpicture}[scale=3.6]
\node (1) at (0.21, 0.599) {1};
\node (2) at (-0.497, 0.974) {2};
\node (3) at (-0.984, -0.761) {3};
\node (4) at (-0.087, -0.783) {4};
\node (5) at (-0.414, -0.219) {5};
\node (6) at (-1.0, 0.165) {6};
\node (7) at (0.366, -0.189) {7};
\node (8) at (-0.279, 0.215) {8};
\node[shape=rectangle, color=white, text=black] (title) at (-0.4,-1.0) {$G_{16}: \texttt{G\`{}hkqk}$};
\draw (1) -- (2); \draw (1) -- (5); \draw (1) -- (7); \draw (2) -- (6); \draw (2) -- (8); \draw (3) -- (4); \draw (3) -- (5); \draw (3) -- (6); \draw (4) -- (7); \draw (4) -- (8); \draw (5) -- (6); \draw (5) -- (7); \draw (6) -- (8); \draw (7) -- (8); 
\end{tikzpicture}}}

\def\GeightSIXTEENbar{\fbox{\hspace{-1mm}\begin{tikzpicture}[scale=3.8]
\node (1) at (-0.977, -0.184) {1};
\node (2) at (0.387, 0.212) {2};
\node (3) at (0.048, -0.13) {3};
\node (4) at (-0.556, 0.12) {4};
\node (5) at (-0.8, 0.78) {5};
\node (6) at (-0.51, -0.765) {6};
\node (7) at (0.421, -0.706) {7};
\node (8) at (-0.2, 0.778) {8};
\node[shape=rectangle, color=white, text=black] (title) at (-0.4,-1) {$\bar{G}_{16}: \texttt{G]URLO}$};
\draw (1) -- (3); \draw (1) -- (4); \draw (1) -- (6); \draw (1) -- (8); \draw (2) -- (3); \draw (2) -- (4); \draw (2) -- (5); \draw (2) -- (7); \draw (3) -- (7); \draw (3) -- (8); \draw (4) -- (5); \draw (4) -- (6); \draw (5) -- (8); \draw (6) -- (7); 
\end{tikzpicture}}}

\def\GeightSEVENTEEN{\fbox{\begin{tikzpicture}[scale=3.7]
\node (1) at (-0.789, 0.686) {1};
\node (2) at (-0.202, 0.715) {2};
\node (3) at (-1.0, 0.037) {3};
\node (4) at (0.062, -0.918) {4};
\node (5) at (0.36, -0.689) {5};
\node (6) at (-0.118, -0.308) {6};
\node (7) at (-0.658, 0.275) {7};
\node (8) at (-0.334, 0.215) {8};
\node[shape=rectangle, color=white, text=black] (title) at (-0.4,-1.15) {$G_{17}: \texttt{G?C\^{}NK}$};
\draw (1) -- (7); \draw (1) -- (8); \draw (2) -- (7); \draw (2) -- (8); \draw (3) -- (7); \draw (3) -- (8); \draw (4) -- (5); \draw (4) -- (6); \draw (5) -- (6); \draw (6) -- (7); \draw (6) -- (8); \draw (7) -- (8); 
\end{tikzpicture}}}

\def\GeightSEVENTEENbar{\fbox{\begin{tikzpicture}[scale=3.7]
\node (1) at (-1.2, 0.15) {1};
\node (2) at (-0.775, 0.45) {2};
\node (3) at (-0.25, 0.4) {3};
\node (4) at (-0.301, -0.1) {4};
\node (5) at (-0.7, -0.336) {5};
\node (6) at (-1.0, 0.887) {6};
\node (7) at (0.155, -0.588) {7};
\node (8) at (-0.75, -0.922) {8};
\node[shape=rectangle, color=white, text=black] (title) at (-0.4,-1.15) {$\bar{G}_{17}: \texttt{G$\sim$z\_oo}$};
\draw (1) -- (2); \draw (1) -- (3); \draw (1) -- (4); \draw (1) -- (5); \draw (1) -- (6); \draw (2) -- (3); \draw (2) -- (4); \draw (2) -- (5); \draw (2) -- (6); \draw (3) -- (4); \draw (3) -- (5); \draw (3) -- (6); \draw (4) -- (7); \draw (4) -- (8); \draw (5) -- (7); \draw (5) -- (8); 
\end{tikzpicture}}}

\def\GeightEIGHTTEEN{\fbox{\begin{tikzpicture}[scale=3.7]
\node (1) at (0.14, -0.25) {1};
\node (2) at (-0.241, -0.661) {2};
\node (3) at (-0.854, -0.633) {3};
\node (4) at (-1.05, 0.897) {4};
\node (5) at (-0.6, 0.481) {5};
\node (6) at (-0.925, 0.308) {6};
\node (7) at (-0.8, -0.1) {7};
\node (8) at (-0.38, 0.068) {8};
\node[shape=rectangle, color=white, text=black] (title) at (-0.45,-0.95) {$G_{18}: \texttt{G?C\^{}\^{}W}$};
\draw (1) -- (7); \draw (1) -- (8); \draw (2) -- (7); \draw (2) -- (8); \draw (3) -- (7); \draw (3) -- (8); \draw (4) -- (5); \draw (4) -- (6); \draw (5) -- (6); \draw (5) -- (7); \draw (5) -- (8); \draw (6) -- (7); \draw (6) -- (8); 
\end{tikzpicture}}}

\def\GeightEIGHTTEENbar{\fbox{\begin{tikzpicture}[scale=3.7]
\node (1) at (-0.368, 0.294) {1};
\node (2) at (-0.05, 0.55) {2};
\node (3) at (-0.045, 0.123) {3};
\node (4) at (-0.482, -0.291) {4};
\node (5) at (0.295, 0.418) {5};
\node (6) at (-0.342, 0.725) {6};
\node (7) at (-1.0, -0.745) {7};
\node (8) at (-0.64, -0.939) {8};
\node[shape=rectangle, color=white, text=black] (title) at (-0.3,-1.15) {$\bar{G}_{18}: \texttt{G$\sim$z\_\_c}$};
\draw (1) -- (2); \draw (1) -- (3); \draw (1) -- (4); \draw (1) -- (5); \draw (1) -- (6); \draw (2) -- (3); \draw (2) -- (4); \draw (2) -- (5); \draw (2) -- (6); \draw (3) -- (4); \draw (3) -- (5); \draw (3) -- (6); \draw (4) -- (7); \draw (4) -- (8); \draw (7) -- (8); 
\end{tikzpicture}}}

\def\GeightNINETEEN{\fbox{\begin{tikzpicture}[scale=3.7]
\node (1) at (-0.789, 0.686) {1};
\node (2) at (-0.202, 0.715) {2};
\node (3) at (-1.0, 0.037) {3};
\node (4) at (0.062, -0.918) {4};
\node (5) at (0.36, -0.689) {5};
\node (6) at (-0.118, -0.308) {6};
\node (7) at (-0.658, 0.275) {7};
\node (8) at (-0.334, 0.215) {8};
\node[shape=rectangle, color=white, text=black] (title) at (-0.4,-1.15) {$G_{19}: \texttt{G?C\^{}NG}$};
\draw (1) -- (7); \draw (1) -- (8); \draw (2) -- (7); \draw (2) -- (8); \draw (3) -- (7); \draw (3) -- (8); \draw (4) -- (5); \draw (4) -- (6); \draw (5) -- (6); \draw (6) -- (7); \draw (6) -- (8); 
\end{tikzpicture}}}

\def\GeightNINETEENbar{\fbox{\begin{tikzpicture}[scale=3.7]
\node (1) at (-1.2, 0.15) {1};
\node (2) at (-0.775, 0.45) {2};
\node (3) at (-0.25, 0.4) {3};
\node (4) at (-0.301, -0.1) {4};
\node (5) at (-0.7, -0.336) {5};
\node (6) at (-1.0, 0.887) {6};
\node (7) at (0.155, -0.588) {7};
\node (8) at (-0.75, -0.922) {8};
\node[shape=rectangle, color=white, text=black] (title) at (-0.4,-1.15) {$\bar{G}_{19}: \texttt{G$\sim$z\_os}$};
\draw (1) -- (2); \draw (1) -- (3); \draw (1) -- (4); \draw (1) -- (5); \draw (1) -- (6); \draw (2) -- (3); \draw (2) -- (4); \draw (2) -- (5); \draw (2) -- (6); \draw (3) -- (4); \draw (3) -- (5); \draw (3) -- (6); \draw (4) -- (7); \draw (4) -- (8); \draw (5) -- (7); \draw (5) -- (8); \draw (7) -- (8); 
\end{tikzpicture}}}

\def\GeightTWENTY{\fbox{\begin{tikzpicture}[scale=3.7]
\node (1) at (-0.378, 0.296) {1};
\node (2) at (-0.62, 0.446) {2};
\node (3) at (-0.546, -0.196) {3};
\node (4) at (-0.832, -0.19) {4};
\node (5) at (-0.945, 0.211) {5};
\node (6) at (-0.209, 0.914) {6};
\node (7) at (-0.383, -0.691) {7};
\node (8) at (-1.0, -0.791) {8};
\node[shape=rectangle, color=white, text=black] (title) at (-0.4,-1.0) {$G_{20}: \texttt{G\}$\sim$\textsf{@}\`{}\_}$};
\draw (1) -- (2); \draw (1) -- (3); \draw (1) -- (4); \draw (1) -- (5); \draw (1) -- (6); \draw (2) -- (3); \draw (2) -- (4); \draw (2) -- (5); \draw (2) -- (6); \draw (3) -- (5); \draw (3) -- (7); \draw (3) -- (8); \draw (4) -- (5); \draw (4) -- (7); \draw (4) -- (8);
\end{tikzpicture}}}

\def\GeightTWENTYbar{\fbox{\hspace{-2mm}\begin{tikzpicture}[scale=3.7]
\node (1) at (-0.148, -0.794) {1};
\node (2) at (0.458, -0.591) {2};
\node (3) at (-1.0, 0.578) {3};
\node (4) at (-0.701, 0.87) {4};
\node (5) at (0.112, 0.191) {5};
\node (6) at (-0.384, 0.273) {6};
\node (7) at (0.128, -0.25) {7};
\node (8) at (-0.132, -0.277) {8};
\node[shape=rectangle, color=white, text=black] (title) at (-0.4,-1.0) {$\bar{G}_{20}: \texttt{G\textsf{@}?\}][}$};
\draw (1) -- (7); \draw (1) -- (8); \draw (2) -- (7); \draw (2) -- (8); \draw (3) -- (4); \draw (3) -- (6); \draw (4) -- (6); \draw (5) -- (6); \draw (5) -- (7); \draw (5) -- (8); \draw (6) -- (7); \draw (6) -- (8); \draw (7) -- (8); 
\end{tikzpicture}}}

\def\GeightTWENTYONE{\fbox{\begin{tikzpicture}[scale=3.5]
\node (1) at (-1.0, -0.628) {1};
\node (2) at (-0.638, -0.92) {2};
\node (3) at (0.118, 0.908) {3};
\node (4) at (0.496, 0.65) {4};
\node (5) at (-0.162, 0.37) {5};
\node (6) at (-0.307, -0.409) {6};
\node (7) at (-0.591, -0.165) {7};
\node (8) at (0.117, 0.192) {8};
\node[shape=rectangle, color=white, text=black] (title) at (-0.3,-1.15) {$G_{21}: \texttt{G?NMPk}$};
\draw (1) -- (6); \draw (1) -- (7); \draw (2) -- (6); \draw (2) -- (7); \draw (3) -- (5); \draw (3) -- (8); \draw (4) -- (5); \draw (4) -- (8); \draw (5) -- (6); \draw (5) -- (7); \draw (6) -- (8); \draw (7) -- (8); 
\end{tikzpicture}}}

\def\GeightTWENTYONEbar{\fbox{\hspace{-1.5mm}\begin{tikzpicture}[scale=3.5]
\node (1) at (-0.162, 0.37) {1};
\node (2) at (0.117, 0.192) {2};
\node (3) at (-0.307, -0.409) {3};
\node (4) at (-0.591, -0.165) {4};
\node (5) at (0.118, 0.908) {5};
\node (6) at (-0.638, -0.92) {6};
\node (7) at (-1.0, -0.628) {7};
\node (8) at (0.496, 0.65) {8};
\node[shape=rectangle, color=white, text=black] (title) at (-0.25,-1.15) {$\bar{G}_{21}: \texttt{G$\sim$opmO}$};
\draw (1) -- (2); \draw (1) -- (3); \draw (1) -- (4); \draw (1) -- (5); \draw (1) -- (8); \draw (2) -- (3); \draw (2) -- (4); \draw (2) -- (5); \draw (2) -- (8); \draw (3) -- (4); \draw (3) -- (6); \draw (3) -- (7); \draw (4) -- (6); \draw (4) -- (7); \draw (5) -- (8); \draw (6) -- (7);
\end{tikzpicture}}}

\def\GeightTWENTYTWO{\fbox{\begin{tikzpicture}[scale=3.5]
\node (1) at (0.35, 0.696) {1};
\node (2) at (-0.345, 0.988) {2};
\node (3) at (-1.0, -0.726) {3};
\node (4) at (-0.196, -0.951) {4};
\node (5) at (-0.601, -0.206) {5};
\node (6) at (-0.3, -0.43) {6};
\node (7) at (-0.038, 0.269) {7};
\node (8) at (-0.371, 0.359) {8};
\node[shape=rectangle, color=white, text=black] (title) at (-0.45,-1.15) {$G_{22}: \texttt{G?Ku][}$};
\draw (1) -- (7); \draw (1) -- (8); \draw (2) -- (7); \draw (2) -- (8); \draw (3) -- (5); \draw (3) -- (6); \draw (4) -- (5); \draw (4) -- (6); \draw (5) -- (7); \draw (5) -- (8); \draw (6) -- (7); \draw (6) -- (8); \draw (7) -- (8); 
\end{tikzpicture}}}

\def\GeightTWENTYTWObar{\fbox{\hspace{-2mm}\begin{tikzpicture}[scale=3.5]
\node (1) at (-0.569, -0.174) {1};
\node (2) at (-0.307, -0.411) {2};
\node (3) at (-0.174, 0.383) {3};
\node (4) at (0.091, 0.215) {4};
\node (5) at (-0.608, -0.886) {5};
\node (6) at (-1.0, -0.648) {6};
\node (7) at (0.52, 0.545) {7};
\node (8) at (-0.043, 0.976) {8};
\node[shape=rectangle, color=white, text=black] (title) at (-0.25,-1.15) {$\bar{G}_{22}: \texttt{G$\sim$rH\`{}\_}$};
\draw (1) -- (2); \draw (1) -- (3); \draw (1) -- (4); \draw (1) -- (5); \draw (1) -- (6); \draw (2) -- (3); \draw (2) -- (4); \draw (2) -- (5); \draw (2) -- (6); \draw (3) -- (4); \draw (3) -- (7); \draw (3) -- (8); \draw (4) -- (7); \draw (4) -- (8); \draw (5) -- (6); 
\end{tikzpicture}}}

\def\GeightTWENTYTHREE{\fbox{\begin{tikzpicture}[scale=3.5]
\node (1) at (-1.0, 0.946) {1};
\node (2) at (-0.263, -0.731) {2};
\node (3) at (-0.919, -0.783) {3};
\node (4) at (-0.379, 0.19) {4};
\node (5) at (-0.96, 0.298) {5};
\node (6) at (-0.624, 0.469) {6};
\node (7) at (-0.486, -0.169) {7};
\node (8) at (-0.782, -0.219) {8};
\node[shape=rectangle, color=white, text=black] (title) at (-0.45,-1) {$G_{23}: \texttt{G?eRzw}$};
\draw (1) -- (5); \draw (1) -- (6); \draw (2) -- (7); \draw (2) -- (8); \draw (3) -- (7); \draw (3) -- (8); \draw (4) -- (5); \draw (4) -- (6); \draw (4) -- (7); \draw (4) -- (8); \draw (5) -- (7); \draw (5) -- (8); \draw (6) -- (7); \draw (6) -- (8); 
\end{tikzpicture}}}

\def\GeightTWENTYTHREEbar{\fbox{\begin{tikzpicture}[scale=3.5]
\node (1) at (-0.164, 0.272) {1};
\node (2) at (-0.403, -0.372) {2};
\node (3) at (-0.602, -0.172) {3};
\node (4) at (-0.055, -0.18) {4};
\node (5) at (-1.0, -0.532) {5};
\node (6) at (-0.69, -0.739) {6};
\node (7) at (-0.094, 0.943) {7};
\node (8) at (0.286, 0.78) {8};
\node[shape=rectangle, color=white, text=black] (title) at (-0.25,-1) {$\bar{G}_{23}: \texttt{G$\sim$XkCC}$};
\draw (1) -- (2); \draw (1) -- (3); \draw (1) -- (4); \draw (1) -- (7); \draw (1) -- (8); \draw (2) -- (3); \draw (2) -- (4); \draw (2) -- (5); \draw (2) -- (6); \draw (3) -- (4); \draw (3) -- (5); \draw (3) -- (6); \draw (5) -- (6); \draw (7) -- (8); 
\end{tikzpicture}}}

\def\GeightTWENTYFOUR{\fbox{\begin{tikzpicture}[scale=3.4]
\node (1) at (0.658, -0.159) {1};
\node (2) at (-0.293, 1.0) {2};
\node (3) at (-0.642, -0.65) {3};
\node (4) at (-1.0, -0.229) {4};
\node (5) at (-0.531, -0.15) {5};
\node (6) at (-0.051, -0.423) {6};
\node (7) at (-0.681, 0.311) {7};
\node (8) at (-0.018, 0.3) {8};
\node[shape=rectangle, color=white, text=black] (title) at (-0.2,-0.95) {$G_{24}: \texttt{G?Mre[}$};
\draw (1) -- (6); \draw (1) -- (8); \draw (2) -- (7); \draw (2) -- (8); \draw (3) -- (5); \draw (3) -- (6); \draw (3) -- (7); \draw (4) -- (5); \draw (4) -- (6); \draw (4) -- (7); \draw (5) -- (8); \draw (6) -- (8); \draw (7) -- (8); 
\end{tikzpicture}}}

\def\GeightTWENTYFOURbar{\fbox{\begin{tikzpicture}[scale=3.7]
\node (1) at (-0.959, -0.061) {1};
\node (2) at (-0.57, 0.003) {2};
\node (3) at (-1.0, 0.456) {3};
\node (4) at (-0.666, 0.505) {4};
\node (5) at (-0.699, -0.421) {5};
\node (6) at (-0.405, -0.7) {6};
\node (7) at (-0.933, -0.772) {7};
\node (8) at (-0.902, 0.991) {8};
\node[shape=rectangle, color=white, text=black] (title) at (-0.5,-1) {$\bar{G}_{24}: \texttt{G$\sim$pKX\_}$};
\draw (1) -- (2); \draw (1) -- (3); \draw (1) -- (4); \draw (1) -- (5); \draw (1) -- (7); \draw (2) -- (3); \draw (2) -- (4); \draw (2) -- (5); \draw (2) -- (6); \draw (3) -- (4); \draw (3) -- (8); \draw (4) -- (8); \draw (5) -- (6); \draw (5) -- (7); \draw (6) -- (7); 
\end{tikzpicture}}}

\def\GeightTWENTYFIVE{\fbox{\begin{tikzpicture}[scale=3.6]
\node (1) at (-1.0, -0.422) {1};
\node (2) at (-0.408, -0.962) {2};
\node (3) at (0.641, 0.598) {3};
\node (4) at (0.235, 0.856) {4};
\node (5) at (0.267, 0.12) {5};
\node (6) at (-0.101, 0.366) {6};
\node (7) at (-0.233, -0.396) {7};
\node (8) at (-0.464, -0.16) {8};
\node[shape=rectangle, color=white, text=black] (title) at (-0.15,-1.2) {$G_{25}: \texttt{G\textsf{@}Ku][}$};
\draw (1) -- (7); \draw (1) -- (8); \draw (2) -- (7); \draw (2) -- (8); \draw (3) -- (4); \draw (3) -- (5); \draw (3) -- (6); \draw (4) -- (5); \draw (4) -- (6); \draw (5) -- (7); \draw (5) -- (8); \draw (6) -- (7); \draw (6) -- (8); \draw (7) -- (8); 
\end{tikzpicture}}}

\def\GeightTWENTYFIVEbar{\fbox{\begin{tikzpicture}[scale=3.6]
\node (1) at (-0.307, 0.375) {1};
\node (2) at (-0.0, 0.225) {2};
\node (3) at (-0.27, -0.397) {3};
\node (4) at (-0.592, -0.28) {4};
\node (5) at (0.305, 0.783) {5};
\node (6) at (-0.111, 0.949) {6};
\node (7) at (-0.247, -0.963) {7};
\node (8) at (-1.0, -0.692) {8};
\node[shape=rectangle, color=white, text=black] (title) at (-0.3,-1.2) {$\bar{G}_{25}: \texttt{G\}rH\`{}\_}$};
\draw (1) -- (2); \draw (1) -- (3); \draw (1) -- (4); \draw (1) -- (5); \draw (1) -- (6); \draw (2) -- (3); \draw (2) -- (4); \draw (2) -- (5); \draw (2) -- (6); \draw (3) -- (7); \draw (3) -- (8); \draw (4) -- (7); \draw (4) -- (8); \draw (5) -- (6); 
\end{tikzpicture}}}

\def\GeightTWENTYSIX{\fbox{\begin{tikzpicture}[scale=3.5]
\node (1) at (0.12, 0.963) {1};
\node (2) at (0.736, 0.477) {2};
\node (3) at (-0.404, -0.964) {3};
\node (4) at (-1.0, -0.476) {4};
\node (5) at (-0.499, -0.165) {5};
\node (6) at (-0.067, 0.391) {6};
\node (7) at (0.207, 0.154) {7};
\node (8) at (-0.205, -0.38) {8};
\node[shape=rectangle, color=white, text=black] (title) at (-0.15,-1.2) {$G_{26}: \texttt{G?NMX\{}$};
\draw (1) -- (6); \draw (1) -- (7); \draw (2) -- (6); \draw (2) -- (7); \draw (3) -- (5); \draw (3) -- (8); \draw (4) -- (5); \draw (4) -- (8); \draw (5) -- (6); \draw (5) -- (7); \draw (5) -- (8); \draw (6) -- (7); \draw (6) -- (8); \draw (7) -- (8); 
\end{tikzpicture}}}

\def\GeightTWENTYSIXbar{\fbox{\begin{tikzpicture}[scale=3.6]
\node (1) at (-0.448, -0.348) {1};
\node (2) at (-0.135, -0.281) {2};
\node (3) at (-0.372, 0.386) {3};
\node (4) at (-0.649, 0.242) {4};
\node (5) at (-0.522, -0.914) {5};
\node (6) at (-1.0, 0.796) {6};
\node (7) at (-0.282, 0.927) {7};
\node (8) at (0.185, -0.809) {8};
\node[shape=rectangle, color=white, text=black] (title) at (-0.3,-1.2) {$\bar{G}_{26}: \texttt{G$\sim$ope?}$};
\draw (1) -- (2); \draw (1) -- (3); \draw (1) -- (4); \draw (1) -- (5); \draw (1) -- (8); \draw (2) -- (3); \draw (2) -- (4); \draw (2) -- (5); \draw (2) -- (8); \draw (3) -- (4); \draw (3) -- (6); \draw (3) -- (7); \draw (4) -- (6); \draw (4) -- (7); 
\end{tikzpicture}}}

\def\GeightTWENTYSEVEN{\fbox{\begin{tikzpicture}[scale=3.5]
\node (1) at (0.242, -0.87) {1};
\node (2) at (0.54, -0.625) {2};
\node (3) at (-1.0, 0.614) {3};
\node (4) at (-0.712, 0.874) {4};
\node (5) at (-0.398, -0.148) {5};
\node (6) at (-0.04, 0.154) {6};
\node (7) at (0.013, -0.28) {7};
\node (8) at (-0.457, 0.28) {8};
\node[shape=rectangle, color=white, text=black] (title) at (-0.4,-1) {$G_{27}: \texttt{G\`{}?MX\{}$};
\draw (1) -- (2); \draw (1) -- (7); \draw (2) -- (7); \draw (3) -- (4); \draw (3) -- (8); \draw (4) -- (8); \draw (5) -- (6); \draw (5) -- (7); \draw (5) -- (8); \draw (6) -- (7); \draw (6) -- (8); \draw (7) -- (8); 
\end{tikzpicture}}}

\def\GeightTWENTYSEVENbar{\fbox{\begin{tikzpicture}[scale=3.6]
\node (1) at (-0.141, 0.347) {1};
\node (2) at (0.1, 0.092) {2};
\node (3) at (-0.408, -0.35) {3};
\node (4) at (-0.7, -0.128) {4};
\node (5) at (-0.659, 0.344) {5};
\node (6) at (0.1, -0.35) {6};
\node (7) at (-1.0, -0.701) {7};
\node (8) at (0.437, 0.678) {8};
\node[shape=rectangle, color=white, text=black] (title) at (-0.2,-0.95) {$\bar{G}_{27}: \texttt{G]$\sim$pe?}$};
\draw (1) -- (3); \draw (1) -- (4); \draw (1) -- (5); \draw (1) -- (6); \draw (1) -- (8); \draw (2) -- (3); \draw (2) -- (4); \draw (2) -- (5); \draw (2) -- (6); \draw (2) -- (8); \draw (3) -- (5); \draw (3) -- (6); \draw (3) -- (7); \draw (4) -- (5); \draw (4) -- (6); \draw (4) -- (7); 
\end{tikzpicture}}}

\def\GeightTWENTYEIGHT{\fbox{\begin{tikzpicture}[scale=3.8]
\node (1) at (0.566, 0.516) {1};
\node (2) at (0.352, 0.779) {2};
\node (3) at (-0.022, -0.085) {3};
\node (4) at (-0.958, -0.232) {4};
\node (5) at (-1.0, -0.516) {5};
\node (6) at (-0.716, -0.588) {6};
\node (7) at (0.069, 0.299) {7};
\node (8) at (-0.494, -0.174) {8};
\node[shape=rectangle, color=white, text=black] (title) at (-0.25,-.8) {$G_{28}: \texttt{G\_C\^{}\textsf{@}\{}$};
\draw (1) -- (2); \draw (1) -- (7); \draw (2) -- (7); \draw (3) -- (7); \draw (3) -- (8); \draw (4) -- (5); \draw (4) -- (6); \draw (4) -- (8); \draw (5) -- (6); \draw (5) -- (8); \draw (6) -- (8); \draw (7) -- (8); 
\end{tikzpicture}}}

\def\GeightTWENTYEIGHTbar{\fbox{\begin{tikzpicture}[scale=3.8]
\node (1) at (-0.88, 0.264) {1};
\node (2) at (-0.14, 0.158) {2};
\node (3) at (-0.398, -0.193) {3};
\node (4) at (-0.565, 0.26) {4};
\node (5) at (-1.0, -0.182) {5};
\node (6) at (-0.353, -0.579) {6};
\node (7) at (-0.988, -0.706) {7};
\node (8) at (-0.376, 0.977) {8};
\node[shape=rectangle, color=white, text=black] (title) at (-0.3,-0.9) {$\bar{G}_{28}: \texttt{G\^{}z\_\}?}$};
\draw (1) -- (3); \draw (1) -- (4); \draw (1) -- (5); \draw (1) -- (6); \draw (1) -- (8); \draw (2) -- (3); \draw (2) -- (4); \draw (2) -- (5); \draw (2) -- (6); \draw (2) -- (8); \draw (3) -- (4); \draw (3) -- (5); \draw (3) -- (6); \draw (4) -- (7); \draw (5) -- (7); \draw (6) -- (7); 
\end{tikzpicture}}}

\def\GeightTWENTYNINE{\fbox{\begin{tikzpicture}[scale=3.8]
\node (1) at (0.21, 0.15) {1};
\node (2) at (-0.152, 0.967) {2};
\node (3) at (0.176, 0.794) {3};
\node (4) at (-0.818, -0.648) {4};
\node (5) at (-1.0, -0.36) {5};
\node (6) at (-0.221, 0.364) {6};
\node (7) at (-0.444, -0.301) {7};
\node (8) at (-0.65, 0.0) {8};
\node[shape=rectangle, color=white, text=black] (title) at (-0.3,-.8) {$G_{29}: \texttt{GGDc\{\{}$};
\draw (1) -- (7); \draw (1) -- (8); \draw (2) -- (3); \draw (2) -- (6); \draw (3) -- (6); \draw (4) -- (5); \draw (4) -- (7); \draw (4) -- (8); \draw (5) -- (7); \draw (5) -- (8); \draw (6) -- (7); \draw (6) -- (8); \draw (7) -- (8); 
\end{tikzpicture}}}

\def\GeightTWENTYNINEbar{\fbox{\begin{tikzpicture}[scale=3.8]
\node (1) at (-0.015, -0.331) {1};
\node (2) at (-0.472, 0.114) {2};
\node (3) at (-0.199, 0.27) {3};
\node (4) at (-0.364, -0.508) {4};
\node (5) at (0.315, -0.117) {5};
\node (6) at (0.304, -0.852) {6};
\node (7) at (-1.0, 0.518) {7};
\node (8) at (-0.318, 0.905) {8};
\node[shape=rectangle, color=white, text=black] (title) at (-0.5,-0.8) {$\bar{G}_{29}: \texttt{GvyZB?}$};
\draw (1) -- (2); \draw (1) -- (3); \draw (1) -- (4); \draw (1) -- (5); \draw (1) -- (6); \draw (2) -- (4); \draw (2) -- (5); \draw (2) -- (7); \draw (2) -- (8); \draw (3) -- (4); \draw (3) -- (5); \draw (3) -- (7); \draw (3) -- (8); \draw (4) -- (6); \draw (5) -- (6); 
\end{tikzpicture}}}

\def\GeightTHIRTY{\fbox{\begin{tikzpicture}[scale=3.8]
\node (1) at (-0.739, 0.821) {1};
\node (2) at (-1.0, 0.509) {2};
\node (3) at (0.453, -0.587) {3};
\node (4) at (0.114, -0.777) {4};
\node (5) at (0.108, 0.247) {5};
\node (6) at (0.162, -0.153) {6};
\node (7) at (-0.058, -0.314) {7};
\node (8) at (-0.384, 0.254) {8};
\node[shape=rectangle, color=white, text=black] (title) at (-0.55,-0.8) {$G_{30}: \texttt{G\`{}?xu[}$};
\draw (1) -- (2); \draw (1) -- (8); \draw (2) -- (8); \draw (3) -- (4); \draw (3) -- (6); \draw (3) -- (7); \draw (4) -- (6); \draw (4) -- (7); \draw (5) -- (6); \draw (5) -- (7); \draw (5) -- (8); \draw (6) -- (8); \draw (7) -- (8); 
\end{tikzpicture}}}

\def\GeightTHIRTYbar{\fbox{\begin{tikzpicture}[scale=3.8]
\node (1) at (-0.77, 0.088) {1};
\node (2) at (-0.277, 0.277) {2};
\node (3) at (-0.518, -0.473) {3};
\node (4) at (0.15, -0.204) {4};
\node (5) at (-0.25, -0.25) {5};
\node (6) at (-0.663, 0.767) {6};
\node (7) at (-1.0, 0.536) {7};
\node (8) at (0.148, -0.831) {8};
\node[shape=rectangle, color=white, text=black] (title) at (-0.5,-0.85) {$\bar{G}_{30}: \texttt{G]$\sim$EH\_}$};
\draw (1) -- (3); \draw (1) -- (4); \draw (1) -- (5); \draw (1) -- (6); \draw (1) -- (7); \draw (2) -- (3); \draw (2) -- (4); \draw (2) -- (5); \draw (2) -- (6); \draw (2) -- (7); \draw (3) -- (5); \draw (3) -- (8); \draw (4) -- (5); \draw (4) -- (8); \draw (6) -- (7); 
\end{tikzpicture}}}

\def\GeightTHIRTYONE{\fbox{\begin{tikzpicture}[scale=3.7]
\node (1) at (-1.0, -0.628) {1};
\node (2) at (-0.638, -0.92) {2};
\node (3) at (0.118, 0.908) {3};
\node (4) at (0.496, 0.65) {4};
\node (5) at (-0.162, 0.37) {5};
\node (6) at (0.117, 0.192) {6};
\node (7) at (-0.591, -0.165) {7};
\node (8) at (-0.307, -0.409) {8};
\node[shape=rectangle, color=white, text=black] (title) at (0.1,-0.95) {$G_{31}: \texttt{G\textsf{@}Ku]W}$};
\draw (1) -- (7); \draw (1) -- (8); \draw (2) -- (7); \draw (2) -- (8); \draw (3) -- (4); \draw (3) -- (5); \draw (3) -- (6); \draw (4) -- (5); \draw (4) -- (6); \draw (5) -- (7); \draw (5) -- (8); \draw (6) -- (7); \draw (6) -- (8); 
\end{tikzpicture}}}

\def\GeightTHIRTYONEbar{\fbox{\hspace{-2mm}\begin{tikzpicture}[scale=3.6]
\node (1) at (-0.307, -0.409) {1};
\node (2) at (-0.591, -0.165) {2};
\node (3) at (-0.162, 0.37) {3};
\node (4) at (0.117, 0.192) {4};
\node (5) at (-1.0, -0.628) {5};
\node (6) at (-0.638, -0.92) {6};
\node (7) at (0.118, 0.908) {7};
\node (8) at (0.496, 0.65) {8};
\node[shape=rectangle, color=white, text=black] (title) at (0.1,-1.0) {$\bar{G}_{31}: \texttt{G\}rH\`{}c}$};
\draw (1) -- (2); \draw (1) -- (3); \draw (1) -- (4); \draw (1) -- (5); \draw (1) -- (6); \draw (2) -- (3); \draw (2) -- (4); \draw (2) -- (5); \draw (2) -- (6); \draw (3) -- (7); \draw (3) -- (8); \draw (4) -- (7); \draw (4) -- (8); \draw (5) -- (6); \draw (7) -- (8); 
\end{tikzpicture}}}

\def\GeightTHIRTYTWO{\fbox{\begin{tikzpicture}[scale=3.7]
\node (1) at (0.483, -0.614) {1};
\node (2) at (0.072, -0.941) {2};
\node (3) at (-0.106, 0.83) {3};
\node (4) at (-0.463, 0.366) {4};
\node (5) at (-0.819, 0.638) {5};
\node (6) at (-1.0, 0.112) {6};
\node (7) at (-0.294, -0.306) {7};
\node (8) at (0.005, -0.085) {8};
\node[shape=rectangle, color=white, text=black] (title) at (-0.3,-1.2) {$G_{32}: \texttt{G\`{}G]vK}$};
\draw (1) -- (2); \draw (1) -- (7); \draw (1) -- (8); \draw (2) -- (7); \draw (2) -- (8); \draw (3) -- (4); \draw (3) -- (5); \draw (3) -- (8); \draw (4) -- (6); \draw (4) -- (7); \draw (5) -- (6); \draw (5) -- (7); \draw (6) -- (8); \draw (7) -- (8); 
\end{tikzpicture}}}

\def\GeightTHIRTYTWObar{\fbox{\begin{tikzpicture}[scale=3.7]
\node (1) at (-0.336, -0.11) {1};
\node (2) at (-0.551, 0.109) {2};
\node (3) at (-0.017, 0.318) {3};
\node (4) at (-0.586, -0.486) {4};
\node (5) at (-0.857, -0.342) {5};
\node (6) at (-0.276, 0.488) {6};
\node (7) at (0.077, 0.854) {7};
\node (8) at (-1.0, -0.831) {8};
\node[shape=rectangle, color=white, text=black] (title) at (-0.25,-0.95) {$\bar{G}_{32}: \texttt{G]v\`{}Go}$};
\draw (1) -- (3); \draw (1) -- (4); \draw (1) -- (5); \draw (1) -- (6); \draw (2) -- (3); \draw (2) -- (4); \draw (2) -- (5); \draw (2) -- (6); \draw (3) -- (6); \draw (3) -- (7); \draw (4) -- (5); \draw (4) -- (8); \draw (5) -- (8); \draw (6) -- (7); 
\end{tikzpicture}}}

\def\GeightTHIRTYTHREE{\fbox{\hspace{-1mm}\begin{tikzpicture}[scale=3.7]
\node (1) at (-0.206, 0.883) {1};
\node (2) at (-0.659, 0.506) {2};
\node (3) at (0.69, -0.55) {3};
\node (4) at (0.673, -0.867) {4};
\node (5) at (1.0, -0.517) {5};
\node (6) at (-0.252, 0.232) {6};
\node (7) at (0.032, 0.457) {7};
\node (8) at (0.353, -0.192) {8};
\node[shape=rectangle, color=white, text=black] (title) at (0.0,-0.9) {$G_{33}: \texttt{G\textsf{@}NE?\{}$};
\draw (1) -- (6); \draw (1) -- (7); \draw (2) -- (6); \draw (2) -- (7); \draw (3) -- (4); \draw (3) -- (5); \draw (4) -- (5); \draw (4) -- (8); \draw (5) -- (8); \draw (6) -- (8); \draw (7) -- (8); 
\end{tikzpicture}}}

\def\GeightTHIRTYTHREEbar{\hspace{-1mm}\fbox{\begin{tikzpicture}[scale=3.7]
\node (1) at (0.115, -0.308) {1};
\node (2) at (0.31, 0.052) {2};
\node (3) at (-0.413, -0.329) {3};
\node (4) at (-0.408, 0.132) {4};
\node (5) at (-0.161, 0.524) {5};
\node (6) at (-0.779, 0.565) {6};
\node (7) at (-1.0, 0.108) {7};
\node (8) at (0.429, -0.745) {8};
\node[shape=rectangle, color=white, text=black] (title) at (-0.25,-0.85) {$\bar{G}_{33}: \texttt{G\}ox$\sim$?}$};
\draw (1) -- (2); \draw (1) -- (3); \draw (1) -- (4); \draw (1) -- (5); \draw (1) -- (8); \draw (2) -- (3); \draw (2) -- (4); \draw (2) -- (5); \draw (2) -- (8); \draw (3) -- (6); \draw (3) -- (7); \draw (3) -- (8); \draw (4) -- (6); \draw (4) -- (7); \draw (5) -- (6); \draw (5) -- (7); \draw (6) -- (7); 
\end{tikzpicture}}}

\def\GeightTHIRTYFOUR{\fbox{\hspace{-1mm}\begin{tikzpicture}[scale=3.7]
\node (1) at (-0.206, 0.883) {1};
\node (2) at (-0.659, 0.506) {2};
\node (3) at (0.69, -0.55) {3};
\node (4) at (0.673, -0.867) {4};
\node (5) at (1.0, -0.517) {5};
\node (6) at (-0.252, 0.232) {6};
\node (7) at (0.032, 0.457) {7};
\node (8) at (0.353, -0.192) {8};
\node[shape=rectangle, color=white, text=black] (title) at (0.0,-0.9) {$G_{34}: \texttt{G\textsf{@}NE\textsf{@}\{}$};
\draw (1) -- (6); \draw (1) -- (7); \draw (2) -- (6); \draw (2) -- (7); \draw (3) -- (4); \draw (3) -- (5); \draw (3) -- (8); \draw (4) -- (5); \draw (4) -- (8); \draw (5) -- (8); \draw (6) -- (8); \draw (7) -- (8); 
\end{tikzpicture}}}

\def\GeightTHIRTYFOURbar{\fbox{\hspace{-1mm}\begin{tikzpicture}[scale=3.7]
\node (1) at (0.115, -0.308) {1};
\node (2) at (0.31, 0.052) {2};
\node (3) at (-0.413, -0.329) {3};
\node (4) at (-0.408, 0.132) {4};
\node (5) at (-0.161, 0.524) {5};
\node (6) at (-0.779, 0.565) {6};
\node (7) at (-1.0, 0.108) {7};
\node (8) at (0.429, -0.745) {8};
\node[shape=rectangle, color=white, text=black] (title) at (-0.25,-0.85) {$\bar{G}_{34}: \texttt{G\}ox\}?}$};
\draw (1) -- (2); \draw (1) -- (3); \draw (1) -- (4); \draw (1) -- (5); \draw (1) -- (8); \draw (2) -- (3); \draw (2) -- (4); \draw (2) -- (5); \draw (2) -- (8); \draw (3) -- (6); \draw (3) -- (7); \draw (4) -- (6); \draw (4) -- (7); \draw (5) -- (6); \draw (5) -- (7); \draw (6) -- (7);
\end{tikzpicture}}}

\def\GeightTHIRTYFIVE{\fbox{\hspace{-1mm}\begin{tikzpicture}[scale=3.7]
\node (1) at (-0.206, 0.883) {1};
\node (2) at (-0.659, 0.506) {2};
\node (3) at (0.69, -0.55) {3};
\node (4) at (0.673, -0.867) {4};
\node (5) at (1.0, -0.517) {5};
\node (6) at (0.353, -0.192) {6};
\node (7) at (0.032, 0.457) {7};
\node (8) at (-0.252, 0.232) {8};
\node[shape=rectangle, color=white, text=black] (title) at (0.0,-0.9) {$G_{35}: \texttt{G\textsf{@}K]MK}$};
\draw (1) -- (7); \draw (1) -- (8); \draw (2) -- (7); \draw (2) -- (8); \draw (3) -- (4); \draw (3) -- (5); \draw (4) -- (5); \draw (4) -- (6); \draw (5) -- (6); \draw (6) -- (7); \draw (6) -- (8); \draw (7) -- (8); 
\end{tikzpicture}}}

\def\GeightTHIRTYFIVEbar{\fbox{\hspace{-1mm}\begin{tikzpicture}[scale=3.7]
\node (1) at (0.115, -0.308) {1};
\node (2) at (0.31, 0.052) {2};
\node (3) at (-0.413, -0.329) {3};
\node (4) at (-0.408, 0.132) {4};
\node (5) at (-0.161, 0.524) {5};
\node (6) at (0.429, -0.745) {6};
\node (7) at (-1.0, 0.108) {7};
\node (8) at (-0.779, 0.565) {8};
\node[shape=rectangle, color=white, text=black] (title) at (-0.25,-0.85) {$\bar{G}_{35}: \texttt{G\}r\`{}po}$};
\draw (1) -- (2); \draw (1) -- (3); \draw (1) -- (4); \draw (1) -- (5); \draw (1) -- (6); \draw (2) -- (3); \draw (2) -- (4); \draw (2) -- (5); \draw (2) -- (6); \draw (3) -- (6); \draw (3) -- (7); \draw (3) -- (8); \draw (4) -- (7); \draw (4) -- (8); \draw (5) -- (7); \draw (5) -- (8); 
\end{tikzpicture}}}

\def\GeightTHIRTYSIX{\fbox{\hspace{-1mm}\begin{tikzpicture}[scale=3.7]
\node (1) at (-0.206, 0.883) {1};
\node (2) at (-0.659, 0.506) {2};
\node (3) at (0.69, -0.55) {3};
\node (4) at (0.673, -0.867) {4};
\node (5) at (1.0, -0.517) {5};
\node (6) at (-0.252, 0.232) {6};
\node (7) at (0.032, 0.457) {7};
\node (8) at (0.353, -0.192) {8};
\node[shape=rectangle, color=white, text=black] (title) at (0.0,-0.9) {$G_{36}: \texttt{G\textsf{@}NEH\{}$};
\draw (1) -- (6); \draw (1) -- (7); \draw (2) -- (6); \draw (2) -- (7); \draw (3) -- (4); \draw (3) -- (5); \draw (3) -- (8); \draw (4) -- (5); \draw (4) -- (8); \draw (5) -- (8); \draw (6) -- (7); \draw (6) -- (8); \draw (7) -- (8); 
\end{tikzpicture}}}

\def\GeightTHIRTYSIXbar{\fbox{\hspace{-1mm}\begin{tikzpicture}[scale=3.7]
\node (1) at (0.115, -0.308) {1};
\node (2) at (0.31, 0.052) {2};
\node (3) at (-0.413, -0.329) {3};
\node (4) at (-0.408, 0.132) {4};
\node (5) at (-0.161, 0.524) {5};
\node (6) at (-0.779, 0.565) {6};
\node (7) at (-1.0, 0.108) {7};
\node (8) at (0.429, -0.745) {8};
\node[shape=rectangle, color=white, text=black] (title) at (-0.25,-0.85) {$\bar{G}_{36}: \texttt{G\}oxu?}$};
\draw (1) -- (2); \draw (1) -- (3); \draw (1) -- (4); \draw (1) -- (5); \draw (1) -- (8); \draw (2) -- (3); \draw (2) -- (4); \draw (2) -- (5); \draw (2) -- (8); \draw (3) -- (6); \draw (3) -- (7); \draw (4) -- (6); \draw (4) -- (7); \draw (5) -- (6); \draw (5) -- (7); 
\end{tikzpicture}}}


\def\Hone{\fbox{\hspace{-1mm}\begin{tikzpicture}[scale=3]
\node (1) at (-1.0, -0.657) {1};
\node (2) at (0.1, -0.19) {2};
\node (3) at (0.168, 0.891) {3};
\node (4) at (-0.243, 0.914) {4};
\node (5) at (-0.0, -0.5) {5};
\node (6) at (-0.553, -0.197) {6};
\node (7) at (-0.427, -0.504) {7};
\node (8) at (-0.194, 0.235) {8};
\node[shape=rectangle, color=white, text=black] (title) at (-0.3,-0.8) {$H_1: \texttt{G\textsf{@}BMP\{}$};
\draw (1) -- (6); \draw (1) -- (7); \draw (2) -- (6); \draw (2) -- (7); \draw (3) -- (4); \draw (3) -- (8); \draw (4) -- (8); \draw (5) -- (6); \draw (5) -- (7); \draw (5) -- (8); \draw (6) -- (8); \draw (7) -- (8); 
\end{tikzpicture}}}

\def\Honebar{\fbox{\hspace{-1mm}\begin{tikzpicture}[scale=3]
\node (1) at (-0.616, -0.153) {1};
\node (2) at (-0.427, -0.446) {2};
\node (3) at (-0.235, 0.218) {3};
\node (4) at (0.016, 0.083) {4};
\node (5) at (-0.078, -0.338) {5};
\node (6) at (0.044, 0.755) {6};
\node (7) at (0.383, 0.531) {7};
\node (8) at (-1.0, -0.649) {8};
\node[shape=rectangle, color=white, text=black] (title) at (-0.25,-0.955) {$\bar{H}_1: \texttt{G\}\{pm?}$};
\draw (1) -- (2); \draw (1) -- (3); \draw (1) -- (4); \draw (1) -- (5); \draw (1) -- (8); \draw (2) -- (3); \draw (2) -- (4); \draw (2) -- (5); \draw (2) -- (8); \draw (3) -- (5); \draw (3) -- (6); \draw (3) -- (7); \draw (4) -- (5); \draw (4) -- (6); \draw (4) -- (7); \draw (6) -- (7);
\end{tikzpicture}}}

\def\Htwobar{\fbox{\hspace{-1mm}\begin{tikzpicture}[scale=2.3]
\node (1) at (0.37, -0.699) {1};
\node (2) at (-1.0, -0.642) {2};
\node (3) at (0.459, 0.686) {3};
\node (4) at (-0.915, 0.681) {4};
\node (5) at (0.202, 0.001) {5};
\node (6) at (-0.731, 0.022) {6};
\node (7) at (-0.302, -0.507) {7};
\node (8) at (-0.231, 0.458) {8};
\node[shape=rectangle, color=white, text=black] (title) at (-0.3,-1.35) {$\bar{H}_2: \texttt{G?h]\textsf{@}c}$};
\draw (1) -- (5); \draw (1) -- (7); \draw (2) -- (6); \draw (2) -- (7); \draw (3) -- (5); \draw (3) -- (8); \draw (4) -- (6); \draw (4) -- (8); \draw (5) -- (6); \draw (7) -- (8); 
\end{tikzpicture}}}

\def\Htwo{\fbox{\hspace{-2mm}\begin{tikzpicture}[scale=2.2]
\node (1) at (-0.4, -0.25) {1};
\node (2) at (0.19, -0.35) {2};
\node (3) at (-0.35, 0.35) {3};
\node (4) at (0.3, 0.25) {4};
\node (5) at (0.924, -0.058) {5};
\node (6) at (-1.0, 0.098) {6};
\node (7) at (0.043, 0.973) {7};
\node (8) at (-0.148, -0.92) {8};
\node[shape=rectangle, color=white, text=black] (title) at (-0.25,-1.25) {$H_2: \texttt{G$\sim$U\`{}\}W}$};
\draw (1) -- (2); \draw (1) -- (3); \draw (1) -- (4); \draw (1) -- (6); \draw (1) -- (8); \draw (2) -- (3); \draw (2) -- (4); \draw (2) -- (5); \draw (2) -- (8); \draw (3) -- (4); \draw (3) -- (6); \draw (3) -- (7); \draw (4) -- (5); \draw (4) -- (7); \draw (5) -- (7); \draw (5) -- (8); \draw (6) -- (7); \draw (6) -- (8); 
\end{tikzpicture}}}

\def\Hthree{\fbox{\hspace{-1mm}\begin{tikzpicture}[scale=2.4]
\node (1) at (0.275, -0.878) {1};
\node (2) at (-0.626, 0.88) {2};
\node (3) at (-0.441, -0.157) {3};
\node (4) at (0.228, 0.186) {4};
\node (5) at (-0.602, -0.714) {5};
\node (6) at (-1.0, 0.084) {6};
\node (7) at (0.678, -0.135) {7};
\node (8) at (0.28, 0.734) {8};
\node[shape=rectangle, color=white, text=black] (title) at (-0.25,-1.2) {$H_3: \texttt{G\textsf{@}hkac}$};
\draw (1) -- (5); \draw (1) -- (7); \draw (2) -- (6); \draw (2) -- (8); \draw (3) -- (4); \draw (3) -- (5); \draw (3) -- (6); \draw (4) -- (7); \draw (4) -- (8); \draw (5) -- (6); \draw (7) -- (8); 
\end{tikzpicture}}}

\def\Hthreebar{\fbox{\hspace{-2mm}\begin{tikzpicture}[scale=2.4]
\node (1) at (0.228, -0.797) {1};
\node (2) at (-1.0, -0.143) {2};
\node (3) at (-0.116, 0.117) {3};
\node (4) at (-0.563, -0.865) {4};
\node (5) at (-0.823, 0.624) {5};
\node (6) at (0.833, -0.316) {6};
\node (7) at (-0.093, 0.957) {7};
\node (8) at (0.61, 0.424) {8};
\node[shape=rectangle, color=white, text=black] (title) at (-0.1,-1.2) {$\bar{H}_3: \texttt{G\}UR\textbackslash W}$};
\draw (1) -- (2); \draw (1) -- (3); \draw (1) -- (4); \draw (1) -- (6); \draw (1) -- (8); \draw (2) -- (3); \draw (2) -- (4); \draw (2) -- (5); \draw (2) -- (7); \draw (3) -- (7); \draw (3) -- (8); \draw (4) -- (5); \draw (4) -- (6); \draw (5) -- (7); \draw (5) -- (8); \draw (6) -- (7); \draw (6) -- (8); 
\end{tikzpicture}}}

\def\Hfour{\fbox{\hspace{-1mm}\begin{tikzpicture}[scale=2.4]
\node (1) at (-1.0, 0.541) {1};
\node (2) at (-0.229, 0.879) {2};
\node (3) at (-0.353, -0.958) {3};
\node (4) at (0.55, -0.496) {4};
\node (5) at (-0.965, -0.48) {5};
\node (6) at (0.006, -0.005) {6};
\node (7) at (-0.773, 0.031) {7};
\node (8) at (0.219, 0.488) {8};
\node[shape=rectangle, color=white, text=black] (title) at (-0.2,-1.25) {$H_4: \texttt{G\textsf{@}h\^{}Ug}$};
\draw (1) -- (5); \draw (1) -- (7); \draw (1) -- (8); \draw (2) -- (6); \draw (2) -- (7); \draw (2) -- (8); \draw (3) -- (4); \draw (3) -- (5); \draw (3) -- (7); \draw (4) -- (6); \draw (4) -- (8); \draw (5) -- (6); \draw (5) -- (7); \draw (6) -- (8); 
\end{tikzpicture}}}

\def\Hfourbar{\fbox{\hspace{-1mm}\begin{tikzpicture}[scale=2.4]
\node (1) at (0.608, 0.03) {1};
\node (2) at (0.193, 0.678) {2};
\node (3) at (-0.027, -0.16) {3};
\node (4) at (-0.251, 0.254) {4};
\node (5) at (-0.698, 0.849) {5};
\node (6) at (0.347, -0.833) {6};
\node (7) at (-0.529, -0.772) {7};
\node (8) at (-1.0, -0.046) {8};
\node[shape=rectangle, color=white, text=black] (title) at (-0.25,-1.2) {$\bar{H}_4: \texttt{G\}U\_hS}$};
\draw (1) -- (2); \draw (1) -- (3); \draw (1) -- (4); \draw (1) -- (6); \draw (2) -- (3); \draw (2) -- (4); \draw (2) -- (5); \draw (3) -- (6); \draw (3) -- (8); \draw (4) -- (5); \draw (4) -- (7); \draw (5) -- (8); \draw (6) -- (7); \draw (7) -- (8); 
\end{tikzpicture}}}

\begin{figure}[h]
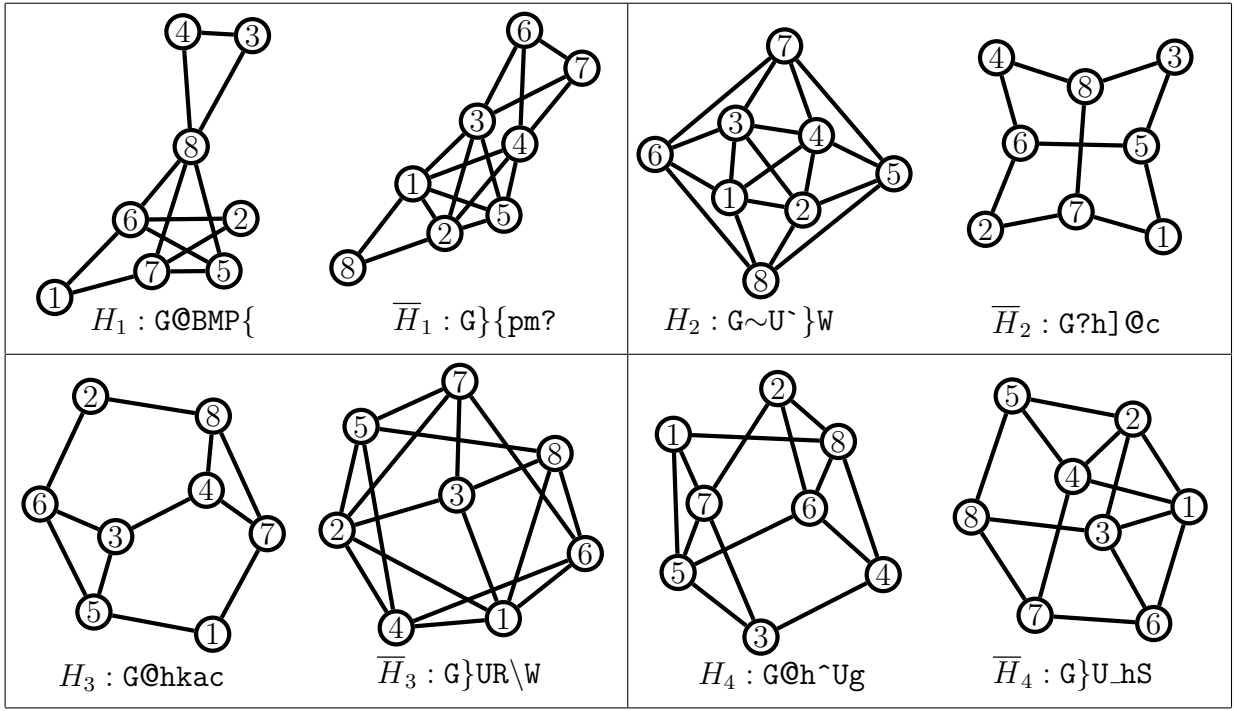

\begin{tabular}{|cc|cc|}\hline
\Hone & \Honebar 	&	\Htwo & \Htwobar 		\\\hline
\Hthree & \Hthreebar 	&	\Hfour & \Hfourbar 		\\\hline
\end{tabular} 

\caption{Graphs that are not complementary vanishing  due to the Gr\"obner basis argument. Each graph is labelled by its graph6 string, as used by the nauty package.}
\label{fig:grob}
\end{figure}

Let $G$ be a graph.  Define $A$ to be the variable matrix whose $i,j$-entry is a variable $a_{i,j} = a_{j,i}$ if $\{i,j\}$ is an edge or $i = j$, and otherwise we set the entry to be $0$.  Similarly, we define a variable matrix $B$ with variables $b_{i,j}$ for $\overline{G}$.  Thus, $AB = O$ is equivalent to a system of $n^2$ polynomial equations $f_{i,j} = 0$ for all $1\leq i \leq n$ and $1\leq j \leq n$, where 
\[f_{i,j} = \sum_{k\in N_G[i]\cap N_{\overline{G}}[j]}a_{i,k}b_{k,j}.\]
With this in mind, we let $I_{A,B} = \langle f_{i,j}\rangle$ be the ideal generated by all $f_{i,j}$ in the polynomial ring over $\mathbb{R}$ (for computational purposes, we sometimes consider the polynomial ring over $\mathbb{Q}$). 
Recall that the Gr\"obner basis of an ideal is a particularly nice generating set for an ideal, see \cite{CLO15} for background on this. 
The only fact about Gr\"obner bases that we need is that there exist computer algebra systems, such as SageMath, that can be used to  calculate the Gr\"obner basis $\mathbf{B}$ of $I_{A,B}$ so that $\langle\mathbf{B}\rangle = I_{A,B}$.

Suppose $\mathbf{B} = \{1\}$.  Then we know a combination of $f_{i,j}$ generates $1$, which is impossible when we assume $f_{i,j} = 0$ for all $i,j$.  Therefore, if $\mathbf{B} = \{1\}$, then $G$ is not complementary vanishing.  More formally, this argument implies the following.  

\begin{proposition}[Gr\"obner basis argument]
\label{prop:groebner}
Let $G$ be a graph on $n$ vertices.  Let $A\in\S(G)$ and $B\in\S(\overline{G})$ be the corresponding variable matrices.  If the Gr\"obner basis of $I_{A,B}$ contains $1$, then $G$ is not complementary vanishing.  
\end{proposition}

The calculation of a Gr\"obner basis can be expensive.  There are a few ways to reduce the number of variables:
\begin{enumerate}
\item Lemmas~\ref{one.diag.ent.zero}, \ref{prop.subneighborhood} and \ref{prop.one.private.neighbor} imply  that certain diagonal entries in $A$ and $B$ are either $0$ or nonzero.
\item By replacing $A$ by $k A$, we may assume one of the nonzero variables in $A$ is $1$.  (This variable can be an off-diagonal entry, or a diagonal entry that is guaranteed to be nonzero by one of our lemmas).  The same argument applies simultaneously for the matrix $B$.
\item Pick a spanning forest $T$ of $G$.  By replacing $A$ by $\pm DAD$ and $B$ by $D^{-1}BD^{-1}$ for an appropriate invertible diagonal matrix $D$, we may assume the entries in $A$ corresponding to the edges of $T$ along with some (arbitrary) nonzero diagonal entry are $1$.  We may switch the role of $A$ and $B$, but we cannot apply this to $A$ and $B$ simultaneously. 
\end{enumerate}

Note that this last technique relies on the choice of the spanning tree, and that the Gr\"obner basis depends on the choice of the monomial order, which often depends on the order of the variables.  In most of the cases we have tried, we used the degree reverse lexicographic order of the monomials (which is the default setting for SageMath) with the off-diagonal variables preferred over the diagonal variables.  We use this method to show the graphs in Figure~\ref{fig:grob}
are not complementary vanishing.

\begin{example}
Let $H_2$ and $\overline{H}_2$ be the graphs in Figure~\ref{fig:grob}.  Assume $A = \begin{bmatrix} a_{i,j} \end{bmatrix} \in\S(H_2)$ and $B = \begin{bmatrix} b_{i,j} \end{bmatrix}\in\S(\overline{H}_2)$ satisfy $AB = O$.  According to Lemmas~\ref{prop.subneighborhood} and \ref{prop.one.private.neighbor}, we may assume 
\[\begin{aligned}
a_{i,i} = 0 &\text{ for }i \in  \{1,2,3,4\}, \\
a_{i,i} \neq 0 &\text{ for }i \in \{5,6,7,8\}, \\
b_{i,i} \neq 0 &\text{ for }i \in  \{1,2,3,4\}, \\
b_{i,i} = 0 &\text{ for }i \in \{5,6,7,8\}. \\
\end{aligned}\]
Moreover, we may pick a spanning tree $T$ of $H_2$ using the edges 
\[E(T) = \{\{1,8\}, \{2,8\}, \{3,7\}, \{4,7\}, \{5,8\}, \{6,7\}, \{6,8\}\}.\]

By replacing $A$ with $\pm DAD$ and $B$ with $D^{-1}BD^{-1}$ for some appropriate diagonal matrix $D$, we may assume the entries in $A$ corresponding to $E(T)$ along with $a_{5,5}$ are $1$.  Moreover, by replacing $B$ with $kB$ for some nonzero $k$, we may further assume $b_{1,5} = 1$.  In conclusion, we have 
\[A = \begin{bmatrix}
0 & a_{1,2} & a_{1,3} & a_{1,4} & 0 & a_{1,6} & 0 & 1 \\
a_{1,2} & 0 & a_{2,3} & a_{2,4} & a_{2,5} & 0 & 0 & 1 \\
a_{1,3} & a_{2,3} & 0 & a_{3,4} & 0 & a_{3,6} & 1 & 0 \\
a_{1,4} & a_{2,4} & a_{3,4} & 0 & a_{4,5} & 0 & 1 & 0 \\
0 & a_{2,5} & 0 & a_{4,5} & 1 & 0 & a_{5,7} & 1 \\
a_{1,6} & 0 & a_{3,6} & 0 & 0 & a_{6,6} & 1 & 1 \\
0 & 0 & 1 & 1 & a_{5,7} & 1 & a_{7,7} & 0 \\
1 & 1 & 0 & 0 & 1 & 1 & 0 & a_{8,8}
\end{bmatrix}\]
and 
\[B = \begin{bmatrix}
b_{1,1} & 0 & 0 & 0 & 1 & 0 & b_{1,7} & 0 \\
0 & b_{2,3} & 0 & 0 & 0 & b_{2,6} & b_{2,7} & 0 \\
0 & 0 & b_{3,3} & 0 & b_{3,5} & 0 & 0 & b_{3,8} \\
0 & 0 & 0 & b_{4,4} & 0 & b_{4,6} & 0 & b_{4,8} \\
1 & 0 & b_{3,5} & 0 & 0 & b_{5,6} & 0 & 0 \\
0 & b_{2,6} & 0 & b_{4,6} & b_{5,6} & 0 & 0 & 0 \\
b_{1,7} & b_{2,7} & 0 & 0 & 0 & 0 & 0 & b_{7,8} \\
0 & 0 & b_{3,8} & b_{4,8} & 0 & 0 & b_{7,8} & 0
\end{bmatrix}.\]
By treating the remaining $a_{i,j}$ and $b_{i,j}$ as variables, the $8^2$ polynomial entries of $AB$ generate the ideal $I_{A,B}$.  

We then find the Gr\"obner basis of $I_{A,B}$ under the degree reverse lexicographic order, which requires an ordering of the variables.  
Order the variables by
\[a_{i,j}\ (i\neq j) \succ b_{i,j}\ (i\neq j) \succ a_{i,i} \succ b_{i,i}\]
(and order the off-diagonal entries by lexicographic order).  This means that the algorithm prioritizes elimination of $a_{i,j}$ first, and then $b_{i,j}$, and so on.  With these settings, the Gr\"obner basis of $I_{A,B}$ is $\{1\}$.
This means some polynomials in $I_{A,B}$ generates $1$, which is impossible if they are all zero.  Therefore, $H_2$ is not complementary vanishing.
\end{example}


\begin{remark}
The same technique applies to  $H_3$ and $H_4$ in Figure~\ref{fig:grob}, so they are not complementary vanishing.  For $H_1$, one may choose a spanning tree in $\overline{H}_1$ instead. In this case the Gr\"obner basis will contain a polynomial of a single nonzero variable, e.g., $a_{5,8}$.  Therefore, it again shows $H_1$ is not complementary vanishing.
\end{remark}

\subsection{Solving for \texorpdfstring{$B$}{B}}
\label{subsec:random}

Let $G$ be a graph and $A\in\S(G)$.  We present an algorithm for checking if there is a matrix $B\in\S(\overline{G})$ such that $AB = O$.  

Let $G$ be a graph with $n$ vertices and $m$ edges.  We define $\Scl(G)$ as the topological closure of $\S(G)$; that is, $\Scl(G)$ consists of all real symmetric matrices whose $i,j$-entry can be nonzero only when $\{i,j\}$ is an edge or $i = j$.  Thus, $\Scl(G)$ is a vector subspace of dimension $n+m$ in the space of real symmetric matrices of order $n$.

When $A\in\S(G)$ is given, define
\[\mathcal{W} = \{B \in \Scl(\overline{G}): AB = O\},\]
which is a vector space.  Our task is to determine whether the vector space $\mathcal{W}$ contains a matrix $B\in\S(\overline{G})$.  To do so, we need to check the existence of a matrix $B\in \mathcal{W}$ that is nonzero on the edges of $\overline{G}$.  

Let $\mathbf{B}$ be a basis of $\mathcal{W}$.  Then every matrix in $\mathcal{W}$ is zero at a given entry if and only if every matrix in $\mathbf{B}$ is zero at the given entry.  Moreover, when $S$ is a subset of indices, $\mathcal{W}$ has a matrix that is nowhere-zero on $S$ if and only if for each index $\{i,j\}\in S$, there is a vector in $\mathbf{B}$ that is not zero at index $\{i,j\}$.  With these observations, whether there is a matrix $B\in\S(\overline{G})$ with $AB = O$ can be determined through standard linear algebra techniques.  This idea is summarized in Proposition~\ref{prop:solveB}.

\begin{proposition}
\label{prop:solveB}
Let $G$ be a graph, $A \in \S(G)$, and let $\mathcal{W}$ be as defined above with $\mathbf{B}$ a basis of $\mathcal{W}$.  Then a matrix $B\in\S(\overline{G})$ exists with $AB = O$ if and only if for every edge $\{i,j\}\in E(\overline{G})$, there is a matrix $C\in\mathbf{B}$ whose $i,j$-entry is nonzero.  Moreover, if the latter condition holds, then a random linear combination of $\mathbf{B}$ with each coefficient in Gaussian distribution gives a matrix $B$ in $\S(\overline{G})$ with probability $1$.
\end{proposition}

\subsection{Proof of Theorem~\ref{thm:conditions}}
We summarize some necessary and sufficient conditions for a graph to be complementary vanishing: 

\begin{itemize}
\item Methods which can verify a graph is not complementary vanishing:
\begin{itemize}
\item \textbf{Diagonal lemmas}: Lemma~\ref{one.diag.ent.zero} and Corollary~\ref{cor:noInducedPath}
\item \textbf{Odd cycle lemma}: Lemma~\ref{lem:oddcycle} in conjunction with Lemmas~\ref{prop.subneighborhood} and \ref{prop.one.private.neighbor}
\item \textbf{Gr\"obner basis arguments}:  Proposition~\ref{prop:groebner} and its modifications
\end{itemize}
\item Methods to verify a graph is complementary vanishing:
\begin{itemize}
\item \textbf{Twin proposition}:  Proposition~\ref{prop:twin}
\item \textbf{Duplication lemma}:  Lemma~\ref{lem:duplication}
\item \textbf{Random trial}:  Randomly pick matrices $A\in\S(G)$ and apply Proposition~\ref{prop:solveB} to see if $B\in\S(\overline{G})$ with $AB = O$ exists; or do the same for $\overline{G}$
\item \textbf{Manual construction}:  Judiciously construct matrices $A\in\S(G)$ and $B\in\S(\overline{G})$ with $AB = O$
\end{itemize}
\end{itemize}

\begin{table}[h]
\begin{center}
\begin{tabular}{c|ccc|ccc|c}
 ~ & \multicolumn{3}{c|}{Not complementary vanishing} & \multicolumn{3}{c|}{Complementary vanishing} & ~ \\
 $n$ & \textsf{diag} & \textsf{oc} & \textsf{grob} & \textsf{twin} & \textsf{dup} & \textsf{certificate found} & \textsf{total} \\
 \hline
 1 & ~ & ~ & ~ & ~ & ~ & 1 & 1 \\
 2 & ~ & ~ & ~ & ~ & ~ & ~ & 0 \\
 3 & ~ & ~ & ~ & ~ & ~ & ~ & 0 \\
 4 & 1 & ~ & ~ & ~ & ~ & ~ & 1 \\
 5 & 5 & ~ & ~ & ~ & ~ & ~ & 5 \\
 6 & 34 & ~ & ~ & ~ & ~ & ~ & 34 \\
 7 & 326 & 1 & ~ & ~ & ~ & 4 & 331 \\
 8 & 4905 & 8 & 4 & 6 & 14 & 12 & 4949 \\
\end{tabular}
\end{center}
\caption{The number of pairs $\{G,\overline{G}\}$ verified using each method, where $G$ and $\overline{G}$ are connected graphs on $n$ vertices.  The columns \textsf{diag}, \textsf{oc}, \textsf{grob} counts the graph pairs that are not complementary vanishing  due to the diagonal lemmas, the odd-cycle lemma, and Gr\"obner basis arguments.  The columns \textsf{twin}, \textsf{dup}, and \textsf{certificate found} count the graph pairs that are complementary vanishing  because of the twin proposition, the duplication lemma, or because the certificate is found by random trial or manual construction.  The methods were tested in order from left to right.}
\label{tab:summary}
\end{table}

It turns out that these techniques are enough for us to classify all  pairs $\{G,\overline{G}\}$ with $G$ and $\overline{G}$ connected on $8$ or few vertices.  The results are shown in Table~\ref{tab:summary}.  
Therefore, we have the following lemma, which one can check by computer is equivalent to Theorem~\ref{thm:conditions}.

\begin{lemma}
Let $G$ be a graph on $n\leq 8$ vertices such that both $G$ and $\overline{G}$ are connected.  Then $G$ is complementary vanishing  if and only if $G$ and $\overline{G}$ do not satisfy the hypotheses of Lemmas \ref{one.diag.ent.zero}, \ref{lem:oddcycle}, Corollary~\ref{cor:noInducedPath}, and are not one of the four pairs in Figure~\ref{fig:grob}.
\end{lemma}
We note that for one of the pairs in Appendix~A, namely $G_8$, we were unable to find a certificate using random trials; the only construction we know of for this graph was found manually.

We note that in principle, one could try to apply these techniques to determine which pairs of graphs on $9$ vertices are complementary vanishing.  However, applying the diagonal conditions to graphs on $9$ vertices leaves to many possibilities for random trails to effectively resolve.

\section{Conclusion} \label{Sec5}
In this paper we began the study of complementary vanishing graphs.  In view of Theorem \ref{thm:biconnected}, we have essentially reduced the problem to graphs $G$ such that both $G$ and $\overline{G}$ are connected.  In Theorem \ref{thm:conditions} we completely solved this problem for graphs on at most 8 vertices.  

As things currently stand, there does not seem to be a simple description which categorizes graphs that are complementary vanishing. However, there are certain classes of graphs where one might be able to determine the answer.  For example, it is natural to ask what happens for random graphs.

\begin{question}
If $G$ is chosen uniformly at random amongst all $n$-vertex graphs, does the probability that $G$ is complementary vanishing tend to 1 as $n$ tends towards infinity?
\end{question}
If the answer to this question is ``yes'', then it is likely that this will be difficult to prove.  Indeed, this would imply that the random graph satisfies the Graph Complement Conjecture asymptotically almost surely, which is an open problem.

Recall that there are four connected graphs $G$ on at most $8$ vertices such that their complements are also connected and such that these $G$ were shown not to be complementary vanishing  using a Gr\"{o}bner Basis argument.  
It would be nice if one could prove that these four graphs are not complementary vanishing  by utilizing combinatorial techniques.

\begin{question}
Can one use a combinatorial argument to prove that the graphs in Figure~\ref{fig:grob} are not complementary vanishing?
\end{question}

\textbf{Acknowledgments.} The authors would like to express their sincere gratitude to the organizers of the Mathematics Research Community on Finding Needles in Haystacks: Approaches to Inverse Problems Using Combinatorics and Linear Algebra, and the AMS for funding the program through NSF grant 1916439. 
J. C.-H. Lin was supported by the Young Scholar Fellowship Program (grant
no.\ MOST-110-2628-M-110-003) from the Ministry of Science and Technology of Taiwan.
S. Spiro was supported by the National Science Foundation Graduate Research Fellowship under Grant No.\ DGE-1650112. 

\bibliographystyle{abbrv}
\bibliography{refs}

\begin{thebibliography}{10}

\bibitem{AAC13}
B.~Ahmadi, F.~Alinaghipour, M.~S. Cavers, S.~Fallat, K.~Meagher, and
  S.~Nasserasr.
\newblock Minimum number of distinct eigenvalues of graphs.
\newblock {\em Electron. J. Linear Algebra}, 26:673--691, 2013.

\bibitem{AIM08}
{AIM Minimum Rank-Special Graphs Work Group}.
\newblock Zero forcing sets and the minimum rank of graphs.
\newblock {\em Linear Algebra Appl.}, 428(7):1628--1648, 2008.

\bibitem{AIM06}
{American Institute of Mathematics Workshop}.
\newblock Spectra of families of matrices described by graphs, digraphs, and
  sign patterns.
\newblock \url{http://aimath.org/pastworkshops/matrixspectrum.html}, 2006.

\bibitem{BBF12}
F.~Barioli, W.~Barrett, S.~M. Fallat, H.~T. Hall, L.~Hogben, and H.~van~der
  Holst.
\newblock On the graph complement conjecture for minimum rank.
\newblock {\em Linear Algebra Appl.}, 436(12):4373--4391, 2012.

\bibitem{BF04}
F.~Barioli and S.~M. Fallat.
\newblock On two conjectures regarding an inverse eigenvalue problem for
  acyclic symmetric matrices.
\newblock {\em Electron. J. Linear Algebra}, 11:41--50, 2004.

\bibitem{BBF20}
W.~Barrett, S.~Butler, S.~M. Fallat, H.~T. Hall, L.~Hogben, J.~C.-H. Lin, B.~L.
  Shader, and M.~Young.
\newblock The inverse eigenvalue problem of a graph: multiplicities and minors.
\newblock {\em J. Combin. Theory Ser. B}, 142:276--306, 2020.

\bibitem{BFH17}
W.~Barrett, S.~Fallat, H.~T. Hall, L.~Hogben, J.~C.-H. Lin, and B.~L. Shader.
\newblock Generalizations of the strong {A}rnold property and the minimum
  number of distinct eigenvalues of a graph.
\newblock {\em Electron. J. Combin.}, 24(2):Paper No. 2.40, 28, 2017.

\bibitem{Y91}
Y.~Colin~de Verdi\`ere.
\newblock On a new graph invariant and a criterion for planarity.
\newblock In {\em Graph structure theory ({S}eattle, {WA}, 1991)}, volume 147
  of {\em Contemp. Math.}, pages 137--147. Amer. Math. Soc., Providence, RI,
  1993.

\bibitem{CLO15}
D.~A. Cox, J.~Little, and D.~O'Shea.
\newblock {\em Ideals, varieties, and algorithms: An introduction to
  computational algebraic geometry and commutative algebra}.
\newblock Springer, fourth edition, 2015.

\bibitem{EKSTRAND20131862}
J.~Ekstrand, C.~Erickson, H.~T. Hall, D.~Hay, L.~Hogben, R.~Johnson,
  N.~Kingsley, S.~Osborne, T.~Peters, J.~Roat, A.~Ross, D.~D. Row, N.~Warnberg,
  and M.~Young.
\newblock Positive semidefinite zero forcing.
\newblock {\em Linear Algebra Appl.}, 439(7):1862--1874, 2013.

\bibitem{FH07}
S.~M. Fallat and L.~Hogben.
\newblock The minimum rank of symmetric matrices described by a graph: a
  survey.
\newblock {\em Linear Algebra Appl.}, 426(2-3):558--582, 2007.

\bibitem{GK60}
F.~R. Gantmacher and M.~G. Krein.
\newblock {\em Oszillationslnatrizen, Oszillationskerne und kleine Schwingungen
  nzechanischer Systenze}.
\newblock Translated by Alfred Stohr, Berlin, Akademie-Verlag, 1960.

\bibitem{JORET2012488}
G.~Joret and D.~R. Wood.
\newblock Nordhaus-{G}addum for treewidth.
\newblock {\em European J. Combin.}, 33(4):488--490, 2012.

\bibitem{KLV97}
A.~Kotlov, L.~Lov\'{a}sz, and S.~Vempala.
\newblock The {C}olin de {V}erdi\`ere number and sphere representations of a
  graph.
\newblock {\em Combinatorica}, 17(4):483--521, 1997.

\bibitem{LOS19}
R.~H. Levene, P.~Oblak, and H.~\v{S}migoc.
\newblock A {N}ordhaus-{G}addum conjecture for the minimum number of distinct
  eigenvalues of a graph.
\newblock {\em Linear Algebra Appl.}, 564:236--263, 2019.

\bibitem{LNP14}
X.~Li, M.~Nathanson, and R.~Phillips.
\newblock Minimum vector rank and complement critical graphs.
\newblock {\em Electron. J. Linear Algebra}, 27:100--123, 2014.

\end{thebibliography}


\includepdf[pages={1-}, landscape=true]{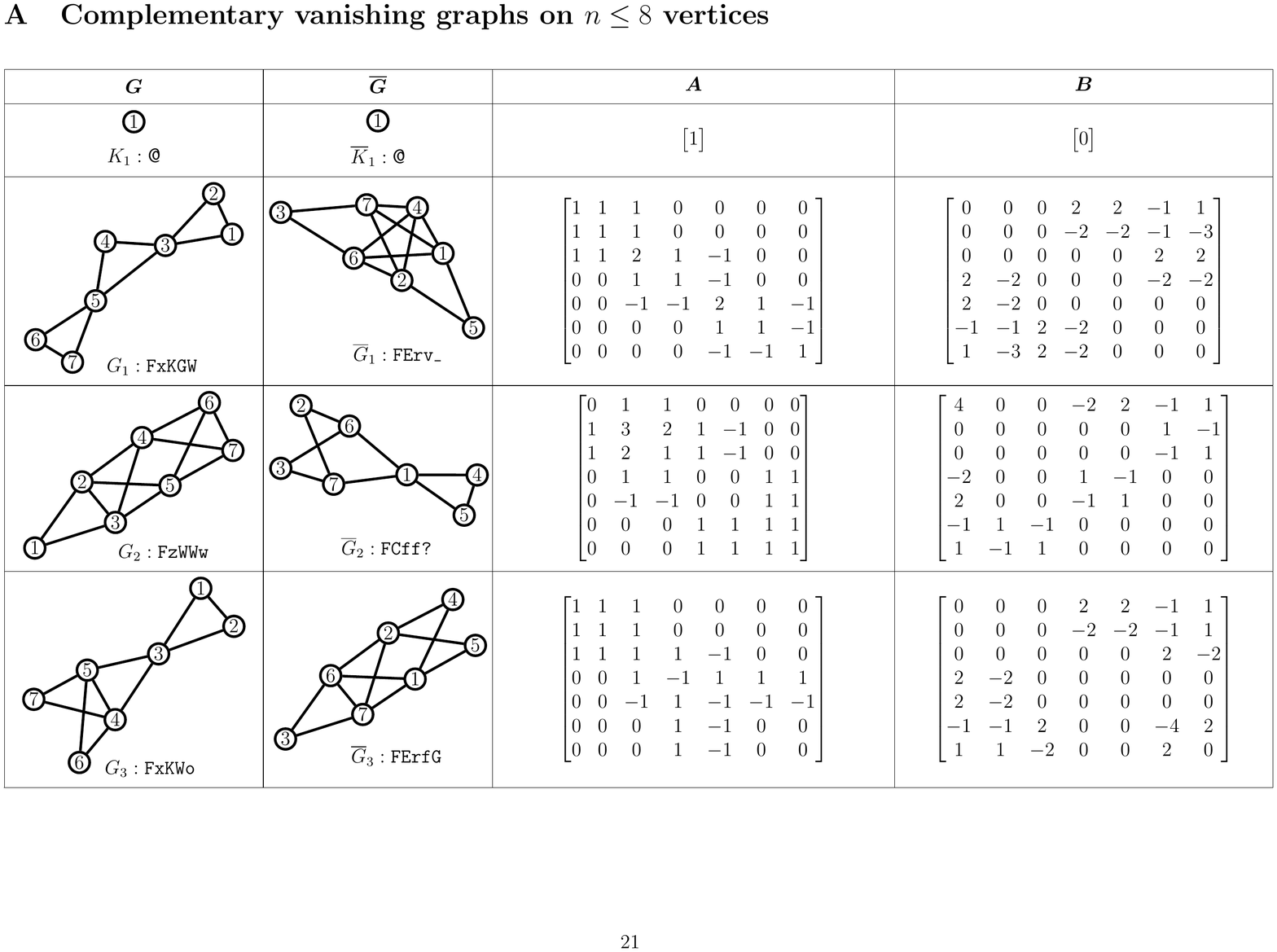}

\end{document}